\documentclass[11pt]{amsart}
\usepackage{geometry}
\usepackage{graphicx}
\usepackage{amssymb}
\usepackage{epstopdf}
\usepackage{amsmath,amscd}
\usepackage{amsthm}
\usepackage{enumerate}
\usepackage{url,verbatim}
\usepackage{subfig}
\usepackage{xcolor}

\usepackage[normalem]{ulem}
\usepackage{fixme}

\RequirePackage[colorlinks,citecolor=blue,urlcolor=blue]{hyperref}

\theoremstyle{plain}
\newtheorem{theorem}{Theorem}
\newtheorem{lemma}[theorem]{Lemma}
\newtheorem{proposition}[theorem]{Proposition}

\theoremstyle{definition}
\newtheorem{definition}[theorem]{Definition}
\newtheorem{example}[theorem]{Example}
\newtheorem{assumption}[theorem]{Assumption}

\theoremstyle{remark}
\newtheorem{remark}[theorem]{Remark}

\newtheorem{construction}[theorem]{Construction}

\newcommand\ol{\overline}

\newcommand\EE{{\mathbb E}}

\newcommand\RR{{\mathbb R}}
\newcommand\ZZ{{\mathbb Z}}
\newcommand\NN{{\mathbb N}}

\newcommand\HH{{\mathbb H}}

\newcommand\si{\sigma}

\newcommand\q{\quad}

\newcommand\Si{\Sigma}

\renewcommand\ell{l}

\newcommand\GT{\mathbb{G}\mathbb{T}}
\newcommand\CC{\mathbb{C}}
\newcommand\bm{\mathbf{m}}
\newcommand\SH{\mathrm{SH}}

\newcounter{mycount}

\numberwithin{equation}{section}
\numberwithin{theorem}{section}
\numberwithin{figure}{section}

\title[Fluctuations of dimer heights on contracting square-hexagon lattices]{Fluctuations of dimer heights on contracting square-hexagon lattices}

\date{}

\author{Zhongyang Li}

\begin{document}
\maketitle

\begin{abstract}We study perfect matchings on the square-hexagon lattice with $1\times n$ periodic edge weights such that the boundary condition is given by either (1) each remaining vertex on the bottom boundary is followed by $(m-1)$ removed vertices; (2) the bottom boundary can be divided into finitely many alternating line segments where all the vertices along each line segment are either removed or remained. In Case (1), we show that under certain homeomorphism from the liquid region to the upper half plane, the height fluctuations converge to the Gaussian free field in the upper half plane. In Case (2), when the edge weights $x_1,\ldots,x_n$ in one period satisfy the condition that $x_{i+1}=O\left(\frac{x_i}{N^{\alpha}}\right)$, where $\alpha>0$ is a constant independent of $N$, we show that the height fluctuations converge to a sum of independent Gaussian free fields.\end{abstract}

\section{Introduction}

A perfect matching, or a dimer configuration on a graph is a subset of edges such that each vertex is incident to exactly one edge in the subset. Dimer configurations appear naturally in statistical physics to model the structure of matter, for example, the perfect matchings on the hexagon lattice is a mathematical model for the molecule structure of graphite. With explicit combinatorial correspondence, the dimer model is also closely related to other lattice models in statistical mechanics, including the Ising model (\cite{ZL12,ZLsp}), the 1-2 model (\cite{ZL14,ZL141,GL15,GL17}) and a general polygon model (\cite{GLPO}). By developing the technique of Kasteleyn, Temperley and Fisher (\cite{Ka61,TF61}), the partition function (weighted sum of configurations) of dimer configurations on a finite plane graph can be expressed explicitly as the determinant or pfaffian of a weighted adjacency matrix; the local statistics can be computed (\cite{Ken01}). By study the spectral curve of the periodic dimer model using algebraic geometry technique, the sharp phase transition result can be established (\cite{KOS,Ken06}). The asymptotics of the rescaled dimer height function on a graph approximating a simply-connected domain can also studied by a variational principle (\cite{CKP,KO}); and also by the asymptotics of certain symmetric functions (see \cite{OR01,AB07,AB11,LP14,LP142,GP15,bg,bk,BL17,Li18}).

The Gaussian free field (GFF) is a high-dimensional time analogue of Brownian motion. The main aim of this paper is to investigate the connection between the height fluctuations of the dimer model on a contracting square-hexagon lattice and the Gaussian free field. It was first shown in \cite{RK00,RK01} that the (non-rescaled) height function for the dimer model with uniform underlying measure on a simply-connected square grid with Temperley boundary condition converge to a GFF in distribution. The result was later proved for the whole-plane isoradial graph (\cite{BdT07}) and the simply-connected isoradial graph with Temperley boundary condition (\cite{Li13}). For boundary conditions other than the Temperley boundary condition, the convergence of height fluctuation for the dimer model with uniform underlying measure on a contracting hexagon lattice to GFF was proved in \cite{LP142}; the corresponding result on a Aztec diamond (contracting square grid) with uniform underlying measure was proved in \cite{bk}, by analyzing the asymptotics of the Schur function in a neighborhood of $(1,1,\ldots,1)$ (\cite{bg16}).

A related model is the dimer model on a contracting square-hexagon lattice whose underlying measure depends on periodically assigned edge weights with period $1\times n$. In \cite{BL17,Li18}, we studied this model by establishing an identity of the partition function of the dimer model on such a graph and the value of the Schur function depending on edge weights; and then analyze the asymptotics of the Schur function in a neighborhood of a generic point $(x_1,\ldots,x_N)$. The law of large numbers for the rescaled height function was proved for two specific boundary conditions on the bottom boundary (1) each remaining vertex on the boundary is followed by $(m-1)$ removed vertices, where $m\geq 2$ is a positive integer; (2) the bottom boundary is divided to alternate line segments with either all vertices removed or all the vertices preserved in each segment. We shall call the first boundary condition the \textbf{uniform boundary condition} and the second boundary condition the \textbf{piecewise boundary condition}. In this paper, we study the non-rescaled height fluctuations for the dimer model on the contracting square-hexagon lattice with the above two boundary conditions, and show that they converge to GFF in certain sense, building on the analysis of the Schur function at a generic point (see \cite{BL17,Li18}) and the techniques to relate fluctuations of particle systems determined by Schur generating functions and GFF  developed in \cite{bg16}.

The organization of the paper is as follows. In Sect.~\ref{bkg}, we introduce the contracting square-hexagon lattice and the main technical tools used in this paper. In Sect.~\ref{unr}, we introduce the uniform boundary conditions and review the limit shape result for the dimer model on a contracting square hexagon lattice with $1\times n$ periodic edge weights and  the uniform boundary conditions. In Sect.~\ref{unc}, we prove that certain statistics constructed from the  dimer model on a contracting square hexagon lattice with $1\times n$ periodic edge weights and  the uniform boundary conditions convergence to Gaussian distribution in the scaling limit. In Sect.~\ref{pbr}, we introduce the piecewise boundary conditions and review the limit shape result for the dimer model on a contracting square hexagon lattice with $1\times n$ periodic edge weights and  the piecewise boundary conditions. In Sect.~\ref{pbc}, we prove that certain statistics contracted from the  dimer model on a contracting square hexagon lattice with $1\times n$ periodic edge weights and  the piece boundary conditions convergence to a sum of finitely many independent Gaussian random variables  in the scaling limit; where the number of independent Gaussian random variables depends on the size of the period $n$. In Sect.~\ref{gffu}, we show that the statistics constructed from the  dimer model on a contracting square hexagon lattice with $1\times n$ periodic edge weights and  the uniform boundary conditions convergence to GFF in the upper half plane, under an homeomorphism from the liquid region to the upper half plane.
In Sect.~\ref{gffp}, we show that the statistics constructed from the  dimer model on a contracting square hexagon lattice with $1\times n$ periodic edge weights and  the piecewise boundary conditions convergence to a sum of $n$ independent GFFs in the upper half plane.

\section{Background}\label{bkg}

In this section, we define a general class graph (the contracting square-hexagon lattice) on which the height fluctuations of the dimer model is studied in this paper. Dimer model on such graphs has been studied in (\cite{BF15,bbccr,BL17,Li18}), and the limit shape result was explicitly established. We also review the main technical tools used in this paper, including the Schur function, the Young diagram, etc.

\subsection{Square-hexagon Lattices}Consider a doubly-infinite binary sequence indexed by integers $\ZZ=\{\ldots,-2,-1,0,1,2,\ldots\}$.
\begin{eqnarray}
\check{c}=(\ldots,c_{-2},c_{-1},c_0,c_1,c_2,\ldots)\in\{0,1\}^{\ZZ}.\label{ca}
\end{eqnarray}

The \textbf{whole-plane square-hexagon lattice} associated with the  sequence $\check{c}$, is a bipartite plane graph $\mathrm{SH}(\check{c})$ defined as follows. Its vertex set is a subset of $\frac{\ZZ}{2}\times \frac{\ZZ}{2} $. Each vertex of $\mathrm{SH}(\check{c})$ is either black or white, and we identify the vertices with points on the plane. 
 For $m\in \ZZ $, the black vertices have $y$-coordinate $m$; while the white vertices have $y$-coordinate $m-\frac{1}{2}$. We will label all the vertices with coordinate $m$ as vertices in the $(2m)$th row, and all the vertices with coordinate $m-\frac{1}{2}$ as vertices in the $(2m-1)$th row. We further require that
 \begin{itemize}
\item  each black vertex in the $(2m)$th row is adjacent to two white vertices in the $(2m+1)$th row; and
\item if $c_m=1$, each white vertex on the $(2m-1)$th row is adjacent to exactly one black vertex in the $(2m)$th row; 
 if $c_m=0$, each white vertex on the $(2m-1)$th row is adjacent to two black vertices in the $(2m)$th row. 
\end{itemize}
 See Figure \ref{lcc}.
 
 \begin{figure}
\subfloat[Structure of $\mathrm{SH}(\check{c})$ between the $(2m)$th row and the $(2m+1)$th row]{\includegraphics[width=.6\textwidth]{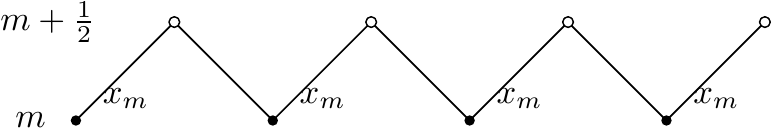}}\\
\subfloat[Structure of $\mathrm{SH}(\check{c})$ between the $(2m-1)$th row and the $(2m)$th row when $c_m=0$]{\includegraphics[width = .6\textwidth]{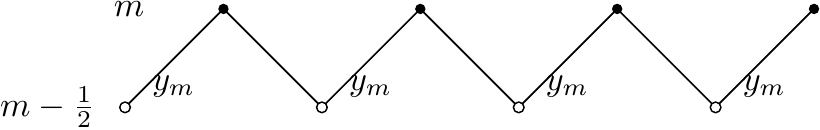}}\\
\subfloat[Structure of $\mathrm{SH}(\check{c})$ between the $(2m-1)$th row and the $(2m)$th row when $c_m=1$]{\includegraphics[width = .55\textwidth]{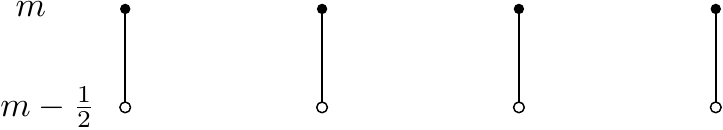}}
\caption{Graph structures of the square-hexagon lattice on the $(2m-1)$th, $(2m)$th, and $(2m+1)$th rows depend on the values of $(c_m)$. Black vertices are along the $(2m)$th row, while white vertices are along the $(2m-1)$th and $(2m+1)$th row.}
\label{lcc}
\end{figure}

Note that for any $\check{c}\in \{0,1\}^{\ZZ}$, the faces of $\SH(\check{c})$ is either a square or a hexagon. if $c_i=0$ for all $i\in\ZZ$, $\SH(\check{c})$ is a square grid; while if $c_i=1$ for all $\SH(\check{c})$ is a  hexagonal lattice. See Figure \ref{fig:SH} for an example of a square-hexagon lattice.

\begin{figure}
\includegraphics{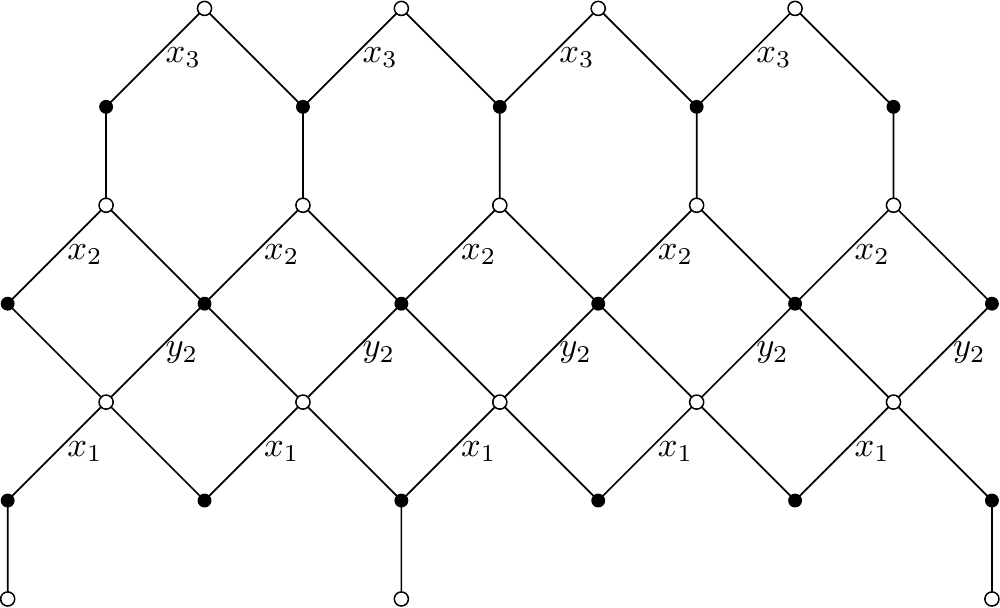}
\caption{Contracting square-hexagon lattice with $N=3$, $m=3$, $\Omega=(1,3,6), (c_1,c_2,c_3)=(1,0,1)$.}
\label{fig:SH}
\end{figure}

We shall assign edge weights to the whole-plane square-hexagon lattice $\SH(\check{c})$ satisfying the following assumption; see Figure \ref{lcc}.

\begin{assumption}\label{apew} For $m\geq 1$, we assign weight $x_m>0$ to each NE-SW edge joining the $(2m)$th row  to the $(2m+1)$th row of $\mathrm{SH}(\check{c})$. We assign weight $y_m>0$ to each NE-SW edge joining the $(2m-1)$th row to the $(2m)$th row of $\mathrm{SH}(\check{c})$, if such an edge exists. We assign weight $1$ to all the other edges. 
\end{assumption}

A \textbf{contracting square-hexagon lattice} is built from a whole-plane square-hexagon lattice as follows:

\begin{definition}\label{dfr}Let $N\in \NN$. Let $\Omega=(\Omega_1,\ldots,\Omega_N)$ be an $N$-tuple of positive integers, such that  $1=\Omega_1<\Omega_2<\cdots<\Omega_{N}$. Set $m=\Omega_N-N$.

 The contracting square-hexagon lattice $\mathcal{R}(\Omega,\check{c})$ is a subgraph of $\mathrm{SH}(\check{c})$ built of $2N$ or $2N+1$ rows.
 The rows of $\mathcal{R}(\Omega,\check{c})$ inductively, starting from the bottom, can be enumerated as follows:
\begin{itemize}
\item The first row consists of vertices $(i,j)$ with $i=\Omega_1-\frac{1}{2},\ldots,\Omega_N-\frac{1}{2}$ and $j=\frac{1}{2}$. We call this row the boundary row of $\mathcal{R}(\Omega,\check{c})$.
\item When $k=2s$, for $s=1,\ldots N$,  the $k$th row consists of vertices $(i,j)$ with $j=\frac{k}{2}$ and incident to at least one vertex in the $(2s-1)th$ row of the whole-plane square-hexagon lattice $\mathrm{SH}(\check{c})$ lying between the leftmost vertex and rightmost vertex of the $(2s-1)$th row of $\mathcal{R}(\Omega,\check{c})$
\item When $k=2s+1$, for $s=1,\ldots N$,  the $k$th row consists of vertices $(i,j)$ with $j=\frac{k}{2}$ and incident to two vertices in the $(2s)$th row of  of $\mathcal{R}(\Omega,\check{c})$.
\end{itemize}
\end{definition}

\subsection{Partitions, Young diagrams and Schur functions}

We denote by $\GT_N$ the set of $N$-tuples $\lambda$ of integers satisfying $\lambda_1\geq \lambda_2\ldots\geq \lambda_N$, and let $\GT_N^+$ be a subset of $\GT_N$ consisting of all the $\lambda$'s in $\GT_N$ such that $\lambda_N\geq 0$. For $\lambda\in \GT_N^+$, Let
\begin{eqnarray*}
|\lambda|:=\sum_{i=1}^{N}\lambda_i.
\end{eqnarray*}
 
 A graphic way to represent a non-negative signature $\mu$ is through its
\emph{Young diagram} $Y_\lambda$, a collection of $|\lambda|$ boxes arranged on
non-increasing rows aligned on the left: with
$\lambda_1$ boxes on the first row, $\lambda_2$ boxes on the second row,\dots $\lambda_N$
boxes on the $N$th row.  Note that elements in $\GT_N^+$ are in bijection with all the Young diagrams with $N$ rows (rows are allowed to have zero length).

\begin{definition}
 Let $Y,W$ be two Young diagrams. We say that $Y\subset W$ \emph{differ by a
 horizontal strip} if the
  collection of boxes in $Z=W\setminus Y$ contains at most one box in every
  column. We say that they \emph{differ by a vertical strip} if $Z$ contains at
  most one box in every row.

  We say that two non-negative signatures $\lambda$ and $\mu$ \emph{interlace}, and
  write $\lambda \prec \mu$ if $Y_\lambda\subset Y_\mu$ differ by a horizontal
  strip. We say they \emph{cointerlace} and write $\lambda\prec'\mu$ if
  $Y_\lambda\subset Y_\mu$ differ by a vertical strip.
\end{definition}

\begin{definition}
Let $\lambda\in\GT_N^+$ be a partition of length $N$. We define the
counting measure $m(\lambda)$ corresponding to $\lambda$ as follows.

\begin{equation}
m(\lambda)=\frac{1}{N}\sum_{i=1}^{N}\delta\left(\frac{\lambda_i+N-i}{N}\right).\label{ml}
\end{equation}
\end{definition}

\begin{definition}Let $\lambda\in \GT_N$. The rational Schur function is
\begin{eqnarray*}
s_{\lambda}(u_1,\ldots,u_N)=\frac{\det_{i,j=1,\ldots,N}(u_i^{\lambda_j+N-j})}{\prod_{1\leq i<j\leq N}(u_i-u_j)}
\end{eqnarray*}
\end{definition}

\subsection{Dimer model}

\begin{definition}\label{dfvl}A dimer configuration, or a perfect matching $M$ of a contracting square-hexagon lattice $\mathcal{R}(\Omega,\check{c})$ is a set of edges $((i_1,j_1),(i_2,j_2))$, such that each vertex of $\mathcal{R}(\Omega,\check{c})$ belongs to an unique edge in $M$.
 The set of perfect matchings of $\mathcal{R}(\Omega,\check{c})$ is denoted by
  $\mathcal{M}(\Omega,\check{a})$.
\end{definition}

\begin{definition}
  Let $M\in \mathcal{M}(\Omega,\check{c})$ be a perfect matching of
  $\mathcal{R}(\Omega,\check{c})$. We call an edge $e=((i_1,j_1),(i_2,j_2))\in
  M$ a \emph{$V$-edge} if $\max\{j_1,j_2\}\in\NN$ (i.e.\@ if its higher
  extremity is black) and we call it a \emph{$\Lambda$-edge}
  otherwise. In other words, the edges going upwards starting from an odd row
  are $V$-edges and those ones starting from an even row are $\Lambda$-edges. We
  also call the corresponding vertices-$(i_1,j_1)$ and $(i_2,j_2)$ $V$-vertices
  and $\Lambda$-vertices accordingly.
\end{definition}

\begin{definition} The partition function of the dimer model of a finite graph
$G$ with edge weights $(w_e)_{e\in E(G)}$ is given by
\begin{equation*}
Z=\sum_{M\in \mathcal{M}}\prod_{e\in M}w_e,
\end{equation*}
where $\mathcal{M}$ is the set of all perfect matchings of $G$. The
Boltzmann dimer probability measure on $M$ induced by the weights $w$ is
thus defined by declaring that probability of a perfect matching is equal to
\begin{equation*}
  \frac{1}{Z}\prod_{e\in M} w_e.
\end{equation*}
\end{definition}

We shall associate to each perfect matching in $\mathcal{M}(\Omega,\check{a})$
a sequence of non-negative signatures, one for each row of the graph.

\begin{construction}\label{ct}To the boundary row $\Omega=(\Omega_1<\cdots<\Omega_N)$ of a contracting
square-hexagon lattice is naturally associated a non-negative signature $\omega$
of length $N$ by:
\begin{equation*}
  \omega=(\Omega_N-N,\dotsc,\Omega_1-1).
\end{equation*}

  Let $j\in\{2,\dots,2N+1\}$. Assume that the $j$th row of
  $\mathcal{R}(\Omega,\check{a})$ has $n_j$ V-vertices and $m_j$
  $\Lambda$-vertices. The a dimer configuration at the $j$th row of $\mathcal{R}(\Omega,\check{a})$
  corresponds to a signature $\mu\in \GT_{n_j}^+$, such that 
  \begin{itemize}
  \item $\mu=(\mu_1,\ldots,\mu_{n_j})$;
  \item We label all the $V$-vertices on the $j$th row by the 1st $V$-vertex, the 2nd $V$-vertex, \ldots, the $n_j$th $V$-vertex,
  such that the $1$st $V$-vertex is the rightmost $V$-vertex on the $j$th row.
  for $1\leq k\leq n_j$, $\mu_k$ is the number of $\Lambda$-vertices to the left of the $k$th $V$-vertex.
  \end{itemize}
\end{construction}

Then we have 
\begin{theorem}[\cite{BL17} Theorem 2.13]
  For given $\Omega$, $\check{c}$, let $\omega$ be the signature associated to
  $\Omega$. Then the construction~\ref{ct} defines a
  bijection between the set of perfect matchings
  $\mathcal{M}(\Omega,\check{c})$ and the set $S(\omega,\check{c})$ of
  sequences of non-negative signatures
  \begin{equation*}
    \{(\mu^{(N)},\nu^{(N)},\dots,\mu^{(1)}, \nu^{(1)}, \mu^{(0)}\}
  \end{equation*}
  where the signatures satisfy the following properties:
  \begin{itemize}
    \item All the parts of $\mu^{(0)}$ are equal to 0;
    \item The signature $\mu^{(N)}$ is equal to $\omega$;
    \item For $0\leq i\leq N$, $\mu^{(i)}\in \GT_i^{+}$.
    \item The signatures satisfy the following (co)interlacement relations:
      \begin{equation*}
        \mu^{(N)} \prec' \nu^{(N)} \succ \mu^{(N-1)} \prec' \cdots
        \mu^{(1)} \prec' \nu^{(1)} \succ \mu^{(0)}.
      \end{equation*}
  \end{itemize}
  Moreover, if $a_m=1$, then $\mu^{(N+1-k)}=\nu^{(N+1-k)}$.
  \label{myb}
\end{theorem}

The following proposition, proved in \cite{BL17}, shows that the partition function of dimer configurations on a contracting square-hexagon lattice can be computed by a Schur function depending on the boundary condition and the edge weights. Therefore it opens the door for investigating the asymptotics of periodic dimer model on a contracting square-hexagon lattice by studying the corresponding Schur functions.

\begin{proposition}\label{p16}Let $\mathcal{R}(\Omega,\check{c})$ be a contracting square-hexagon lattice built from a whole-plane square-hexagon lattice $\SH(\check{c})$ with edge weights $\{x_i,y_i,1\}_{1\leq i\leq N}$ assigned as in Assumption \ref{apew}.
Let
\begin{eqnarray*}
I_2=\{i|i\in\{1,2,\ldots,N\}, c_i=0\}
\end{eqnarray*}
Then 
 the partition function for perfect matchings on $\mathcal{R}(\Omega,\check{c})$ is given by
\begin{eqnarray*}
Z=\left[\prod_{i\in I_2}\Gamma_i\right] s_{\omega}(x_{1},\ldots,x_{N})
\end{eqnarray*}
where $\omega$ is the $N$-tuple corresponding to the boundary row of $\mathcal{R}(\Omega,\check{c})$, and $\Gamma_i$ is defined by

\begin{eqnarray}
\Gamma_i=\prod_{t=i+1}^{N}\left(1+y_{i}x_{t}\right).\label{gi}
\end{eqnarray}

\end{proposition}

\begin{proof}See Proposition 2.18 of \cite{BL17}.
\end{proof}

\section{Uniform Boundary Conditions}\label{unr}

In this section, we introduce the uniform boundary conditions on the bottom boundary of a contracting square hexagon lattice, and review the limit shape result of the dimer model on such a lattice.

Consider a contracting square-hexagon lattice $\mathcal{R}(\Omega,\check{c})$ with edge weights assigned as in Assumption \ref{apew}. Suppose that the configuration on the bottom row corresponds to the following signature
\begin{eqnarray}
\lambda(N)=((m-1)(N-1),(m-1)(N-2),\ldots,(m-1),0),\label{ln}
\end{eqnarray}
where $m\geq 1$ is a positive integer. More precisely, each remaining vertex on the boundary row is followed by $(m-1)$ removed vertices in the boundary row; the left most vertex and the rightmost vertex on the boundary row are remaining vertices. In Example  1.3.7 of \cite{IGM15}, the Schur function of such a signature is computed explicitly as follows
\begin{eqnarray}
s_{\lambda(N)}(x_1,\ldots,x_N)=\prod_{1\leq i<j\leq N}\frac{x_i^m-x_j^m}{x_i-x_j}.\label{sm}
\end{eqnarray}

\begin{definition}\label{df33}Let
\begin{eqnarray}
X=(x_1,x_2,\ldots,x_N)\in\RR^N\label{xn}
\end{eqnarray}
Let $\rho_N$ be a probability measure on $\GT_N$. The the \textbf{Schur generating function} with respect to $\rho_N$, $X$ is given by
\begin{eqnarray*}
\mathcal{S}_{\rho_N,X}(u_1,\ldots,u_N)=\sum_{\lambda\in \GT_N}\rho_N(\lambda)\frac{s_{\lambda}(u_1,\ldots,u_N)}{s_{\lambda}(x_1,\ldots,x_N)}
\end{eqnarray*}
\end{definition}

For a positive integer $s$, let $\ol{s}=s\mod n$. We make the following assumption on edge weights.

\begin{assumption}\label{pw}Assume that the edge weights $x_i$ ($1\leq i\leq N$), $y_j$ ($j\in I_2$) changes periodically with period $n$; i.e.
\begin{eqnarray}
x_{\ol{i}}=x_i\label{px};\\
y_{\ol{j}}=y_j\label{py}
\end{eqnarray}
for $1\leq i\leq n$.
\end{assumption}

\begin{lemma}\label{lmm212}Let $x_i>0 (1\leq i\leq N)$, $y_j (j\in I_2)$ be edge weights of a contracting square- hexagon lattice $\mathcal{R}(\Omega,\check{c})$ satisfying  Assumptions \ref{apew} and \ref{pw}.
Let 
\begin{eqnarray*}
X^{(N-t)}&=&(x_{\ol{t+1}},\ldots,x_{\ol{N}}),\\
Y^{(t)}&=&(x_{\ol{1}},\ldots,x_{\ol{t}}) 
\end{eqnarray*}
for each integer $t$ satisfying $0\leq t\leq N-1$, where $x_{i}>0\ (1\leq i\leq n)$ are weights of NE-SW edges joining the $(2i)$th row to the $(2i+1)$th row of the contracting square-hexagon lattice, see Figure \ref{fig:SH}. Let $\lambda(N)$ be the partition corresponding to the configuration on the boundary row; and let $\rho^k$ be the probability measure on $\GT_{N-t}^+$ which is the distribution of partitions corresponding to the dimer configuration on the $k$th row of vertices of $\mathcal{R}(\Omega,\check{c})$, counting from the bottom. Then
we have
\begin{eqnarray*}
\mathcal{S}_{\rho^k,X^{(N-t)}}(u_1,\ldots,u_{N-t})&=&\frac{s_{\lambda(N)}\left(u_1,\ldots,u_{N-t},Y^{(t)}\right)}{s_{\lambda(N)}(X^{(N)})}\prod_{i\in\{1,\ldots,t\}\cap I_2}\prod_{j=1}^{N-t}\left(\frac{1+y_{\ol{i}}u_j}{1+y_{\ol{i}}x_{\ol{t+j}}}\right),
\end{eqnarray*}
if $k=2t+1$, for $t=0,1,\ldots,N-1$

Moreover,
\begin{eqnarray*}
\mathcal{S}_{\rho_{N-t},X^{(N-t)}}(u_1,\ldots,u_{N-t})&=&\frac{s_{\lambda(N)}\left(u_1,\ldots,u_{N-t},Y^{(t)}\right)}{s_{\lambda(N)}(X^{(N)})}\prod_{i\in\{1,\ldots,t+1\}\cap I_2}\prod_{j=1}^{N-t}\left(\frac{1+y_{\ol{i}}u_j}{1+y_{\ol{i}}x_{\ol{t+j}}}\right),
\end{eqnarray*}
for $k=2t+2,\ t=0,1,\ldots,N-1$.
\end{lemma}
\begin{proof}See Lemma 3.17 of \cite{BL17}.
\end{proof}

\begin{lemma}\label{l33}Let $\kappa\in(0,1)$. Let $\mathcal{R}(\Omega,\check{c})$ be a contracting square-hexagon lattice.
Let $\{\lambda(\lfloor1-\kappa)N\rfloor)\}_{N\in \NN}$ be a sequence of partitions corresponding to dimer configurations on the $[2(N-\lfloor (1-\kappa )N\rfloor)+1]$th row of $\mathcal{R}(\Omega,\check{c})$, counting from the bottom. Let $X$ be an $N$-tuple of integers given by (\ref{xn}), which are also edge weights of $\mathcal{R}(\Omega,\check{c})$ satisfying (\ref{px}). Let $\rho_{\lfloor (1-\kappa )N\rfloor}:=\rho^{2(N-\lfloor (1-\kappa )N)\rfloor)+1}$ be a probability measure on $\GT_{\lfloor (1-\kappa )N\rfloor}^+$.  Note that $\rho_{N}:=\delta_{\lambda(N)}$ is the distribution of partitions corresponding to the dimer configurations on the bottom row, in which $\lambda(N)$ has probability 1 to occur, while any other configuration has probability 0 to occur. Let $\mathcal{S}_{\rho_{\lfloor(1-\kappa)N\rfloor},X}(u_1,\ldots,u_{\lfloor(1-\kappa)N\rfloor})$ be the Schur generating function corresponding to $\rho_{\lfloor(1-\kappa)N\rfloor}$ and $X$.
  Then we have
\begin{enumerate}
\item Assume $1\leq i\leq n$, then
\begin{eqnarray*}
&&\lim_{N\rightarrow\infty}\left.\frac{1}{(1-\kappa)N}\frac{\partial\log \mathcal{S}_{\rho_{(1-\kappa)N},X}(u_1,\ldots,u_{\lfloor(1-\kappa)N\rfloor})}{\partial u_i}\right|_{(u_1,\ldots,u_{\lfloor(1-\kappa)N\rfloor})=(x_{1+N-\lfloor(1-\kappa) N\rfloor},\ldots,x_{N})}\\
&=& H_i(X,Y,\kappa);
\end{eqnarray*}
where
\begin{eqnarray*}
H_i(X,Y,\kappa)&=&\frac{1}{(1-\kappa)n}\left\{\left[\sum_{j\in\{1,2,\ldots,n\},j\neq i}\left(\frac{mx_{i}^{m-1}}{x_{i}^m-x_j^m}-\frac{1}{x_{i}-x_j}\right)\right]+\frac{m-1}{2x_{i}}\right\}\\&&+\frac{\kappa}{(1-\kappa)n}\sum_{j\in\{1,2,\ldots,n\}\cap I_2}\frac{y_j}{1+y_j x_{i}}
\end{eqnarray*}
\item Assume $1\leq i,j\leq \lfloor(1-\kappa)N\rfloor$ and $i\neq j$. Then
\begin{eqnarray*}
\left.\lim_{N\rightarrow\infty}\frac{\partial^2\log \mathcal{S}_{\lfloor\rho_{(1-\kappa)N}\rfloor,X}(u_1,\ldots,u_{\lfloor(1-\kappa)N\rfloor})}{\partial u_i\partial u_j}\right|_{(u_1,\ldots,u_{\lfloor(1-\kappa)N\rfloor})=(x_{1+N-\lfloor(1-\kappa) N\rfloor},\ldots,x_{N})}=G(x_i,x_j);
\end{eqnarray*}
where
\begin{eqnarray*}
G(x_i,x_j)=
&=&\left\{\begin{array}{cc}\frac{m^2x_i^{m-1}x_j^{m-1}}{(x_i^m-x_j^m)^2}-\frac{1}{(x_i-x_j)^2}&\mathrm{if}\ x_i\neq x_j\\0&\mathrm{otherwise}\end{array}\right.
\end{eqnarray*}
\item Assume $i,j,k\in\{1,2,\ldots,n\}$ are three distinct integers, then
\begin{eqnarray*}
\left.\lim_{N\rightarrow\infty}\frac{\partial^3\log \mathcal{S}_{\rho_{\lfloor(1-\kappa)N\rfloor},X}(u_1,\ldots,u_{\lfloor(1-\kappa)N\rfloor})}{\partial u_i\partial u_j\partial u_k}\right|_{(u_1,\ldots,u_N)=(x_1,\ldots,x_N)}=0.
\end{eqnarray*}
\end{enumerate}
\end{lemma}

\begin{proof}Applying Lemma \ref{lmm212}, (\ref{sm}) and Definition \ref{df33}, and by explicit computations.
\end{proof}

\begin{remark}\label{rm25}Lemma \ref{l33} still holds if we define $\rho_{\lfloor (1-\kappa )N\rfloor}:=\rho^{2(N-\lfloor (1-\kappa )N)\rfloor)+2}$ and $\{\lambda(\lfloor1-\kappa)N\rfloor)\}_{N\in \NN}$ be a sequence of signatures corresponding to dimer configurations on the $[2(N-\lfloor (1-\kappa )N)\rfloor)+2]$th row of $\mathcal{R}(\Omega,\check{c})$.
\end{remark}

\begin{proposition}\label{plm}Let $\mathcal{R}(\Omega(N),\check{c})$ be a contracting square hexagon lattice with the configuration at the bottom boundary given by 
\begin{eqnarray*}
\Omega(N)=(1,m+1,2m+1,\ldots,(N-1)m+1)
\end{eqnarray*}
Assume also that the edge weights are assigned as in Assumption \ref{apew} (see  Figure \ref{fig:SH} for an example) and periodically with period $n$; i.e. the edge weights satisfy (\ref{px}) and (\ref{py}).
Let
\begin{eqnarray*}
F_{\kappa,m}(z)=\frac{\kappa z}{n(1-\kappa)}\sum_{i\in\{1,2,\ldots,n\}\cap I_2}\frac{y_i}{1+y_iz}+\sum_{j=1}^{n}\frac{z}{n(z-x_{j})}+\frac{z}{n(1-\kappa)}\sum_{j=1}^{n}\left(\frac{mz^{m-1}}{z^m-x_j^m}-\frac{1}{z-x_j}\right)
\end{eqnarray*}
Let $\rho_N^k$ be the measure on the configurations of the $k$th row, and let $\kappa\in (0,1)$, such that $k=[2\kappa N]$, Then the corresponding counting measure $m(\rho_N^k)$ converges to $\mathbf{m}^{\kappa}$ in probability as $N\rightarrow\infty$, and the moments of $\mathbf{m}^{\kappa}$ is given by
\begin{eqnarray}
\int_{\RR}y^{p}\textbf{m}^{\kappa}(dy)=\sum_{i=1}^{n}\frac{1}{2(p+1)\pi \mathbf{i}}\oint_{x_{t+i}}\frac{dz}{z}\left[ F_{\kappa,m}(z)\right]^{p+1}\label{mtl}
\end{eqnarray}

\end{proposition}

\begin{proof}See Section 8 of \cite{BL17}.
\end{proof}

We can compute the Stieltjes transform of the limit measure $\mathbf{m}^{\kappa}$ when $x$ is in neighborhood of infinity by
\begin{eqnarray*}
\sum_{j=0}^{\infty}\frac{\int_{\RR}y^{j}\textbf{m}^{\kappa}(dy)}{x^{j+1}}=-\sum_{i=1}^{n}\frac{1}{2\pi\mathbf{i}}\oint_{x_{t+i}}\frac{dz}{z}\mathrm{log}\left(1-\frac{F_{\kappa,m}(z)}{x}\right)
\end{eqnarray*}
Integrating by parts we have
\begin{eqnarray*}
\sum_{j=0}^{\infty}\frac{\int_{\RR}y^{j}\textbf{m}^{\kappa}(dy)}{x^{j+1}}=\sum_{i=1}^{n}\frac{1}{2\pi\mathbf{i}}\oint_{x_{t+i}}\log(z)\frac{\partial_z\left(1-\frac{F_{\kappa,m}(z)}{x}\right)}{\left(1-\frac{F_{\kappa,m}(z)}{x}\right)}dz
\end{eqnarray*}

The integrand has poles at roots of 
\begin{eqnarray}
F_{\kappa,m}(z)=x.\label{fmz}
\end{eqnarray}

\begin{definition}
  \label{df41}Let $\mathcal{R}$ be the rescaled square-hexagon lattice, i.e. $\mathcal{R}=\frac{1}{N}\mathcal{R}(\Omega,\check{c})$, with coordinates $(\chi,\kappa)$.
  Let $\mathcal{L}$ be the set of $(\chi,\kappa)$ inside $\mathcal{R}$ such that
  the density $d\mathbf{m}^{\kappa}\left(\frac{\chi}{1-\kappa}\right)$ is not
  equal to 0 or 1. Then $\mathcal{L}$ is called the \emph{liquid region}. Its boundary
  $\partial \mathcal{L}$ is called the \emph{frozen boundary}.
\end{definition}

\section{Central Limit Theorem for Uniform Boundary Conditions} \label{unc}

In this section, we construct certain statistics from the (random) dimer configuration on a contracting square hexagon lattice with uniform boundary conditions, and show that the converge in distribution to Gaussian random variables in the scaling limit. The main theorem proved in this section is Theorem \ref{gff1}.

Let $X$ be given by (\ref{xn}), and let $V_N(X)$ be the Vandermonde determinant, i.e.
\begin{eqnarray*}
V_N(X)=\prod_{1\leq i<j\leq N}(x_j-x_i)
\end{eqnarray*}

\begin{proposition}\label{pn41}Let $\rho_N$ be a probability measure on $\GT_N$, and let $\lambda\in\GT_N$. Let
\begin{eqnarray*}
U=(u_1,u_2,\ldots,u_N)\in \CC^N.
\end{eqnarray*}
 Then
\begin{eqnarray}
\mathbf{E}\sum_{i=1}^{N}(\lambda_i+N-i)^k:&=&\sum_{\lambda\in\GT_N}\rho_N(\lambda)\sum_{i=1}^N(\lambda_i+N-i)^k\notag\\
&=&\left.\frac{1}{V_N(U)}\sum_{i=1}^{N}(u_i\partial_i)^k V_N(U) \mathcal{S}_{\rho_N,X}(U)\right|_{U=X};\label{e1}
\end{eqnarray}
and
\begin{eqnarray}
&&\mathbf{E}\left(\sum_{i=1}^{N}(\lambda_i+N-i)^k\sum_{j=1}^{N}(\lambda_j+N-j)^l\right)\notag\\
&=&\left.\frac{1}{V_N(U)}\sum_{i=1}^N(u_i\partial_i)^k\sum_{j=1}^{N}(u_j\partial_j)^l V_N(U) \mathcal{S}_{\rho_N,X}(U)\right|_{U=X}\label{e2}
\end{eqnarray}
\end{proposition}
\begin{proof}Let $\lambda\in \GT_N$, and let $s_{\lambda}$ be the Schur function with respect to $\lambda$. By Proposition 4.3 of \cite{bg}, we have
\begin{eqnarray}
\frac{1}{V_N(U)}\sum_{i=1}^{N}(u_i\partial_i)^k V_N(U) s_{\lambda}(U)=\sum_{i=1}^{N}(\lambda_i+N-1)^k s_{\lambda}(U).\label{ds}
\end{eqnarray}
Dividing by $s_{\lambda}(X)$ to both sides of (\ref{ds}), then taking expectations for $\lambda$ with respect to the distribution $\rho_N$; then evaluate at $U=X$, we obtain (\ref{e1}). The expression (\ref{e2}) can be obtained similarly by performing the above process twice.
\end{proof}

Let $f(x_1,\ldots,x_r)$ be a function of $r$ variables. Define
\begin{eqnarray*}
\mathrm{Sym}_{x_1,\ldots,x_r}f(x_1,\ldots,x_r)=\frac{1}{r!}\sum_{\sigma\in \Si_r} f(x_{\sigma(1)},x_{\sigma(2)},\ldots,x_{\sigma(r)}),
\end{eqnarray*}
where $\Si_r$ is the symmetric group of $r$ elements.

For an integer $l>0$, and $t\in(0,1]$, let
\begin{eqnarray}
U_t&=&(u_1,u_2,\ldots,u_{\lfloor tN\rfloor});\label{ut}\\
X_{t}&=&(x_{N-\lfloor tN\rfloor+1},\ldots, x_N).\label{xt}
\end{eqnarray}
and
\begin{eqnarray}
\mathcal{F}_{(l,t)}(U_t)=\frac{1}{\mathcal{S}_{\rho_{\lfloor tN\rfloor},X_t}(U_t)V_{\lfloor tN\rfloor}(U_t)}\sum_{i=1}^{\lfloor tN\rfloor}(u_i\partial_i)^l V_{\lfloor tN\rfloor}(U_t)\mathcal{S}_{\rho_{\lfloor tN\rfloor},X_t}(U_t);\label{flu}
\end{eqnarray}
where $\rho_{\lfloor tN\rfloor}$ is a probability measure on $\GT_{\lfloor tN\rfloor}$.

In order to analyze the asymptotics, we first introduce the following technical lemma.

\begin{lemma}\label{l552}Let $f(z)$ be a complex analytic function in a neighborhood of $1$ and let $r$ be a positive integer. Then
\begin{eqnarray*}
\mathrm{Sym}_{z_1,\ldots,z_{r+1}}\left.\left(\frac{f(z_1)}{(z_1-z_2)\ldots(z_1-z_{r+1})}\right)\right|_{(z_1,\ldots,z_{r+1})=(1,\ldots,1)}=\left.\frac{1}{(r+1)!}\frac{\partial^r f(z)}{\partial z^r}\right|_{z=1}.
\end{eqnarray*}
\end{lemma}

\begin{proof}See Lemma 5.5 of \cite{bg}.
\end{proof}

\begin{proposition}\label{p36}Assume the assumption of Lemma \ref{l33} holds. We use the notation $\partial_i$ to denote $\frac{\partial}{\partial u_i}$. Then we have
\begin{enumerate}
\item the functions $\mathcal{F}_{(l,t)}(U_t)|_{U_t=X_t}$ have $N$-degree at most $l+1$;
\item for $1\leq i\leq N$, the functions $\partial_i\mathcal{F}_{(l,t)}(U_t)|_{U_t=X_t}$ have $N$-degree at most $l$; moreover
\begin{eqnarray*}
&&\partial_i\mathcal{F}_{(l,t)}(U_t)|_{U_t=X_t}=\partial_i\left[\sum_{r=0}^{l}\left(\begin{array}{c}l\\r\end{array}\right)(r+1)!\right.\\
&&\times\sum_{\{a_1,\ldots,a_{r+1}\}\subset\{1,2,\ldots,\lfloor{tN}\rfloor\}}\left.\left.\mathrm{Sym}_{a_1,\ldots,a_{r+1}}\left(\frac{u_{a_1}^l(\partial_{a_1}[\log S_{\rho_{\lfloor tN\rfloor},X_t}])^{l-r}}{(u_{a_1}-u_{a_2})\ldots(u_{a_1}-u_{a_{r+1}})}\right)\right]\right|_{U_t=X_t}+T_{(l,t)}(U_t)|_{U_t=X_t}
\end{eqnarray*}
where $T_{(l,t)}(X_t)$ has $N$-degree less than $l$.
\item for any  $1\leq i, j\leq N$ and $i\neq j$, the functions $\partial_i\partial_j \mathcal{F}_{(l,t)}(U_t)|_{U_t=X_t}$ have $N$-degree at most $l-1$.
\end{enumerate}
\end{proposition}
\begin{proof}When  $t=1$ and $X_1=(1,\ldots,1)$, the proposition is proved in Lemma 5.5 of \cite{bg16}. Consider a general $S_{\rho_{\lfloor tN\rfloor},X_t}$ with $X_t$ given by (\ref{xt}). Since $S_{\rho_{\lfloor tN\rfloor},X_t}(X_t)=1$, the function $\log S_{\rho_{\lfloor tN\rfloor},X_t}$ is well defined in a neighborhood of $X_t$. Note that
\begin{eqnarray*}
\frac{\partial_i S_{\rho_{\lfloor tN\rfloor},X_t}}{S_{\rho_{\lfloor tN\rfloor},X_t}}=\partial_i(\log S_{\rho_{\lfloor tN\rfloor},X_t}).
\end{eqnarray*}
This way we can write $\mathcal{F}_{(l,t)}(U_t)$ as a large sum of factors of the form
\begin{eqnarray*}
\frac{c_0u_i^{l-s_0}(\partial_i^{s_1}[\log S_{\rho_{\lfloor tN\rfloor}, X_t}])^{d_1}\ldots (\partial_i^{s_t}[\log S_{\rho_{\lfloor tN\rfloor}, X_t}])^{d_t}}{(u_i-u_{a_1})\ldots (u_i-u_{a_r})}
\end{eqnarray*}
where $i,\ a_1,\ \ldots,\ a_r$ are distinct indices, $s_j,d_j\in \NN\cup\{0\}$ for $j=1,\ldots t$, and
\begin{eqnarray}
s_1<s_2<\ldots<s_t;\notag\\
r+s_0+\sum_{j=1}^t s_jd_j=l.\label{fsc}
\end{eqnarray}
Moreover, $c_0$ depends on $r$, $s_j$, $d_j$, but is independent of $N$, $a_1,\ldots,a_r$. By symmetry we can write
\begin{eqnarray}
&&\mathcal{F}_{(l,t)}(U_t)=\sum_{r,\{s_j\},\{d_j\}}(r+1)!\times\sum_{\{a_1,\ldots,a_{r+1}\}\subset \{1,2,\ldots,\lfloor tN\rfloor\}}\label{exf}\\
&& \mathrm{Sym}_{a_1,\ldots,a_{r+1}}\left(\frac{c_0u_{a_1}^{l-s_0}(\partial_{a_1}^{s_1}[\log S_{\rho_{\lfloor tN\rfloor}, X_t}])^{d_1}\ldots (\partial_{a_1}^{s_t}[\log S_{\rho_{\lfloor tN\rfloor}, X_t}])^{d_t}}{(u_{a_1}-u_{a_2})\ldots (u_{a_1}-u_{a_{r+1}})}\right),\notag
\end{eqnarray}
where the first sum are over $r,\{s_j\},\{d_j\}$ satisfying (\ref{fsc}), and $c_0$ depends on $r,\{s_j\},\{d_j\}$.

By Lemma \ref{l33}, for each $1\leq w\leq t$ the degree of $N$ in each factor $(\partial_{a_1}^{s_w}[\log S_{\rho_{\lfloor tN\rfloor}, X_t}])^{d_w}$ is at most $d_w$. For each given choice of $\{a_1,\ldots,a_{r+1}\}$, we define an equivalence relation on the set $\{a_1,\ldots,a_{r+1}\}$: for $1\leq i,j\leq r+1$, we say $a_i$ and $a_j$ are equivalent if and only if $[a_i\mod n]=[a_j\mod n]$. Let $A_1,\ldots, A_{w}$ be all the distinct equivalence classes under this equivalence relation, where $w$ is a positive integer satisfying $w\leq r+1$. For $1\leq i\leq w$, let $C_i=\{a_1,\ldots,a_{r+1}\}\setminus A_i$

For $1\leq i\leq r+1$, let $a_1,\ldots, \hat{a}_i,\ldots,a_{r+1}$ be $r$ distinct integers obtained from $a_1,\ldots, a_{r+1}$ by removing $a_i$. Then
\begin{eqnarray}
&&\mathrm{Sym}_{a_1,\ldots,a_{r+1}}\left(\frac{c_0u_{a_1}^{l-s_0}(\partial_{a_1}^{s_1}[\log S_{\rho_{\lfloor tN\rfloor}, X_t}])^{d_1}\ldots (\partial_{a_1}^{s_t}[\log S_{\rho_{\lfloor tN\rfloor}, X_t}])^{d_t}}{(u_{a_1}-u_{a_2})\ldots (u_{a_1}-u_{a_{r+1}})}\right)\label{sm1}\\
&=&\frac{1}{(r+1)!}\sum_{i=1}^{w}{{r}\choose{|C_i|}}|A_i|!|C_i|!\mathrm{Sym}_{A_i}\left[\mathrm{Sym}_{C_i}\right.\notag\\&&\left(\frac{c_0u_{a_i}^{l-s_0}(\partial_{a_i}^{s_1}[\log S_{\rho_{\lfloor tN\rfloor}, X_t}])^{d_1}\ldots (\partial_{a_i}^{s_t}[\log S_{\rho_{\lfloor tN\rfloor}, X_t}])^{d_t}}{\prod_{j\in C_i}(u_{a_i}-u_{a_j})}\right)\notag\\
&&\times\left.\left(\frac{1}{\prod_{a_j\in A_i}(u_{a_i}-u_{a_j})}\right)\right]\notag\\
&=&\sum_{i=1}^{w}\frac{|A_i|}{r+1}\mathrm{Sym}_{A_i}\left[\left(\frac{c_0u_{a_i}^{l-s_0}(\partial_{a_i}^{s_1}[\log S_{\rho_{\lfloor tN\rfloor}, X_t}])^{d_1}\ldots (\partial_{a_i}^{s_t}[\log S_{\rho_{\lfloor tN\rfloor}, X_t}])^{d_t}}{\prod_{j\in C_i}(u_{a_i}-u_{a_j})}\right)\right.\notag\\
&&\times\left.\left(\frac{1}{\prod_{a_j\in A_i}(u_{a_i}-u_{a_j})}\right)\right]\notag
\end{eqnarray}
By Lemma \ref{l33} and (\ref{fsc}), the degree of $N$ in 
\begin{eqnarray*}
\left(\frac{c_0u_{a_i}^{l-s_0}(\partial_{a_i}^{s_1}[\log S_{\rho_{\lfloor tN\rfloor}, X_t}])^{d_1}\ldots (\partial_{a_i}^{s_t}[\log S_{\rho_{\lfloor tN\rfloor}, X_t}])^{d_t}}{\prod_{j\in C_i}(u_{a_i}-u_{a_j})}\right)
\end{eqnarray*}
is at most $l-r$. By Lemma \ref{l552}, the degree of $N$ in (\ref{sm1}) is at most $l-r$. Summing over all the choices $\{a_1,\ldots,a_{r+1}\}\subset\{1,2,\ldots,\lfloor tN\rfloor\}$ (there are $O(N^{r+1})$ such choices), we obtain that the degree of $N$ in $\mathcal{F}_{(l,t)}(U_t)$ is at most $l+1$; then Part (1) of the proposition follows.
\end{proof}

For positive integers $l_1,l_2$, we define
\begin{eqnarray*}
\mathcal{G}_{l_1,l_2,t}(U_t)&=&l_1\sum_{r=0}^{l_1-1}\left(\begin{array}{c}l_1-1\\r\end{array}\right)\sum_{\{a_1,\ldots,a_{r+1}\}\subset\{1,2,\ldots,\lfloor tN\rfloor\}}(r+1)!\\&&\times\mathrm{Sym}_{a_1,\ldots,a_{r+1}}\frac{u_{a_1}^{l_1}\partial_{a_1}[\mathcal{F}_{(l_2,t)}](\partial_{a_1}[\log S_{\rho_{\lfloor tN\rfloor},X_t}])^{l_1-1-r}}{(u_{a_1}-u_{a_2})\ldots (u_{a_1}-u_{a_{r+1}})}
\end{eqnarray*}

\begin{lemma}Assume the assumption of Lemma \ref{l33} holds.
Let $l_1,l_2$ be arbitrary positive integers, and $t\in(0,1]$, then
\begin{eqnarray}
&&\frac{1}{V_{\lfloor tN\rfloor}S_{\rho_{\lfloor tN\rfloor},X_t}}\sum_{i_1=1}^{\lfloor tN\rfloor}(u_{i_1}\partial_{i_1})^{l_1}\sum_{i_2=1}^{\lfloor tN\rfloor} (u_{i_2}\partial_{i_2})^{l_2}[V_{\lfloor tN\rfloor} S_{\rho_{\lfloor tN\rfloor},X_t}]\notag\\
&=&\mathcal{F}_{(l_1,t)}(U_t)\mathcal{F}_{(l_2,t)}(U_t)+\mathcal{G}_{(l_1,l_2,t)}(U_t)+T(U_t)\label{l5i}
\end{eqnarray}
where $\mathcal{G}_{(l_1,l_2,t)}(U_t)|_{U_t=X_t}$ has $N$-degree at most $l_1+l_2$ and $T(U_t)|_{U_t=X_t}$ has $N$ degree less than $l_1+l_2$. Moreover, for any index $i$ the function $\partial_i\mathcal{G}_{(l_1,l_2,t)}(U_t)|_{U_t=X_t}$ has $N$-degree less than $l_1+l_2$.
\end{lemma}
\begin{proof}The proof follows from similar arguments as in the proof of Lemma 5.7 in \cite{bg16}, in which the case $X=1^N$ and $t=1$ is proved. We sketch the idea here. Note that the left hand side of (\ref{l5i}) is exactly
\begin{eqnarray*}
\frac{1}{V_{\lfloor tN\rfloor} S_{\rho_{\lfloor tN\rfloor},X_t}}\sum_{i_1=1}^{N}(u_{i_1}\partial_{i_1})^{l_1}[V_{\lfloor tN\rfloor} \mathcal{S}_{\rho_{\lfloor tN\rfloor},X_t} \mathcal{F}_{(l_2)}(U_t)]
\end{eqnarray*}
It can be rewritten as the sum of terms of the form
\begin{eqnarray*}
\mathrm{Sym}_{a_1,\ldots,a_{r+1}}\frac{c_0 u_{a_1}^{l_1-s_0}\partial_{a_1}^{s_1}[\mathcal{F}_{(l_2,t)}](\partial_{a_1}^{s_2}[\log S_{\rho_{\lfloor tN\rfloor},X_t}])^{d_2}\ldots(\partial_{a_1}^{s_p}[\log S_{\rho_{\lfloor tN\rfloor},X_t}])^{d_p}}{(u_{a_1}-u_{a_2})(u_{a_1}-u_{a_3})\ldots(u_{a_1}-u_{a_{r+1}})},
\end{eqnarray*}
where $r,s_0,s_1,\ldots,s_p,d_2,\ldots,d_p$ are nonnegative integers and 
\begin{eqnarray*}
&&s_2<s_3<\ldots< s_p;\\
&&s_0+s_1+s_2d_2+\ldots+s_pd_p+r=l_1.
\end{eqnarray*}
Then $\mathcal{F}_{(l_1,t)}(U_t)\mathcal{F}_{(l_2,t)}(U_t)$ comes from the terms with $s_1=0$; $\mathcal{G}_{(l_1,l_2,t)}(U_t)$ comes from the terms with $s_0=0$, $s_1=1$, $s_2=1$, $d_2=l_1-1-r$. The $N$-degrees of these terms can be obtained by applying Lemma \ref{l33}.
\end{proof}

Let $s$ be a positive integer. For a subset $\{j_1,\ldots,j_p\}\subset \{1,2,\ldots s\}$, let $\mathcal{P}^s_{j_1,\ldots,j_p}$ be the set of all pairings of the set $\{1,2,\ldots,s\}\setminus\{j_1,\ldots,j_p\}$. The set $\mathcal{P}^s_{j_1,\ldots,j_p}$ is non-empty only when $s-p$ is even. For a pairing $P$ let  $\prod_{(a,b)\in P}$ denote the product over all pairs $(a,b)$ from this pairing.

\begin{proposition}Assume that the assumption of Lemma \ref{l33} holds.
Let $s,l_1,\ldots,l_s$ be arbitrary positive integers, and let $t\in (0,1]$. Then
\begin{eqnarray*}
&&\frac{1}{V_{\lfloor tN\rfloor} S_{\rho_{\lfloor tN\rfloor},X_t}}\sum_{i_1=1}^{\lfloor tN\rfloor}(u_{i_1}\partial_{i_1})^{l_1}\sum_{i_2=1}^{\lfloor tN\rfloor}(u_{i_2}\partial_{i_2})^{l_2}\ldots \sum_{i_s=1}^{\lfloor tN\rfloor}(u_{i_s}\partial_{i_s})^{l_s}[V_{\lfloor tN\rfloor} S_{\rho_{\lfloor tN\rfloor},X_t}]\\
&=&\sum_{p=0}^{s}\sum_{\{j_1,\ldots,j_p\}\subset\{1,2,\ldots,s\}}\mathcal{F}_{(l_{j_1},t)}(U_t)\ldots \mathcal{F}_{(l_{j_p},t)}(U_t)\left(\sum_{P\in \mathcal{P}^s_{j_1,\ldots,j_p}}\prod_{(a,b)\in P}\mathcal{G}_{(l_a,l_b,t)}(U_t)+T_{j_1,\ldots,j_p}^{1;s}(U_t)\right),
\end{eqnarray*}
where $T_{j_1,\ldots,j_p}^{1;s}(U_t)|_{U_t=X_t}$ has $N$-degree less than $\sum_{i=1}^{s}l_i-\sum_{i=1}^p l_{j_i}$.
\end{proposition}
\begin{proof}The proposition can be proved by induction on $s$, similar to the proof of Proposition 5.10 of \cite{bg}, where the case when $X=1^N$ is proved.
\end{proof}

Let $l$ be a positive integer and $t\in(0,1]$. Let 
\begin{eqnarray*}
E_{l,t}:=\mathcal{F}_{(l,t)}(X_t)=\frac{1}{V_{\lfloor tN\rfloor}S_{\rho_{\lfloor tN\rfloor}, X_t}}\sum_{i=1}^{\lfloor tN\rfloor}(u_i\partial_i)^l V_{\lfloor tN\rfloor} S_{\rho_{\lfloor tN\rfloor}, X_t}(U_t)|_{U_t=X_t}
\end{eqnarray*}

\begin{lemma}\label{lc}Assume the assumption of Lemma \ref{l33} holds. Let $s, l_1,\ldots,l_s$ be arbitrary positive integers, and let $t\in(0,1]$. Then
\begin{eqnarray*}
&&\frac{1}{V_{\lfloor tN\rfloor} S_{\rho_{\lfloor tN\rfloor}, X_t}}\left(\sum_{i_1=1}^{\lfloor t N\rfloor}(u_{i_1}\partial_{i_1})^{l_1}-E_{l_1,t}\right)\left(\sum_{i_2=1}^{\lfloor tN\rfloor}(u_{i_2}\partial_{i_2})^{l_2}-E_{l_2,t}\right)\\
&&\left.\times\cdots\left(\sum_{i_2=1}^{\lfloor tN\rfloor}(u_{i_s}\partial_{i_s})^{l_s}-E_{l_s,t}\right)V_{\lfloor tN\rfloor} S_{\rho_{\lfloor tN\rfloor},X_t}\right|_{U_t=X_t}=\left.\sum_{P\in \mathcal{P}_{\emptyset}^s}\prod_{(a,b)\in P}\mathcal{G}_{(l_a,l_b,t)}(U)\right|_{U_t=X_t}+T_{\emptyset}(U_t)|_{U_t=X_t},
\end{eqnarray*}
where $T_{\emptyset}(U_t)|_{U_t=X_t}$ has $N$-degree less than $\sum_{i=1}^s l_i$.
\end{lemma}

Let $\kappa\in[0,1)$ and $\lambda\in\GT_{N-\lfloor \kappa N\rfloor}$. 
\begin{eqnarray*}
p_j^{(N-\lfloor \kappa N\rfloor)}=\sum_{i=1}^{N-\lfloor\kappa N\rfloor}(\lambda_i+(N-\lfloor\kappa N\rfloor)-i)^j;\qquad \mathrm{for}\ j=1,2,\ldots.
\end{eqnarray*}
Assume the distribution of $\lambda$ is $\rho_{N-\lfloor\kappa N\rfloor}$. Let $\mathbf{E}$ be the expectation under the probability measure $\rho_{N-\lfloor \kappa N\rfloor}$, and let
\begin{eqnarray*}
\mathrm{cov}\left(p_{k}^{(N-\lfloor\kappa N\rfloor)},p_{l}^{(N-\lfloor\kappa N\rfloor)}\right)=\mathbf{E}\left(p_{k}^{(N-\lfloor \kappa N \rfloor)},p_{l}^{(N-\lfloor \kappa N \rfloor)}\right)-\mathbf{E}p_{k}^{(N-\lfloor\kappa N\rfloor)}\mathbf{E}p_{l}^{(N-\lfloor \kappa N\rfloor)}.
\end{eqnarray*}
 then by Lemma \ref{lc}, we have
\begin{eqnarray}
\lim_{N\rightarrow\infty}\frac{\mathrm{cov}\left(p_{k}^{(N-\lfloor \kappa N\rfloor)},p_{l}^{(N-\lfloor \kappa N\rfloor)}\right)}{N^{k+l}}=\lim_{N\rightarrow\infty}\frac{\mathcal{G}_{(k,l)}(x_{\lfloor\kappa N\rfloor+1},\ldots,x_N)}{N^{k+l}}.\label{cov}
\end{eqnarray}

We have
\begin{eqnarray*}
&&\mathcal{G}_{(k,l)}(x_{\lfloor\kappa N\rfloor+1},\ldots,x_N)\\
&=&k\sum_{q=0}^{k-1}\sum_{\{a_1,\ldots,a_{q+1}\}\subset\{1,2,\ldots,N-\lfloor\kappa N \rfloor\}}\left(\begin{array}{c}k-1\\q\end{array}\right)(q+1)!\\&&\left.\mathrm{Sym}_{a_1,\ldots,a_{q+1}}\frac{u_{a_1}^k\partial_{a_1}[\mathcal{F}_{(l)}](\partial_{a_1}[\log S_{\rho_{N-\lfloor\kappa N\rfloor},(x_{\lfloor\kappa N\rfloor+1},\ldots,x_N)}]^{k-1-q})}{(u_{a_1}-u_{a_2})\ldots (u_{a_1}-u_{a_{q+1}})}\right|_{(u_1,\ldots,u_{N-\lfloor \kappa N\rfloor})=X_{1-\kappa}}\\
&\approx&k\sum_{q=0}^{k-1}\sum_{\{a_1,\ldots,a_{q+1}\}\subset\{1,2,\ldots,N-\lfloor\kappa N \rfloor\}}\left(\begin{array}{c}k-1\\q\end{array}\right)(q+1)!\\
&&\mathrm{Sym}_{a_1,\ldots,a_{q+1}}\left(\frac{u_{a_1}^k(\partial_{a_1}[\log S_{\rho_{N-\lfloor\kappa N\rfloor},(x_{\lfloor\kappa N\rfloor+1},\ldots,x_N)}]^{k-1-q})}{(u_{a_1}-u_{a_2})\ldots (u_{a_1}-u_{a_{q+1}})}\right.\\
&\times&\partial_{a_1}\left[\sum_{r=0}^{l}\sum_{\{b_1,\ldots,b_{r+1}\}\subset\{1,2,\ldots,N-\lfloor\kappa N \rfloor\}}\left(\begin{array}{c}l\\r\end{array}\right)(r+1)!\right.\\
&\times &\left.\left.\left.\mathrm{Sym}_{b_1,\ldots,b_{r+1}}\frac{u_{b_1}^l(\partial_{b_1}[\log S_{\rho_{N-\lfloor\kappa N\rfloor},(x_{\lfloor\kappa N\rfloor+1},\ldots,x_N)}])^{l-r}}{(u_{b_1}-u_{b_2})\ldots(u_{b_1}-u_{b_{r+1}})}\right]\right)\right|_{(u_1,\ldots,u_{N-\lfloor \kappa N\rfloor})=X_{1-\kappa}}
\end{eqnarray*}
The approximate equality above contains only leading terms of $\partial_{a_1}[\mathcal{F}_{(l,1-\kappa)}]$; see Proposition \ref{p36} (2).

We first consider the case that
\begin{eqnarray*}
\{a_1,\ldots,a_{q+1}\}\cap\{b_1,\ldots,b_{r+1}\}=\emptyset.
\end{eqnarray*}

By Lemma \ref{l33}, we have
\begin{eqnarray*}
&&\partial_{a_1}\left[\sum_{r=0}^{l}\sum_{\{b_1,\ldots,b_{r+1}\}\subset\{1,2,\ldots,N-\lfloor\kappa N \rfloor\}}\left(\begin{array}{c}l\\r\end{array}\right)(r+1)!\right.\\
&\times &\left.\left.\left.\mathrm{Sym}_{b_1,\ldots,b_{r+1}}\frac{u_{b_1}^l(\partial_{b_1}[\log S_{\rho_{N-\lfloor\kappa N\rfloor},X_{\kappa}}])^{l-r}}{(u_{b_1}-u_{b_2})\ldots(u_{b_1}-u_{b_{r+1}})}\right]\right)\right|_{(u_1,\ldots,u_{N-\lfloor \kappa N\rfloor})=X_{1-\kappa}}\\
&=&\sum_{r=0}^{l}\sum_{\{b_1,\ldots,b_{r+1}\}\subset\{1,2,\ldots,N-\lfloor\kappa N \rfloor\}}\left(\begin{array}{c}l\\r\end{array}\right)(r+1)!(l-r)\\
&\times &\left.\mathrm{Sym}_{b_1,\ldots,b_{r+1}}\frac{u_{b_1}^l(\partial_{b_1}[\log S_{\rho_{N-\lfloor\kappa N\rfloor},X_{\kappa}}])^{l-r-1}\partial_{a_1}\partial_{b_1}[\log S_{\rho_{N-\lfloor\kappa N\rfloor},X_{\kappa}}]}{(u_{b_1}-u_{b_2})\ldots(u_{b_1}-u_{b_{r+1}})}\right|_{(u_1,\ldots,u_{N-\lfloor \kappa N\rfloor})=X_{1-\kappa}}\\
&\approx&\sum_{r=0}^{l}\sum_{\{b_1,\ldots,b_{r+1}\}\subset\{1,2,\ldots,N-\lfloor\kappa N \rfloor\}}\left(\begin{array}{c}l\\r\end{array}\right)(r+1)!(l-r)\\
&\times &\left.\mathrm{Sym}_{b_1,\ldots,b_{r+1}}\frac{u_{b_1}^l(H_{b_1}(X,Y,\kappa))^{l-r-1}N^{l-r-1}G(x_{a_1},x_{b_1})}{(u_{b_1}-u_{b_2})\ldots(u_{b_1}-u_{b_{r+1}})}\right|_{(u_1,\ldots,u_{N-\lfloor \kappa N\rfloor})=X_{1-\kappa}}.
\end{eqnarray*}
Moreover,
\begin{eqnarray*}
&&\mathrm{Sym}_{a_1,\ldots,a_{q+1}}\left(\frac{u_{a_1}^k(\partial_{a_1}[\log S_{\rho_{N-\lfloor\kappa N\rfloor},(x_{\lfloor\kappa N\rfloor+1},\ldots,x_N)}]^{k-1-q})}{(u_{a_1}-u_{a_2})\ldots (u_{a_1}-u_{a_{q+1}})}\right.\\
&\times&\partial_{a_1}\left[\sum_{r=0}^{l}\sum_{\{b_1,\ldots,b_{r+1}\}\subset\{1,2,\ldots,N-\lfloor\kappa N \rfloor\}}\left(\begin{array}{c}l\\r\end{array}\right)(r+1)!\right.\\
&\times &\left.\left.\left.\mathrm{Sym}_{b_1,\ldots,b_{r+1}}\frac{u_{b_1}^l(\partial_{b_1}[\log S_{\rho_{N-\lfloor\kappa N\rfloor},(x_{\lfloor\kappa N\rfloor+1},\ldots,x_N)}])^{l-r}}{(u_{b_1}-u_{b_2})\ldots(u_{b_1}-u_{b_{r+1}})}\right]\right)\right|_{(u_1,\ldots,u_{N-\lfloor \kappa N\rfloor})=(x_{\lfloor\kappa N\rfloor+1},\ldots,x_N)}\\
&\approx&\mathrm{Sym}_{a_1,\ldots,a_{q+1}}\left(\frac{u_{a_1}^k([H_{a_1}(X,Y,\kappa)]^{k-1-q} N^{k-1-q})}{(u_{a_1}-u_{a_2})\ldots (u_{a_1}-u_{a_{q+1}})}\right.\\
&&\sum_{r=0}^{l}\sum_{\{b_1,\ldots,b_{r+1}\}\subset\{1,2,\ldots,N-\lfloor\kappa N \rfloor\}}\left(\begin{array}{c}l\\r\end{array}\right)(r+1)!(l-r)\\
&\times &\left.\mathrm{Sym}_{b_1,\ldots,b_{r+1}}\frac{u_{b_1}^l(H_{b_1}(X,Y,\kappa))^{l-r-1}N^{l-r-1}G(x_{a_1},x_{b_1})}{(u_{b_1}-u_{b_2})\ldots(u_{b_1}-u_{b_{r+1}})}\right|_{(u_1,\ldots,u_{N-\lfloor \kappa N\rfloor})=X_{1-\kappa}}.
\end{eqnarray*}

\begin{lemma}
Let 
\begin{eqnarray*}
H(z)&=&\frac{1}{n}\left[\sum_{j\in\{1,2,\ldots,n\}}\left(\frac{mz^{m-1}}{z^m-x_j^m}-\frac{1}{z-x_j}\right)\right]
+\frac{\kappa}{n}\sum_{j\in\{1,2,\ldots,n\}\cap I_2}\frac{y_j}{1+y_j z}.
\end{eqnarray*}

The contribution of the terms when $\{a_1,\ldots,a_{q+1}\}\cap\{b_1,\ldots,b_{r+1}\}=\emptyset$ to $\mathcal{G}_{(k,l,1-\kappa)}(X_{1-\kappa})$, as $N\rightarrow\infty$, is asymptotically
\begin{eqnarray}
\frac{(1-\kappa)^{k+l}N^{k+l}}{(2\pi\mathbf{i})^2}\sum_{i,j=1}^{n}\oint_{|z-x_i|=\epsilon }\oint_{|w-x_j|=\epsilon}\left(\sum_{i=1}^n\frac{z}{n(z-x_i)}+\frac{z H(z)}{1-\kappa}\right)^k\\
\times\left(\sum_{i=1}^n\frac{w}{n(w-x_i)}+\frac{w H(w)}{1-\kappa}\right)^l
 G(z,w)dz dw\label{ct0}
\end{eqnarray}
where $\epsilon>0$ is sufficiently small such that for each $1\leq i\leq n$ the disk centered at $x_i$ with radius $\epsilon$ contains exactly one singularity $x_i$ of the integrand.
\end{lemma}

\begin{proof}Note that
\begin{eqnarray*}
\sum_{r=0}^{l}{{l}\choose{r}}(l-r)=l\sum_{r=0}^{l-1}{{l-1}\choose {r}}
\end{eqnarray*}
By the computations above, The contribution of the terms when $\{a_1,\ldots,a_{q+1}\}\cap\{b_1,\ldots,b_{r+1}\}=\emptyset$ to $\mathcal{G}_{(k,l,1-\kappa)}(X_{1-\kappa})$, as $N\rightarrow\infty$, is asymptotically
\begin{eqnarray*}
I:&=&k\sum_{q=0}^{k-1}\sum_{\{a_1,\ldots,a_{q+1}\}\subset\{1,2,\ldots,N-\lfloor\kappa N \rfloor\}}{{k-1}\choose{q}}(q+1)!\\
&&\mathrm{Sym}_{a_1,\ldots,a_{q+1}}\left[\frac{u_{a_1}^k([H_{a_1}(X,Y,\kappa)]^{k-1-q} N^{k-1-q})}{(u_{a_1}-u_{a_2})\ldots (u_{a_1}-u_{a_{q+1}})}\right.\\
&&l\sum_{r=0}^{l-1}\sum_{\{b_1,\ldots,b_{r+1}\}\subset\{1,2,\ldots,N-\lfloor\kappa N \rfloor,\ \{b_1,\ldots,b_{r+1}\}\cap \{a_1,\ldots,a_{q+1}\}=\emptyset\}}{l-1\choose r}(r+1)!\\
&\times &\left.\left.\mathrm{Sym}_{b_1,\ldots,b_{r+1}}\frac{u_{b_1}^l(H_{b_1}(X,Y,\kappa))^{l-r-1}N^{l-r-1}G(x_{a_1},x_{b_1})}{(u_{b_1}-u_{b_2})\ldots(u_{b_1}-u_{b_{r+1}})}\right]\right|_{(u_1,\ldots,u_{N-\lfloor \kappa N\rfloor})=X_{1-\kappa}}.
\end{eqnarray*}
We consider the equivalence relation on $\{a_1,\ldots,a_{q+1}\}$ (resp.\ $\{b_1,\ldots,b_{r+1}\}$) such that for $1\leq i\leq j\leq q+1$, $a_i$ and $a_j$  (resp.\ for $1\leq i\leq j\leq r+1$, $b_i$ and $b_j$) are equivalent if and only if $(i\mod n)=(j\mod n)$. Let $A_1,\ldots, A_h$ (resp.\ $B_1,\ldots,B_g$) be all the equivalence classes in $\{a_1,\ldots,a_{q+1}\}$ (resp.\ $\{b_1,\ldots,b_{r+1}\}$) under such an equivalence relation, where $h,g$ are positive integers satisfying $1\leq h\leq q+1$, $1\leq g\leq r+1$. For $1\leq i\leq h$ and $1\leq j\leq g$, let 
\begin{eqnarray*}
C_i=\{a_1,\ldots,a_{q+1}\}\setminus A_i;\qquad D_j=\{b_1,\ldots,b_{r+1}\}\setminus B_j.
\end{eqnarray*}
Then we have
\begin{eqnarray*}
I&=&k\sum_{q=0}^{k-1}\sum_{\{a_1,\ldots,a_{q+1}\}\subset\{1,2,\ldots,N-\lfloor\kappa N \rfloor\}}{{k-1}\choose{q}}(q+1)!\sum_{i=1}^{g}\frac{|A_i|}{(q+1)}\\
&&\mathrm{Sym}_{A_i}\left[\frac{u_{a_i}^k([H_{a_i}(X,Y,\kappa)]^{k-1-q} N^{k-1-q})}{\prod_{a_s\in C_i}(u_{a_i}-u_{a_s})}\right.\frac{1}{\prod_{a_t\in A_i\setminus\{a_i\}}(u_{a_i}-u_{a_t})}\\
&&l\sum_{r=0}^{l-1}\sum_{\{b_1,\ldots,b_{r+1}\}\subset\{1,2,\ldots,N-\lfloor\kappa N \rfloor,\ \{b_1,\ldots,b_{r+1}\}\cap \{a_1,\ldots,a_{q+1}\}=\emptyset\}}{l-1\choose r}(r+1)!\sum_{j=1}^{h}\frac{|B_j|}{(r+1)}\\
&\times &\left.\left.\mathrm{Sym}_{B_j}\left(\frac{u_{b_j}^l(H_{b_j}(X,Y,\kappa))^{l-r-1}N^{l-r-1}G(x_{a_i},x_{b_j})}{\prod_{b_v\in D_j}(u_{b_j}-u_{b_v})}\frac{1}{\prod_{b_w\in B_j\setminus\{b_j\}}(u_{b_j}-u_{b_w})}\right)\right]\right|_{(u_1,\ldots,u_{N-\lfloor \kappa N\rfloor})=X_{1-\kappa}}.
\end{eqnarray*}
By Lemma \ref{l552}, we have
\begin{eqnarray*}
I&=&k\sum_{q=0}^{k-1}\sum_{\{a_1,\ldots,a_{q+1}\}\subset\{1,2,\ldots,N-\lfloor\kappa N \rfloor\}}{{k-1}\choose{q}}(q+1)!\sum_{i=1}^{g}\frac{|A_i|}{(q+1)}\\
&&\frac{1}{|A_i|!}\frac{\partial^{|A_i|-1}}{\partial u_{a_i}^{|A_i|-1}}\left[\frac{u_{a_i}^k([H_{a_i}(X,Y,\kappa)]^{k-1-q} N^{k-1-q})}{\prod_{a_s\in C_i}(u_{a_i}-u_{a_s})}\right.\\
&&l\sum_{r=0}^{l-1}\sum_{\{b_1,\ldots,b_{r+1}\}\subset\{1,2,\ldots,N-\lfloor\kappa N \rfloor,\ \{b_1,\ldots,b_{r+1}\}\cap \{a_1,\ldots,a_{q+1}\}=\emptyset\}}{l-1\choose r}(r+1)!\sum_{j=1}^{h}\frac{|B_j|}{(r+1)}\\
&\times &\frac{1}{|B_j|!}\frac{\partial^{|B_j|-1}}{\partial u_{b_j}^{|B_j|-1}}\left.\left.\left(\frac{u_{b_j}^l(H_{b_j}(X,Y,\kappa))^{l-r-1}N^{l-r-1}G(x_{a_i},x_{b_j})}{\prod_{b_v\in D_j}(u_{b_j}-u_{b_v})}\right)\right]\right|_{(u_1,\ldots,u_{N-\lfloor \kappa N\rfloor})=X_{1-\kappa}}\\
&=&\sum_{q=0}^{k-1}\sum_{\{a_1,\ldots,a_{q+1}\}\subset\{1,2,\ldots,N-\lfloor\kappa N \rfloor\}}\sum_{i=1}^{g}\frac{k!}{(k-1-q)!}\\
&&\frac{1}{(|A_i|-1)!}\frac{\partial^{|A_i|-1}}{\partial u_{a_i}^{|A_i|-1}}\left[\frac{u_{a_i}^k([H_{a_i}(X,Y,\kappa)]^{k-1-q} N^{k-1-q})}{\prod_{a_s\in C_i}(u_{a_i}-u_{a_s})}\right.\\
&&\sum_{r=0}^{l-1}\sum_{\{b_1,\ldots,b_{r+1}\}\subset\{1,2,\ldots,N-\lfloor\kappa N \rfloor,\ \{b_1,\ldots,b_{r+1}\}\cap \{a_1,\ldots,a_{q+1}\}=\emptyset\}}\frac{l!}{(l-1-r)!}\sum_{j=1}^{h}\\
&\times &\frac{1}{(|B_j|-1)!}\frac{\partial^{|B_j|-1}}{\partial u_{b_j}^{|B_j|-1}}\left.\left.\left(\frac{u_{b_j}^l(H_{b_j}(X,Y,\kappa))^{l-r-1}N^{l-r-1}G(x_{a_i},x_{b_j})}{\prod_{b_v\in D_j}(u_{b_j}-u_{b_v})}\right)\right]\right|_{(u_1,\ldots,u_{N-\lfloor \kappa N\rfloor})=X_{1-\kappa}}
\end{eqnarray*}

Using the residue theorem, we obtain
\begin{eqnarray*}
I&=&\sum_{q=0}^{k-1}\sum_{\{a_1,\ldots,a_{q+1}\}\subset\{1,2,\ldots,N-\lfloor\kappa N \rfloor\}}\sum_{i=1}^{g}\frac{k!}{(k-1-q)!}\\
&&\mathrm{Res}_{z=u_{a_i}}\left[\frac{z^k([H(z)]^{k-1-q} N^{k-1-q})}{(z-u_{a_i})^{|A_i|}\prod_{a_s\in C_i}(z-u_{a_s})}\right.\\
&&\sum_{r=0}^{l-1}\sum_{\{b_1,\ldots,b_{r+1}\}\subset\{1,2,\ldots,N-\lfloor\kappa N \rfloor,\ \{b_1,\ldots,b_{r+1}\}\cap \{a_1,\ldots,a_{q+1}\}=\emptyset\}}\frac{l!}{(l-1-r)!}\sum_{j=1}^{h}\\
&\times &\mathrm{Res}_{w=u_{b_j}}\left.\left.\left(\frac{w^l(H(w))^{l-r-1}N^{l-r-1}G(z,w)}{(w-u_{b_j})^{|B_j|}\prod_{b_v\in D_j}(w-u_{b_v})}\right)\right]\right|_{(u_1,\ldots,u_{N-\lfloor \kappa N\rfloor})=X_{1-\kappa}}\\
&\approx&N^{k+l}\sum_{s=1}^{n}
\mathrm{Res}_{z=x_s}\left[\left(\frac{1}{N}\sum_{i=\lfloor \kappa N\rfloor+1}^{N}\frac{z}{z-x_i}+zH(z)\right)^k\right.\\
&&\times\left.\left[\sum_{j=1}^{n}\mathrm{Res}_{w=x_j}\left(\left(\frac{1}{N}\sum_{i=\lfloor \kappa N\rfloor+1}^{N}\frac{w}{w-x_i}+wH(w)\right)^lG(z,w)\right)\right]\right]\\
&\approx&\frac{N^{k+l}(1-\kappa)^{k+l}}{(2\pi \mathbf{i})^2}\sum_{i=1}^{n}\sum_{j=1}^{n}
\oint_{\left|z-x_i\right|=\epsilon}\left(\frac{1}{n}\sum_{p=1}^{n}\frac{z}{z-x_p}+\frac{zH(z)}{1-\kappa}\right)\\
&\times &\oint_{\left|w-x_j\right|=\epsilon}\left(\frac{1}{n}\sum_{q=1}^{n}\frac{w}{w-x_q}+\frac{wH(w)}{1-\kappa}\right)G(z,w)dwdz
\end{eqnarray*}

Then the lemma follows.

\end{proof}

Now we consider the case when
\begin{eqnarray}
|\{a_1,\ldots,a_{q+1}\}\cap\{b_1,\ldots,b_{r+1}\}|=1.\label{io}
\end{eqnarray}
Without loss of generality we suppose that $a_1=b_1$, and all the other indices are distinct. Then we have
\begin{eqnarray*}
&&\mathrm{Sym}_{a_1,\ldots,a_{q+1}}\left(\frac{u_{a_1}^k(\partial_{a_1}[\log S_{\rho_{N-\lfloor\kappa N\rfloor},(x_{\lfloor\kappa N\rfloor+1},\ldots,x_N)}]^{k-1-q})}{(u_{a_1}-u_{a_2})\ldots (u_{a_1}-u_{a_{q+1}})}\right.\\
&\times&\partial_{a_1}\left[\sum_{r=0}^{l}\sum_{\{b_2,\ldots,b_{r+1}\}\subset\{1,2,\ldots,N-\lfloor\kappa N \rfloor\}}\left(\begin{array}{c}l\\r\end{array}\right)(r+1)!\right.\\
&\times &\left.\left.\left.\mathrm{Sym}_{a_1,b_2,\ldots,b_{r+1}}\frac{u_{a_1}^l(\partial_{a_1}[\log S_{\rho_{N-\lfloor\kappa N\rfloor},(x_{\lfloor\kappa N\rfloor+1},\ldots,x_N)}])^{l-r}}{(u_{a_1}-u_{b_2})\ldots(u_{a_1}-u_{b_{r+1}})}\right]\right)\right|_{(u_1,\ldots,u_{N-\lfloor \kappa N\rfloor})=(x_{\lfloor\kappa N\rfloor+1},\ldots,x_N)}\\
&\approx&\mathrm{Sym}_{a_1,\ldots,a_{q+1}}\left(\frac{u_{a_1}^k N^{k-1-q} H(u_{a_1})^{k-1-q}}{(1-\kappa)^{k-1-q}(u_{a_1}-u_{a_2})\ldots (u_{a_1}-u_{a_{q+1}})}\right.\\&&\times\partial_{a_1}\left[\sum_{r=0}^{l}\sum_{\{b_2,\ldots,b_{r+1}\}\subset\{1,2,\ldots,N-\lfloor\kappa N \rfloor\}}\left(\begin{array}{c}l\\r\end{array}\right)(r+1)!\right.\\
&&\times \left.\left.\left.\mathrm{Sym}_{a_1,b_2,\ldots,b_{r+1}}\frac{u_{a_1}^l N^{l-r}[H(u_{a_1})]^{l-r}}{(1-\kappa)^{l-r}(u_{a_1}-u_{b_2})\ldots (u_{a_1}-u_{b_{q+1}})}\right]\right)\right|_{(u_1,\ldots,u_{N-\lfloor \kappa N\rfloor})=(x_{\lfloor\kappa N\rfloor+1},\ldots,x_N)}
\end{eqnarray*}
The summation over indices when (\ref{io}) holds gives terms of order $N^{r+q+1}$. We can see that the contribution of the terms when (\ref{io}) holds to $\mathcal{G}_{(l,k,1-\kappa)}(X_{1-\kappa})$ as $N\rightarrow\infty$, is
\begin{eqnarray*}
\tilde{I}:&=&\tilde{I}_1+\tilde{I}_2;
\end{eqnarray*}
where
\begin{eqnarray*}
\tilde{I}_1&=&k\sum_{q=0}^{k-1}\sum_{\{a_1,\ldots,a_{q+1}\}\subset\{1,2,\ldots,N-\lfloor\kappa N \rfloor\}}{{k-1}\choose{q}}(q+1)!\sum_{i=1}^{g}\frac{|A_i|}{(q+1)}\\
&&\mathrm{Sym}_{A_i}\left\{\frac{u_{a_i}^k([H_{a_i}(X,Y,\kappa)]^{k-1-q} N^{k-1-q})}{\prod_{a_s\in C_i}(u_{a_i}-u_{a_s})}\right.\frac{1}{\prod_{a_t\in A_i\setminus\{a_i\}}(u_{a_i}-u_{a_t})}\\
&&\frac{\partial}{\partial u_{a_i}}\left[\sum_{r=0}^{l}\sum_{\{b_1,\ldots,\hat{b}_j,\ldots,b_{r+1}\}\subset\{1,2,\ldots,N-\lfloor\kappa N \rfloor,\ \{b_1,\ldots,\hat{b}_j,\ldots,b_{r+1}\}\cap \{a_1,\ldots,a_{q+1}\}=\emptyset,\ b_j=a_i\}}\sum_{j=1}^{h}\frac{|B_j|}{(r+1)}\mathbf{1}_{b_j\in B_j}{l\choose r}(r+1)!\right.\\
&\times &\left.\left.\left.\mathrm{Sym}_{B_j}\left(\frac{u_{b_j}^l(H_{b_j}(X,Y,\kappa))^{l-r}N^{l-r}}{\prod_{b_v\in D_j}(u_{b_j}-u_{b_v})}\frac{1}{\prod_{b_w\in B_j\setminus\{b_j\}}(u_{b_j}-u_{b_w})}\right)\right]\right\}\right|_{(u_1,\ldots,u_{N-\lfloor \kappa N\rfloor})=X_{1-\kappa}}\\
\end{eqnarray*}
and
\begin{eqnarray*}
\tilde{I}_2&=&k\sum_{q=0}^{k-1}\sum_{\{a_1,\ldots,a_{q+1}\}\subset\{1,2,\ldots,N-\lfloor\kappa N \rfloor\}}{{k-1}\choose{q}}(q+1)!\sum_{i=1}^{g}\frac{|A_i|}{(q+1)}\\
&&\mathrm{Sym}_{A_i}\left\{\frac{u_{a_i}^k([H_{a_i}(X,Y,\kappa)]^{k-1-q} N^{k-1-q})}{\prod_{a_s\in C_i}(u_{a_i}-u_{a_s})}\right.\frac{1}{\prod_{a_t\in A_i\setminus\{a_i\}}(u_{a_i}-u_{a_t})}\\
&&\frac{\partial}{\partial u_{a_i}}\left[\sum_{r=0}^{l}\sum_{\{b_1,\ldots,\hat{b}_s,\ldots,b_{r+1}\}\subset\{1,2,\ldots,N-\lfloor\kappa N \rfloor,\ \{b_1,\ldots,\hat{b}_s,\ldots,b_{r+1}\}\cap \{a_1,\ldots,a_{q+1}\}=\emptyset,\ b_s=a_i\}}\sum_{j=1}^{h}\frac{|B_j|}{(r+1)}\mathbf{1}_{b_s\notin B_j}\right.\\&&{l\choose r}(r+1)!
\left.\left.\left.\mathrm{Sym}_{B_j}\left(\frac{u_{b_j}^l(H_{b_j}(X,Y,\kappa))^{l-r}N^{l-r}}{\prod_{b_v\in D_j}(u_{b_j}-u_{b_v})}\frac{1}{\prod_{b_w\in B_j\setminus\{b_j\}}(u_{b_j}-u_{b_w})}\right)\right]\right\}\right|_{(u_1,\ldots,u_{N-\lfloor \kappa N\rfloor})=X_{1-\kappa}}\\
\end{eqnarray*}
Here we use $\mathbf{1}$ to denote the indicator function.
By Lemma \ref{l552}, we infer 
\begin{eqnarray*}
\tilde{I}_1&=&\sum_{q=0}^{k-1}\sum_{\{a_1,\ldots,a_{q+1}\}\subset\{1,2,\ldots,N-\lfloor\kappa N \rfloor\}}\frac{k!}{(k-1-q)!}\sum_{i=1}^{g}|A_i|\\
&&\frac{1}{|A_i|!}\frac{\partial^{|A_i|-1}}{\partial u_{a_i}^{|A_i|-1}}\left[\frac{u_{a_i}^k([H(u_{a_i})]^{k-1-q} N^{k-1-q})}{\prod_{a_s\in C_i}(u_{a_i}-u_{a_s})}\right.\\
&&\frac{\partial}{\partial u_{a_i}}\left(\sum_{r=0}^{l}\sum_{\{b_1,\ldots,\hat{b}_j,\ldots,b_{r+1}\}\subset\{1,2,\ldots,N-\lfloor\kappa N \rfloor,\ \{b_1,\ldots,\hat{b}_j,\ldots,b_{r+1}\}\cap \{a_1,\ldots,a_{q+1}\}=\emptyset,\ b_j=a_i\}}\frac{l!}{(l-r)!}\right.\\
&\times &\sum_{j=1}^{h}|B_j|\mathbf{1}_{b_j\in B_j}\left.\left.\left.\frac{1}{|B_j|!}\frac{\partial^{|B_j|-1}}{\partial u_{b_j}^{|B_j|-1}}\frac{u_{b_j}^l(H(u_{b_j}))^{l-r}N^{l-r}}{\prod_{b_v\in D_j}(u_{b_j}-u_{b_v})}\right)\right]\right|_{(u_1,\ldots,u_{N-\lfloor \kappa N\rfloor})=X_{1-\kappa}}.
\end{eqnarray*}
and
\begin{eqnarray*}
\tilde{I}_2&=&\sum_{q=0}^{k-1}\sum_{\{a_1,\ldots,a_{q+1}\}\subset\{1,2,\ldots,N-\lfloor\kappa N \rfloor\}}\frac{k!}{(k-1-q)!}\sum_{i=1}^{g}|A_i|\\
&&\frac{1}{|A_i|!}\frac{\partial^{|A_i|-1}}{\partial u_{a_i}^{|A_i|-1}}\left[\frac{u_{a_i}^k([H(u_{a_i})]^{k-1-q} N^{k-1-q})}{\prod_{a_s\in C_i}(u_{a_i}-u_{a_s})}\right.\\
&&\frac{\partial}{\partial u_{a_i}}\left(\sum_{r=0}^{l}\sum_{\{b_1,\ldots,\hat{b}_s,\ldots,b_{r+1}\}\subset\{1,2,\ldots,N-\lfloor\kappa N \rfloor,\ \{b_1,\ldots,\hat{b}_s,\ldots,b_{r+1}\}\cap \{a_1,\ldots,a_{q+1}\}=\emptyset,\ b_s=a_i\}}\frac{l!}{(l-r)!}\right.\\
&\times &\sum_{j=1}^{h}|B_j|\mathbf{1}_{b_s\notin B_j}\left.\left.\left.\frac{1}{|B_j|!}\frac{\partial^{|B_j|-1}}{\partial u_{b_j}^{|B_j|-1}}\frac{u_{b_j}^l(H(u_{b_j}))^{l-r}N^{l-r}}{\prod_{b_v\in D_j}(u_{b_j}-u_{b_v})}\right)\right]\right|_{(u_1,\ldots,u_{N-\lfloor \kappa N\rfloor})=X_{1-\kappa}}.
\end{eqnarray*}
By the residue theorem, we have 
\begin{eqnarray*}
\tilde{I}_1&=&\sum_{q=0}^{k-1}\sum_{\{a_1,\ldots,a_{q+1}\}\subset\{1,2,\ldots,N-\lfloor\kappa N \rfloor\}}\frac{k!}{(k-1-q)!}\sum_{i=1}^{g}|A_i|\\
&&\frac{1}{|A_i|!}\frac{\partial^{|A_i|-1}}{\partial u_{a_i}^{|A_i|-1}}\left[\frac{u_{a_i}^k([H(u_{a_i})]^{k-1-q} N^{k-1-q})}{\prod_{a_s\in C_i}(u_{a_i}-u_{a_s})}\right.\\
&&\left(\sum_{r=0}^{l}\sum_{\{b_1,\ldots,\hat{b}_j,\ldots,b_{r+1}\}\subset\{1,2,\ldots,N-\lfloor\kappa N \rfloor,\ \{b_1,\ldots,\hat{b}_j,\ldots,b_{r+1}\}\cap \{a_1,\ldots,a_{q+1}\}=\emptyset,\ b_j=a_i\}}\frac{l!}{(l-r)!}\sum_{j=1}^{h}|B_j|\mathbf{1}_{b_j\in B_j}\right.\\
&\times &\left.\left.\left.\frac{1}{|B_j|!}\frac{\partial^{|B_j|}}{\partial u_{a_i}^{|B_j|}}\frac{u_{a_i}^l(H(u_{a_i}))^{l-r}N^{l-r}}{\prod_{b_v\in D_j}(u_{a_i}-u_{b_v})}\right)\right]\right|_{(u_1,\ldots,u_{N-\lfloor \kappa N\rfloor})=X_{1-\kappa}}\\
&=&\sum_{q=0}^{k-1}\sum_{\{a_1,\ldots,a_{q+1}\}\subset\{1,2,\ldots,N-\lfloor\kappa N \rfloor\}}\frac{k!}{(k-1-q)!}\sum_{i=1}^{g}\\
&&\mathrm{Res}_{z=u_{a_i}} \left[\frac{z^k([H(z)]^{k-1-q} N^{k-1-q})}{\prod_{a_s\in C_i}(z-u_{a_s})}\frac{1}{(z-u_{a_i})^{|A_i|}}\right.\\
&&\left(\sum_{r=0}^{l}\sum_{\{b_1,\ldots,\hat{b}_j,\ldots,b_{r+1}\}\subset\{1,2,\ldots,N-\lfloor\kappa N \rfloor,\ \{b_1,\ldots,\hat{b}_j,\ldots,b_{r+1}\}\cap \{a_1,\ldots,a_{q+1}\}=\emptyset,\ b_j=a_i\}}\sum_{j=1}^{h}|B_j|\mathbf{1}_{b_j\in B_j}\frac{l!}{(l-r)!}\right.\\
&\times &\left.\left.\left.\mathrm{Res}_{w=u_{a_i}}\left(\frac{w^l(H(w))^{l-r}N^{l-r}}{\prod_{b_v\in D_j}(w-u_{b_v})}\frac{1}{(w-u_{a_i})^{|B_j|-1}}\frac{1}{(w-z)^2}\right)\right)\right]\right|_{(u_1,\ldots,u_{N-\lfloor \kappa N\rfloor})=X_{1-\kappa}}\\
&=&\frac{1}{(2\pi \mathbf{i})^2}\sum_{j=1}^{n}
\oint _{|z-x_j|=\epsilon}\left(\sum_{i=1}^{N-\lfloor \kappa N\rfloor}\frac{z}{z-u_i}+z N H(z)\right)^k
\oint_{|w-x_j|=\epsilon}\left(\sum_{i=1}^{N-\lfloor \kappa N\rfloor}\frac{w}{w-u_i}+w N H(w)\right)^l\\
&&\frac{1}{(w-z)^2}dwdz
\end{eqnarray*}
We also have
\begin{eqnarray*}
\tilde{I}_2&=&\sum_{q=0}^{k-1}\sum_{\{a_1,\ldots,a_{q+1}\}\subset\{1,2,\ldots,N-\lfloor\kappa N \rfloor\}}\frac{k!}{(k-1-q)!}\sum_{i=1}^{g}|A_i|\\
&&\frac{1}{|A_i|!}\frac{\partial^{|A_i|-1}}{\partial u_{a_i}^{|A_i|-1}}\left[\frac{u_{a_i}^k([H(u_{a_i})]^{k-1-q} N^{k-1-q})}{\prod_{a_s\in C_i}(u_{a_i}-u_{a_s})}\right.\\
&&\left(\sum_{r=0}^{l}\sum_{\{b_1,\ldots,\hat{b}_s,\ldots,b_{r+1}\}\subset\{1,2,\ldots,N-\lfloor\kappa N \rfloor,\ \{b_1,\ldots,\hat{b}_s,\ldots,b_{r+1}\}\cap \{a_1,\ldots,a_{q+1}\}=\emptyset,\ b_s=a_i\}}\frac{l!}{(l-r)!}\right.\\
&\times &\sum_{j=1}^{h}|B_j|\mathbf{1}_{b_s\notin B_j}\left.\left.\left.\frac{1}{|B_j|!}\frac{\partial^{|B_j|-1}}{\partial u_{b_j}^{|B_j|-1}}\frac{u_{b_j}^l(H(u_{b_j}))^{l-r}N^{l-r}}{\prod_{b_v\in D_j\setminus\{b_s\}}(u_{b_j}-u_{b_v})}\frac{1}{(u_{b_j}-u_{b_s})^2}\right)\right]\right|_{(u_1,\ldots,u_{N-\lfloor \kappa N\rfloor})=X_{1-\kappa}}.
\end{eqnarray*}
where $b_j\in B_j$.
Using the residue theorem, we can also infer 
\begin{eqnarray*}
\tilde{I}_2&=&\sum_{q=0}^{k-1}\sum_{\{a_1,\ldots,a_{q+1}\}\subset\{1,2,\ldots,N-\lfloor\kappa N \rfloor\}}\frac{k!}{(k-1-q)!}\sum_{i=1}^{g}\\
&&\mathrm{Res}_{z=u_{a_i}}\left[\frac{z^k([H(z)]^{k-1-q} N^{k-1-q})}{\prod_{a_s\in C_i}(z-u_{a_s})}\frac{1}{(z-u_{a_i})^{|A_i|}}\right.\\
&&\sum_{r=0}^{l}\sum_{\{b_2,\ldots,b_{r+1}\}\subset\{1,2,\ldots,N-\lfloor\kappa N \rfloor,\ \{b_2,\ldots,b_{r+1}\}\cap \{a_1,\ldots,a_{q+1}\}=\emptyset,\ b_1=a_i\}}\frac{l!}{(l-r)!}\sum_{j=1}^{h}\mathbf{1}_{b_1\notin B_j}\\
&\times &\left.\left.\mathrm{Res}_{w=u_{b_j}}\left(\frac{w^l(H(w))^{l-r}N^{l-r}}{\prod_{b_v\in D_j\setminus\{b_1\}}(w-u_{b_v})}\frac{1}{(w-u_{b_j})^{|B_j|}}\frac{1}{(w-z)^2}\right)\right]\right|_{(u_1,\ldots,u_{N-\lfloor \kappa N\rfloor})=X_{1-\kappa}}\\
&=&\frac{1}{(2\pi\mathbf{i})^2}\sum_{j\in\{1,2,\ldots,n\}}\sum_{p\in\{1,2,\ldots,n\},p\neq j}
\oint _{|z-x_j|=\epsilon}\left(\sum_{i=1}^{N-\lfloor \kappa N\rfloor}\frac{z}{z-u_i}+z N H(z)\right)^k\\
&&\oint_{|w-x_p|=\epsilon}\left(\sum_{i=1}^{N-\lfloor \kappa N\rfloor}\frac{w}{w-u_i}+w N H(w)\right)^l
\left.\frac{1}{(w-z)^2}dwdz\right|_{(u_1,\ldots,u_{N-\lfloor \kappa N\rfloor})=X_{1-\kappa}}
\end{eqnarray*}
Therefore we have
\begin{eqnarray}
&&\lim_{N\rightarrow\infty}\frac{\tilde{I}}{N^{k+l}}\label{ct1}\\
&=&\frac{(1-\kappa)^{k+l}}{(2\pi\mathbf{i})^2}\sum_{i=1}^{n}\sum_{j=1}^{n}\oint_{|z-x_i|=\epsilon}\oint_{|w-x_j|=\epsilon}\left(\sum_{i=1}^n\frac{z}{n(z-x_i)}+ \frac{zH(z)}{1-\kappa}\right)^k\notag\\
&&\times\left(\sum_{i=1}^n\frac{w}{n(w-x_i)}+ \frac{wH(w)}{1-\kappa}\right)^l\frac{1}{(z-w)^2}dz dw\notag
\end{eqnarray}

Finally let us consider the case when
\begin{eqnarray*}
|\{a_1,\ldots,a_{q+1}\}\cap\{b_1,\ldots,b_{r+1}\}|\geq 2.
\end{eqnarray*}
By Lemma \ref{l33}, we can see that the contribution of these terms to $\mathcal{G}_{(l,k)}(X_{1-\kappa})$ has $N$-degree strictly less than $k+l$.

Therefore we have the following proposition.
\begin{proposition}Assume the assumption of Lemma \ref{l33} holds. Then
\begin{eqnarray*}
&&\lim_{N\rightarrow\infty}\frac{\mathrm{cov}(p_k^{((1-\kappa)N)},p_{l}^{((1-\kappa)N)})}{N^{k+l}}\\
&=&\frac{(1-\kappa)^{k+l}}{(2\pi\mathbf{i})^2}\sum_{i=1}^{n}\sum_{j=1}^{n}\oint_{|z-x_i|=\epsilon}\oint_{|w-x_j|=\epsilon}\left(\sum_{i=1}^{n}\frac{z}{n(z-x_i)}+ \frac{zH(z)}{1-\kappa}\right)^k\\
&&\times\left(\sum_{j=1}^{n}\frac{z}{n(z-x_j)}+ \frac{zH(z)}{1-\kappa}\right)^l Q(z,w)dzdw
\end{eqnarray*}
where
\begin{eqnarray*}
Q(z,w)=G(z,w)+\frac{1}{(z-w)^2}.
\end{eqnarray*}
\end{proposition}

\begin{proof}By (\ref{cov}), $\lim_{N\rightarrow\infty}\frac{\mathrm{cov}(p_k^{((1-\kappa)N)},p_{l}^{((1-\kappa)N)})}{N^{k+l}}$ should be the sum of (\ref{ct0}) and (\ref{ct1}), divided by $N^{k+l}$. Then the proposition follows.
\end{proof}

\subsection{Multi-Level Correlations}

Define a mapping $\phi: \{1,\ldots,2N+1\}\rightarrow \{\mu^{(i)},\nu^{(j)}:i,j\in \{1,2,\ldots,N\}\}$ as follows
\begin{eqnarray*}
\phi(n)=\left\{\begin{array}{cc}\mu^{\left(\frac{n-1}{2}\right)}&\mathrm{if}\ n\ \mathrm{is\ odd}\\\nu^{\left(\frac{n}{2}\right)}&\mathrm{if}\ n\ \mathrm{is\ even}\end{array}\right.
\end{eqnarray*}

For $i\in \{1,2,\ldots,N\}$, define
\begin{eqnarray*}
C_i=(x_i,x_{i+1},\ldots,x_N)\in \RR^{N-i+1};
\end{eqnarray*}
and for $i\in I_2$, let
\begin{eqnarray*}
B_i=y_i C_i=(y_ix_i,y_ix_{i+1},\ldots,y_ix_N)\in \RR^{N-i+1}.
\end{eqnarray*}

Let $1\leq n_1\leq n_2\leq\ldots \leq n_s=2N+1$ be positive row numbers of the square-hexagon lattice, counting from the top.
For $1\leq i\leq s$, let $\rho_{n_i}$ be the induced probability measure on dimer configurations of the $n_i$th row. Then the induced  probability measure on the state space
\begin{eqnarray*}
\GT_{\lfloor\frac{n_1}{2}\rfloor}\times\GT_{\lfloor\frac {n_2}{2}\rfloor}\times \ldots\times \GT_{\lfloor \frac{n_s}{2}\rfloor}
\end{eqnarray*}
by the measure proportional to the product of weights of present edges of dimer configurations on the square-hexagon lattice $\mathcal{R}(\Omega,\check{c})$ can be expressed as follows:
\begin{eqnarray}
\mathrm{Prob}\left(\phi(n_s),\ldots,\phi(n_1)\right)=\rho_{n_s}\left((\phi(n_s)\right)\prod_{i=2}^{k}\mathrm{Prob}\left[\phi(n_{i-1})|\phi(n_i)\right]\label{prb}
\end{eqnarray}
Here $\mathrm{Prob}\left[\phi(n_{i-1})|\phi(n_i)\right]$ is the probability of $\phi(n_{i-1})$ conditional on $\phi(n_i)$. In particular, we have
\begin{eqnarray}
\mathrm{Prob}[\mu^{(t-1)}|\nu^{(t)}]&=&\mathrm{pr}_{C_{N-t+1}}(\nu^{(t)}\rightarrow\mu^{(t-1)})\label{cc1}\\
\mathrm{Prob}[\nu^{(t)}|\mu^{(t)}]&=&\begin{cases}\mathrm{st}_{B_{N-t+1}}(\mu^{(t)}\rightarrow\nu^{(t)}),\ \mathrm{If}\ N-t+1\in I_2\\ \mathbf{1}_{\nu^{(t)}=\mu^{(t)}},\ \mathrm{otherwise}\end{cases}\label{cc2}
\end{eqnarray}
where $\mathbf{1}_{\nu^{(t)}=\mu^{(t)}}$ is the indicator of $\nu^{(t)}=\mu^{(t)}$; $x_i,y_j$ are edge weights; and
\begin{eqnarray*}
\mathrm{pr}_{C_{N-t+1}}(\nu^{(t)}\rightarrow\mu^{(t-1)})=\begin{cases}x_{N-t+1}^{|\nu^{(t)}|-|\mu^{(t-1)}|}\frac{s_{\mu^{(t-1)}}(x_{N-t+2},\ldots,x_{N})}{s_{\nu^{(t)}}(x_{N-t+1},\ldots,x_N)},\ \mathrm{If}\ \mu^{(t-1)}\prec \nu^{(t)}\\0,\ \mathrm{otherwise}\end{cases}\\
\mathrm{st}_{B_{N-t+1}}(\mu^{(t)}\rightarrow\nu^{(t)})=\begin{cases}\frac{y_{N-t+1}^{|\nu(t)|-|\mu(t)|}}{\prod_{j=N-t+1}^{N}(1+y_{N-t+1}x_j)}\frac{s_{\nu^{(t)}}(x_{N-t+1},\ldots,x_N)}{s_{\mu^{(t)}}(x_{N-t+1},\ldots,x_N)},\ \mathrm{If}\ \mu^{(t)}\subset \nu^{(t)}\\0,\ \mathrm{otherwise}\end{cases};
\end{eqnarray*}
see Section 2.4 of \cite{BL17}.

\begin{definition}\label{df59}(Multi-dimensional Schur generating function) Let $1\leq N_1\leq N_2\leq\ldots\leq N_s$ be positive integers. For a probability measure $\rho$ on $\prod_{t=1}^{s}\GT_{N_t}$, we define the $s$-dimensional Schur generating function with respect to $X=(x_1,\ldots,x_N)$ by
\begin{eqnarray*}
&&\mathcal{S}_{\rho,X}(u_{1,1},\ldots,u_{N_1,1};\ldots;u_{1,s},\ldots,u_{N_s,s})\\
&=&\sum_{\lambda^1\in \GT_{N_1},\ldots,\lambda^s\in\GT_{N_s}}\rho(\lambda^1,\ldots,\lambda^s)\prod_{t=1}^{s}\frac{s_{\lambda^t}(u_{1,t},\ldots,u_{\lfloor\frac{n_t}{2} \rfloor,t})}{s_{\lambda^t}(x_{N-\lfloor\frac{n_s}{2}\rfloor+1},\ldots,x_N)}
\end{eqnarray*}
\end{definition}

The multi-dimensional Schur generating function with respect to $(1,\ldots,1)$ was defined in \cite{BG17}.

\begin{lemma}\label{l59}Let $m_1,\ldots,m_k$ be positive integers. Let $1\leq n_1\leq n_2\leq\ldots \leq n_k\leq 2N+1$ be positive row numbers of the square-hexagon lattice, counting from the top. Assume that $(\phi(n_s),\ldots,\phi(n_1))$ has distribution $\rho$ defined by (\ref{prb}).
For $1\leq s\leq k$, let $\mathcal{D}_{l}^{(s)}$ be the $l$-th order differential operator defined by
\begin{eqnarray}
\mathcal{D}_{l}^{(n_s)}=\frac{1}{V_{\lfloor \frac{n_s}{2}\rfloor}}\left(\sum_{i=1}^{\lfloor \frac{n_s}{2}\rfloor}\left(u_{i,s}\frac{\partial}{\partial u_{i,s}}\right)^l\right)V_{\lfloor \frac{n_s}{2}\rfloor}\label{dd}
\end{eqnarray}
where $V_{\lfloor \frac{n_s}{2}\rfloor}$ is the Vandermonde determinant on $\lfloor \frac{n_s}{2}\rfloor$ variables $u_{1,s},\ldots,u_{\lfloor \frac{m}{2}\rfloor,s}$.
Let 
\begin{eqnarray*}
\ol{X}_{k}=\left(x_{N-\lfloor \frac{n_k}{2}\rfloor+1},\ldots,x_{N}\right).
\end{eqnarray*}
 Then
\begin{eqnarray*}
&&\left.\mathcal{D}_{m_1}^{(n_1)}\mathcal{D}_{m_2}^{(n_2)}\ldots\mathcal{D}_{m_k}^{(n_k)}\mathcal{S}_{\rho,X}(u_{1,1},\ldots,u_{\lfloor \frac{n_1}{2}\rfloor,1};\ldots;u_{1,k},\ldots,u_{\lfloor \frac{n_k}{2}\rfloor,k})\right|_{(u_{1,s},\ldots,u_{\lfloor \frac{n_s}{2}\rfloor,s})=\ol{X}_s,\ \forall 1\leq s\leq k}\\
&=&\mathbf{E}\left(\sum_{i=1}^{\lfloor \frac{n_1}{2}\rfloor}(\phi(n_1)_{i_1}+\lfloor \frac{n_1}{2}\rfloor-i_1)^{m_1}\sum_{i_2=1}^{\lfloor\frac {n_2}{2}\rfloor}(\phi(n_2)_{i_2}+\lfloor\frac{n_2}{2}\rfloor-i_2)^{m_2}\cdot\ldots\cdot\sum_{i_k=1}^{\lfloor\frac{n_k}{2}\rfloor}(\phi(n_k)_{i_k}+\lfloor \frac{n_k}{2}\rfloor-i_k)^{m_k}\right).
\end{eqnarray*}
where $\mathbf{E}$ is the expectation with respect to the probability measure defined by (\ref{prb}), and $\mathcal{S}_{\rho,X}$ is the multi-dimensional Schur generating function as defined in Definition \ref{df59}.
\end{lemma}
\begin{proof}The theorem follows from explicit computations.
\end{proof}

\begin{lemma}\label{ll511}Suppose the assumptions in Lemma \ref{l59} hold. For $1\leq s\leq k$, let
\begin{eqnarray*}
t_s=N-\lfloor\frac{n_s}{2} \rfloor.
\end{eqnarray*}
For $1\leq s\leq k-1$, let
\begin{eqnarray*}
\tilde{\mathcal{D}}_{m_s,k}^{(n_s)}:=\frac{1}{\hat{V}_{\lfloor \frac{n_s}{2}\rfloor}}\left(\sum_{i=t_s+1}^{N}\left(u_{i}\frac{\partial}{\partial u_{i}}\right)^{m_s}\right)\hat{V}_{\lfloor \frac{n_s}{2}\rfloor}\prod_{i\in \{t_{s+1}+1,\ldots,t_s\}\cap I_2}\prod_{j= t_s+1}^{N}\left(\frac{1+y_{\ol{i}}u_j}{1+y_{\ol{i}}x_{\ol{j}}}\right),
\end{eqnarray*}
and let 
\begin{eqnarray*}
\tilde{\mathcal{D}}_{m_k,k}^{(n_k)}:=\frac{1}{\hat{V}_{\lfloor \frac{n_k}{2}\rfloor}}\left(\sum_{i=t_k+1}^{N}\left(u_{i}\frac{\partial}{\partial u_{i}}\right)^{m_k}\right)\hat{V}_{\lfloor \frac{n_k}{2}\rfloor};
\end{eqnarray*}
where $\hat{V}_{\lfloor \frac{n_s}{2}\rfloor}$ is the Vandermonde determinant on $\lfloor \frac{n_s}{2}\rfloor$ variables $u_{t_s+1},\ldots,u_{N}$. Then
\begin{eqnarray*}
&&\left.\mathcal{D}_{m_1}^{(n_1)}\mathcal{D}_{m_2}^{(n_2)}\ldots\mathcal{D}_{m_k}^{(n_k)}\mathcal{S}_{\rho,X}(u_{1,1},\ldots,u_{\lfloor \frac{n_1}{2}\rfloor,1};\ldots;u_{1,k},\ldots,u_{\lfloor \frac{n_k}{2}\rfloor,k})\right|_{(u_{1,s},\ldots,u_{\lfloor \frac{n_s}{2}\rfloor,s})=\ol{X}_s,\ \forall 1\leq s\leq k}\\
&=&\tilde{\mathcal{D}}_{m_1,k}^{(n_1)}\tilde{\mathcal{D}}_{m_2,k}^{(n_2)}\ldots\tilde{\mathcal{D}}_{m_k,k}^{(n_k)}\left\{
\left.\mathcal{S}_{\rho_{\lfloor\frac{n_k}{2} \rfloor},X}(u_{N-\lfloor\frac{n_k}{2} \rfloor+1},\ldots,u_N)\right\}\right|_{(u_{1},\ldots,u_N)=(x_1,\ldots,x_N)}
\end{eqnarray*}
where $\mathcal{S}_{\rho_{\lfloor\frac{n_k}{2} \rfloor},X}(u_1,\ldots,u_N)$ is the one-dimensional Schur generating function defined as in Definition \ref{df33}, and $\rho_{\frac{n_k}{2}}$ is a probability measure on $\GT_{\lfloor \frac{n_k}{2}\rfloor}^+$ defined as in Lemma \ref{l33}.
\end{lemma}
\begin{proof}We shall prove the lemma by induction on $k$. First of all, when $k=1$, the lemma obviously holds. Assume that the lemma holds when $k=l-1$, where $\geq 2$ is a positive integer. Then when $k=l$, we have
\begin{eqnarray*}
&&\left.\mathcal{D}_{m_1}^{(n_1)}\mathcal{D}_{m_2}^{(n_2)}\ldots\mathcal{D}_{m_l}^{(n_l)}\mathcal{S}_{\rho,X}(u_{1,1},\ldots,u_{\lfloor \frac{n_1}{2}\rfloor,1};\ldots;u_{1,l},\ldots,u_{\lfloor \frac{n_l}{2}\rfloor,l})\right|_{(u_{1,s},\ldots,u_{\lfloor \frac{n_s}{2}\rfloor,s})=\ol{X}_s,\ \forall 1\leq s\leq l}\\
&=&\sum_{\lambda^{l}\in \GT_{\lfloor \frac{n_l}{2}\rfloor}}\rho_{\lfloor\frac{n_l}{2} \rfloor}(\lambda^l)\mathcal{D}_{m_l}^{(n_l)}\left(\frac{s_{\lambda^l}(u_{1,l},\ldots,u_{\lfloor \frac{n_l}{2}\rfloor,l})}{s_{\lambda^l}(\ol{X}_l)}\right)\\
&&\left.\mathcal{D}_{m_1,l-1}^{(n_1)}\ldots\mathcal{D}_{m_{l-1},l-1}^{(n_{l-1})}\mathcal{S}_{\rho(\cdot|\lambda^{l}),X}(u_{1,1},\ldots,u_{\lfloor \frac{n_1}{2}\rfloor,1};\ldots;u_{1,{l-1}},\ldots,u_{\lfloor \frac{n_{l-1}}{2}\rfloor,{l-1}})\right|_{(u_{1,s},\ldots,u_{\lfloor \frac{n_s}{2}\rfloor,s})=\ol{X}_s,\ \forall 1\leq s\leq l}
\end{eqnarray*}
where $\rho(\cdot|\lambda^l)$ is a probability on $\prod_{s=1}^{l-1}\GT_{\lfloor\frac{n_s}{2} \rfloor}$ obtained from $\rho$ by conditional on the configuration $\lambda^l$ on $\GT_{\lfloor\frac{n_l}{2} \rfloor}$. By the induction hypothesis and Lemma \ref{lmm212}, we have
\begin{eqnarray*}
&&\left.\mathcal{D}_{m_1}^{(n_1)}\ldots\mathcal{D}_{m_{l-1}}^{(n_{l-1})}\mathcal{S}_{\rho(\cdot|\lambda^{l}),X}(u_{1,1},\ldots,u_{\lfloor \frac{n_1}{2}\rfloor,1};\ldots;u_{1,{l-1}},\ldots,u_{\lfloor \frac{n_{l-1}}{2}\rfloor,{l-1}})\right|_{(u_{1,s},\ldots,u_{\lfloor \frac{n_s}{2}\rfloor,s})=\ol{X}_s,\ \forall 1\leq s\leq l-1}\\
&=&\tilde{\mathcal{D}}_{m_1,l-1}^{(n_1)}\tilde{\mathcal{D}}_{m_2,l-1}^{(n_2)}\ldots\tilde{\mathcal{D}}_{m_{l-1},l-1}^{(n_{l-1})}\left\{\left.\mathcal{S}_{\rho_{\lfloor\frac{n_{l-1}}{2} \rfloor}(\cdot|\lambda^l),X}(u_{N-\lfloor\frac{n_{l-1}}{2} \rfloor+1},\ldots,u_N)\right\}\right|_{(u_{1},\ldots,u_N)=(x_1,\ldots,x_N)}\\
&=&\tilde{\mathcal{D}}_{m_1,l}^{(n_1)}\tilde{\mathcal{D}}_{m_2,l}^{(n_2)}\ldots\tilde{\mathcal{D}}_{m_{l-1},l-1}^{(n_{l-1})}\left\{\left[\prod_{i\in \{t_{l},t_{l}+1,\ldots,t_{l-1}\}\cap I_2}\prod_{j\in t_{l-1}+1}^{N}\left(\frac{1+y_{\ol{i}}u_j}{1+y_{\ol{i}}x_{\ol{j}}}\right)\right]\right.\\
&&\left.\left.\frac{s_{\lambda^l}(x_{N-\lfloor\frac{n_l}{2} \rfloor+1},\ldots,x_{N-\lfloor\frac{n_{l-1}}{2}\rfloor},u_{N-\lfloor\frac{n_{l-1}}{2} \rfloor+1},\ldots,u_N)}{s_{\lambda^l}(x_{N-\lfloor\frac{n_l}{2} \rfloor+1},\ldots,x_N)}\right\}\right|_{(u_{1},\ldots,u_N)=(x_1,\ldots,x_N)}\\
&=&\tilde{\mathcal{D}}_{m_1,l}^{(n_1)}\tilde{\mathcal{D}}_{m_2,l}^{(n_2)}\ldots\tilde{\mathcal{D}}_{m_{l-1},l}^{(n_{l-1})}\left\{\left.\frac{s_{\lambda^l}(x_{N-\lfloor\frac{n_l}{2} \rfloor+1},\ldots,x_{N-\lfloor\frac{n_{l-1}}{2}\rfloor},u_{N-\lfloor\frac{n_{l-1}}{2} \rfloor+1},\ldots,u_N)}{s_{\lambda^l}(x_{N-\lfloor\frac{n_l}{2} \rfloor+1},\ldots,x_N)}\right\}\right|_{(u_{1},\ldots,u_N)=(x_1,\ldots,x_N)}
\end{eqnarray*}
Note also that
\begin{eqnarray*}
&&\left.\mathcal{D}_{m_l}^{(n_l)}\left(\frac{s_{\lambda^l}(u_{1,l},\ldots,u_{\lfloor \frac{n_l}{2}\rfloor,l})}{s_{\lambda^l}(\ol{X}_l)}\right)\right|_{(u_{1,l},\ldots,u_{\lfloor \frac{n_l}{2}\rfloor,l})=\ol{X}_l}=\sum_{j=1}^{\lfloor\frac{n_l}{2} \rfloor}\left(\lambda_j^{l}+\lfloor \frac{n_l}{2}\rfloor-j\right)^{m_l}\\
&=&\tilde{\mathcal{D}}_{m_l,l}^{(n_l)}\left.\left\{\frac{s_{\lambda^l}(u_{N-\lfloor\frac{n_l}{2} \rfloor+1},\ldots,\ldots,u_N)}{s_{\lambda^l}(x_{N-\lfloor\frac{n_l}{2} \rfloor+1},\ldots,x_N)}\right\}\right|_{(u_{1},\ldots,u_N)=(x_1,\ldots,x_N)}
\end{eqnarray*}
Then the lemma follows.
\end{proof}

Let $m_1,\ldots,m_k$ and $n_1,\ldots,n_k$ be as in Lemma \ref{l59}.
For $1\leq s\leq k$, we introduce the notation
\begin{eqnarray*}
\left.E_{l,s}=\mathcal{F}_{(l,\frac{n_s}{2N}) }(u_{N-\lfloor\frac{n_s}{2} \rfloor}+1,\ldots,u_N)\right|_{(u_1,\ldots,u_N)=(x_1,\ldots,x_N)}.\\
\end{eqnarray*}
where $\mathcal{F}$ is defined by (\ref{flu}) and can be expressed as in (\ref{exf}).

Let $s_1<s_2$ be positive integers between $1$ and $k$. Define
\begin{eqnarray}
&&G_{s_1,s_2}(u_{N-\lfloor\frac{n_{s_2}}{2}\rfloor+1},\ldots,u_N)\label{g12}\\
&=&m_{s_1}\sum_{r=0}^{m_{s_1}-1}\left(\begin{array}{c}m_{s_1}-1\\r\end{array}\right)\sum_{\{b_1,\ldots,b_{r+1}\}\subset\{N-\lfloor\frac{n_{s_1}}{2} \rfloor+1,\ldots,N\}}(r+1)!\notag\\
&&\times\mathrm{Sym}_{b_1,\ldots,b_{r+1}}\frac{u_{b_1}^{m_{s_1}}\frac{\partial}{\partial u_{b_1}}\left[\mathcal{F}_{(m_{s_2},\frac{n_{s_2}}{2N})}\right]\left(\frac{\partial}{\partial u_{b_1}}\left[\log \mathcal{S}_{\rho_{\lfloor \frac{n_{s_1}}{2} \rfloor},X}\right]\right)^{m_{s_1}-1-r}}{(u_{b_1}-u_{b_2})\ldots(u_{b_1}-u_{b_{r+1}})}.\notag
\end{eqnarray}

\begin{lemma}\label{ll66}
\begin{enumerate}
\item the degree of $N$ in $\left.G_{s_1,s_2}\right|_{(u_1,\ldots,u_N)=(x_1,\ldots,x_N)}$ is at most $m_{s_1}+m_{s_2}$. Moreover, for any index $i$ the degree of $N$ in $\left.\frac{\partial}{\partial u_i}G_{s_1,s_2}\right|_{(u_1,\ldots,u_N)=(x_1,\ldots,x_N)}$ is less than $m_{s_1}+m_{s_2}$.
\item
\begin{eqnarray*}
&&\frac{1}{\hat{V}_{\lfloor\frac{n_{s_1}}{2}\rfloor} \mathcal{S}_{\rho_{\lfloor \frac{n_{s_1}}{2}\rfloor},X} }m_{s_1}\sum_{i\in\{N-\lfloor \frac{n_{s_1}}{2}\rfloor+1,\ldots,N\}}\left(u_i\frac{\partial}{\partial u_i}\left[\mathcal{F}_{(m_{s_2},\frac{n_{s_2}}{2N})}\right]\right)\\
&&\left.\left(u_i\frac{\partial}{\partial u_i}\right)^{m_{s_1}-1}\left[\hat{V}_{\lfloor\frac{n_{s_1}}{2}\rfloor}\mathcal{S}_{\rho_{\lfloor \frac{n_{s_1}}{2}\rfloor},X} \right]\right|_{(u_1,\ldots,u_N)=(x_1,\ldots,x_N)}\\
&=&\left.G_{s_1,s_2}\right|_{(u_1,\ldots,u_N)=(x_1,\ldots,x_N)}+R
\end{eqnarray*}
where the degree of $N$ in $R$ is less than $l+k$.
\end{enumerate}
\end{lemma}

\begin{proof}We first prove Part (1). By Lemma \ref{l33}, the degree of $N$ in $\left(\frac{\partial}{\partial u_{b_1}}\left[\log \mathcal{S}_{\rho_{\lfloor \frac{n_{s_1}}{2} \rfloor},X}\right]\right)^{m_{s_1}-1-r}$ is at most $m_{s_1}-1-r$. By Lemma \ref{p36}, the degree of $N$ in $\frac{\partial}{\partial u_{b_1}}\left[\mathcal{F}_{(m_{s_2},\frac{n_{s_2}}{2N})}\right]$ is at most $m_{s_2}$. The summation gives $O(N^{r+1})$ terms. Therefore, the degree of $N$ in $\left.G_{s_1,s_2}\right|_{(u_1,\ldots,u_N)=(x_1,\ldots,x_N)}$ is at most $m_{s_1}+m_{s_2}$. The fact that the degree of $N$ in $\frac{\partial}{\partial u_i}\left.G_{s_1,s_2}\right|_{(u_1,\ldots,u_N)=(x_1,\ldots,x_N)}$ is at most $m_{s_1}+m_{s_2}$ also follows from Lemmas \ref{l33} and \ref{p36}, and by discussing the cases when $i\in\{b_1,\ldots,b_{r+1}\}$ and $i\notin\{b_1,\ldots,b_{r+1}\}$ separately.

Now we prove Part (2). We have
\begin{eqnarray*}
&&\frac{1}{\hat{V}_{\lfloor\frac{n_{s_1}}{2}\rfloor} \mathcal{S}_{\rho_{\lfloor \frac{n_{s_1}}{2}\rfloor},X} }m_{s_1}\sum_{i\in\{N-\lfloor \frac{n_{s_1}}{2}\rfloor+1,\ldots,N\}}\left(u_i\frac{\partial}{\partial u_i}\left[\mathcal{F}_{(m_{s_2},\frac{n_{s_2}}{2N})}\right]\right)\\
&&\left.\left(u_i\frac{\partial}{\partial u_i}\right)^{m_{s_1}-1}\left[\hat{V}_{\lfloor\frac{n_{s_1}}{2}\rfloor}\mathcal{S}_{\rho_{\lfloor \frac{n_{s_1}}{2}\rfloor},X} \right]\right|_{(u_1,\ldots,u_N)=(x_1,\ldots,x_N)}\\
&=&m_{s_1}\sum_{t_0+t_1d_1+\ldots t_q d_q+r=m_{s_1}-1, t_1<t_2<\ldots<t_q}\left(\begin{array}{c}m_{s_1}-1\\r\end{array}\right)\sum_{\{b_1,\ldots,b_{r+1}\}\subset\{N-\lfloor\frac{n_{s_1}}{2} \rfloor+1,\ldots,N\}}(r+1)!\notag\\
&&\times\mathrm{Sym}_{b_1,\ldots,b_{r+1}}\frac{u_{b_1}^{m_{s_1}-t_0}\frac{\partial}{\partial u_{b_1}}\left[\mathcal{F}_{(m_{s_2},\frac{n_{s_2}}{2N})}\right]\left(\frac{\partial^{t_1}}{\partial u_{b_1}}\left[\log \mathcal{S}_{\rho_{\lfloor \frac{n_{s_1}}{2} \rfloor},X}\right]\right)^{d_1}\ldots\left(\frac{\partial^{t_q}}{\partial u_{b_1}}\left[\log \mathcal{S}_{\rho_{\lfloor \frac{n_{s_1}}{2} \rfloor},X}\right]\right)^{d_q}}{(u_{b_1}-u_{b_2})\ldots(u_{b_1}-u_{b_{r+1}})}
\end{eqnarray*}
By Lemmas \ref{l33} and \ref{p36}, the degree of $N$ in the expression above is at most
\begin{eqnarray}
d_1+\ldots+d_q+m_{s_2}+r+1.\label{dne}
\end{eqnarray}
Given that $t_0+t_1d_1+\ldots t_q d_q+r=m_{s_1}-1, t_1<t_2<\ldots<t_q$, the maximal of (\ref{dne}) achieves when $t_0=d_2=\ldots=d_q=0$, $t_1=1$, $d_1=m_{s_1}-r-1$; with maximal value $m_{s_1}+m_{s_2}$. Then Part (2) follows.
\end{proof}

\begin{lemma}\label{p510}Let $m_1,\ldots,m_k$ and $n_1,\ldots,n_k$ be as in Lemma \ref{l59}. Then
\begin{eqnarray*}
&&\lim_{N\rightarrow\infty}\frac{1}{N^{m_1+m_2+\ldots+m_k}}\left(\mathcal{D}_{m_1}^{(n_1)}-E_{m_1,s_1}\right)
 \left(\mathcal{D}_{m_2}^{(n_2)}-E_{m_2,s_{2}}\right)\ldots\left(\mathcal{D}_{m_k}^{(n_k)}-E_{m_k,s_k}\right)\\
&& \left.\mathcal{S}_{\rho,X}(u_{1,1},\ldots,u_{\lfloor \frac{n_1}{2}\rfloor,1};\ldots;u_{1,k},\ldots,u_{\lfloor \frac{n_k}{2}\rfloor,k})\right|_{(u_{1,s},\ldots,u_{\lfloor \frac{n_s}{2}\rfloor,s})=\ol{X}_s,\ \forall 1\leq s\leq m}\\
&=&\lim_{N\rightarrow\infty}\frac{1}{N^{m_1+m_2+\ldots+m_k}}\sum_{P\in \mathcal{P}_{\emptyset}^{k}}\left.\prod_{(s_1,s_2)\in P}G_{s_1,s_2}\right|_{(u_1,\ldots,u_N)=(x_1,\ldots,x_N)}.
\end{eqnarray*}
\end{lemma}
\begin{proof}The proposition follows from explicit computations and by analyzing the degree of $N$ of each term in the expansion. We sketch the proof here.

The lemma obviously holds when $k=1$, for which both the left hand side and the right hand side are 0. When $k=2$, by Lemma \ref{ll511}, we have
\begin{eqnarray*}
\mathcal{E}_2:&=&\mathcal{D}_{m_1}^{(n_1)}\mathcal{D}_{m_2}^{(n_2)} \left.\mathcal{S}_{\rho,X}(u_{1,1},\ldots,u_{\lfloor \frac{n_1}{2}\rfloor,1};u_{1,2},\ldots,u_{\lfloor \frac{n_2}{2}\rfloor,2})\right|_{(u_{1,s},\ldots,u_{\lfloor \frac{n_s}{2}\rfloor,s})=\ol{X}_s,\ \mathrm{for}\ s=1,2}\\
&=&\tilde{\mathcal{D}}_{m_{1,2}}^{(n_1)}\tilde{\mathcal{D}}_{m_2,2}^{(n_2)} \left.\mathcal{S}_{\rho_{\lfloor\frac{n_2}{2}\rfloor},X}(u_{N-\lfloor\frac{n_2}{2} \rfloor+1},\ldots,u_N)\right|_{(u_1,\ldots,u_N)=(x_1,\ldots,x_N)}\\
&=&\frac{1}{\mathcal{S}_{\rho_{\lfloor\frac{n_1}{2} \rfloor},X}(u_{N-\lfloor\frac{n_2}{2} \rfloor+1},\ldots,u_N)}\tilde{\mathcal{D}}_{m_1}^{(n_1)}\tilde{\mathcal{D}}_{m_2}^{(n_2)}\\
&&\left.\left[ \exp\left(\log\left[\mathcal{S}_{\rho_{\lfloor\frac{n_2}{2} \rfloor},X}(u_{N-\lfloor\frac{n_2}{2} \rfloor+1},\ldots,u_N)\right]\right)\right]\right|_{(u_1,\ldots,u_N)=(x_1,\ldots,x_N)}\\
&=&\frac{1}{\mathcal{S}_{\rho_{\lfloor\frac{n_1}{2} \rfloor},X}(u_{N-\lfloor\frac{n_2}{2} \rfloor+1},\ldots,u_N)V_{\lfloor\frac{n_1}{2} \rfloor}(u_{N-\lfloor\frac{n_1}{2} \rfloor+1},\ldots,u_N)}\sum_{i\in\{N-\lfloor\frac{n_1}{2} \rfloor+1,\ldots,N\}}\left(u_{i}\frac{\partial}{\partial u_{i}}\right)^{m_1}\\
&&\frac{V_{\lfloor\frac{n_1}{2} \rfloor}(u_{N-\lfloor\frac{n_1}{2} \rfloor+1},\ldots,u_N)}{V_{\lfloor\frac{n_2}{2} \rfloor}(u_{N-\lfloor\frac{n_2}{2} \rfloor+1},\ldots,u_N)}\prod_{l\in\{N-\lfloor\frac{n_2}{2} \rfloor+1,\ldots,N-\lfloor\frac{n_1}{2} \rfloor\}\cap I_2}\prod_{s=N-\lfloor\frac{n_1}{2}\rfloor+1}^{N}\left(\frac{1+y_{\ol{l}}u_s}{1+y_{\ol{l}}x_{\ol{s}}}\right)\\
&&\sum_{j\in\{N-\lfloor\frac{n_2}{2} \rfloor+1,\ldots,N\}}\left(u_{j}\frac{\partial}{\partial u_{j}}\right)^{m_2}V_{\lfloor\frac{n_2}{2} \rfloor}(u_{N-\lfloor\frac{n_2}{2} \rfloor+1},\ldots,u_{N})\\
&&\left.\left[ \exp\left(\log\left[\mathcal{S}_{\rho_{\lfloor \frac{n_2}{2}\rfloor},X}(u_{N-\lfloor\frac{n_2}{2}\rfloor+1},\ldots,u_N)\right]\right)\right]\right|_{(u_1,\ldots,u_N)=(x_1,\ldots,x_N)}
\end{eqnarray*}
By Lemma \ref{lmm212}, we have
\begin{eqnarray*}
\prod_{l\in\{N-\lfloor\frac{n_2}{2} \rfloor+1,\ldots,N-\lfloor\frac{n_1}{2} \rfloor\}\cap I_2}\prod_{s=N-\lfloor\frac{n_1}{2}\rfloor+1}^{N}\left(\frac{1+y_{\ol{l}}u_s}{1+y_{\ol{l}}x_{\ol{s}}}\right)=\frac{\mathcal{S}_{\rho_{\lfloor \frac{n_1}{2}\rfloor},X}(u_{N-\lfloor \frac{n_1}{2}\rfloor}+1,\ldots,u_N)}{\mathcal{S}_{\rho_{\lfloor \frac{n_2}{2}\rfloor},X}(u_{N-\lfloor \frac{n_2}{2}\rfloor}+1,\ldots,u_N)}
\end{eqnarray*}
Then
\begin{eqnarray*}
\mathcal{E}_2&=&\frac{1}{\mathcal{S}_{\rho_{\lfloor\frac{n_1}{2} \rfloor},X}(u_{N-\lfloor\frac{n_2}{2} \rfloor+1},\ldots,u_N)V_{\lfloor\frac{n_1}{2} \rfloor}(u_{N-\lfloor\frac{n_1}{2} \rfloor+1},\ldots,u_N)}\\
&&\sum_{i\in\{N-\lfloor\frac{n_1}{2}\rfloor+1,\ldots,N\}}\left(u_{i}\frac{\partial}{\partial u_{i}}\right)^{m_1}V_{\lfloor\frac{n_1}{2} \rfloor}(u_{N-\lfloor\frac{n_1}{2} \rfloor+1},\ldots,u_N)\mathcal{S}_{\rho_{\lfloor \frac{n_1}{2}\rfloor},X}(u_{N-\lfloor\frac{n_2}{2}\rfloor+1},\ldots,u_N)\\
&&\left.\mathcal{F}_{m_2,\lfloor\frac{n_2}{2N} \rfloor}(u_{N-\lfloor\frac{n_2}{2}\rfloor+1},\ldots,u_N)\right|_{(u_1,\ldots,u_N)=(x_1,\ldots,x_N)}
\end{eqnarray*}
Hence $\mathcal{E}_2$ is a sum of terms of the following form
\begin{eqnarray*}
&&\mathrm{Sym}_{a_1,\ldots,a_{r+1}}\\
&&\left.\left[\frac{c_0u_{a_1}^{m_1-q_0}\frac{\partial^{q_1}}{\partial u_{a_1}^{q_1}}[\mathcal{F}_{m_2,\frac{n_2}{2N}}]\left[\frac{\partial^{q_2}}{\partial u_{a_1}^{q_2}}(\log \mathcal{S}_{\rho_{\lfloor\frac{n_1}{2}\rfloor},X})\right]^{d_2}\ldots \left[\frac{\partial^{q_t}}{\partial u_{a_1}^{q_t}}(\log \mathcal{S}_{\rho_{\lfloor\frac{n_1}{2} \rfloor},X})\right]^{d_t}}{(u_{a_1}-u_{a_2})\ldots (u_{a_1}-u_{a_{r+1}})}\right]\right|_{(u_1,\ldots,u_N)=(x_1,\ldots,x_N)}.
\end{eqnarray*}
such that 
\begin{itemize}
\item  $r$, $q_0,q_1,\ldots,q_t$, $d_2,\ldots,d_t$ are nonnegative integers; and
\item  $q_2<q_3<\ldots<q_t$; and
\item 
\begin{eqnarray}
q_0+q_1+q_2d_2+\ldots+q_td_t+r=m_1; \label{ssdk}
\end{eqnarray}
and
\item $\{a_1,\ldots,a_{r+1}\}\subset\{N-\lfloor \frac{n_1}{2}\rfloor+1,\ldots,N\}$
\end{itemize}
When $q_1=0$, we obtain $E_{m_1,s_1}E_{m_2,s_2}$.

Now we consider the terms corresponding to $q_1\geq 1$. By Lemma \ref{p36}, the degree of $N$ in $\partial_{a_1}^{q_1}\left[\mathcal{F}_{m_2,\frac{n_2}{2N}}\right]$ is at most $m_2$. By Lemma \ref{l33}, the total degree of $N$ in these terms is at most $m_2+d_2+\ldots+d_t+r+1$. By (\ref{ssdk}) and the assumption that $s_1\geq 1$, we have
\begin{eqnarray*}
m_2+d_2+\ldots+d_t+r+1\leq m_2+m_1；
\end{eqnarray*}
and the equality holds when $q_0=d_3=\ldots=d_t=0$, $q_1=q_2=1$; $d_2=m_1-1-r$; this corresponds to $G_{1,2}$, in which the degree of $N$ is at most $m_1+m_2$ . The degree of $N$ is less than $m_1+m_2$ in all the other terms. This completes the proof when $k=2$.

We shall finish the rest of the proof by induction. For $1\leq l\leq k-1$, let
\begin{eqnarray*}
A_l=\prod_{[i\in\{t_{l+1}+1,\ldots,t_l\}]}\prod_{j=t_l+1}^{N}\left(\frac{1+y_{\ol{i}}u_j}{1+y_{\ol{i}}x_{\ol{j}}}\right)
\end{eqnarray*}
By Lemma \ref{lmm212}, we have
\begin{eqnarray*}
A_l=\frac{\mathcal{S}_{\rho_{\lfloor\frac{n_l}{2} \rfloor},X}(u_{N-\lfloor\frac{n_l}{2} \rfloor+1,},\ldots,u_N)}{\mathcal{S}_{\rho_{\lfloor\frac{n_{l+1}}{2} \rfloor},X}(u_{N-\lfloor\frac{n_{l+1}}{2} \rfloor+1,},\ldots,u_N)}
\end{eqnarray*}

 Assume that the lemma holds for $k=r-1$, where $r\geq 2$ is a positive integer. When $k=r$, by induction hypothesis, we have 
\begin{eqnarray*}
&&\frac{1}{\hat{V}_{\lfloor \frac{n_1}{2}\rfloor}\mathcal{S}_{\rho_{\lfloor\frac{n_1}{2} \rfloor},X}}\sum_{[i_1\in\{N-\lfloor\frac{n_1}{2}\rfloor+1,\ldots,N\}]}\left(u_{i_1}\frac{\partial}{\partial u_{i_1}}\right)^{m_1}A_1\frac{\hat{V}_{\lfloor\frac{n_1}{2} \rfloor}}{\hat{V}_{\lfloor\frac{n_2}{2} \rfloor}}\sum_{[i_2\in\{N-\lfloor\frac{n_2}{2} \rfloor\rfloor+1,\ldots,N\}]}\left(u_{i_2}\frac{\partial}{\partial u_{i_2}}\right)^{m_2}
\\&&\frac{\hat{V}_{\lfloor\frac{n_2}{2}\rfloor}}{\hat{V}_{\lfloor\frac{n_3}{2}\rfloor}}A_{2}\cdots\left.\sum_{[i_r\in\{N-\lfloor\frac{n_r}{2} \rfloor+1,\ldots,N\}}\left(u_{i_r}\frac{\partial}{\partial u_{i_r}}\right)^{m_r}\left[\hat{V}_{\lfloor\frac{n_r}{2} \rfloor}\mathcal{S}_{\rho_{\lfloor \frac{n_t}{2}\rfloor},X}\right]\right|_{(u_1,\ldots,u_N)=(x_1,\ldots,x_N)}\\
&=&\frac{1}{\hat{V}_{\lfloor \frac{n_1}{2}\rfloor}\mathcal{S}_{\rho_{\lfloor\frac{n_1}{2} \rfloor},X}}\sum_{[i_1\in\{N-\lfloor\frac{n_1}{2}\rfloor+1,\ldots,N\}]}\left(u_{i_1}\frac{\partial}{\partial u_{i_1}}\right)^{m_1}\left[\hat{V}_{\lfloor \frac{n_1}{2}\rfloor}\mathcal{S}_{\rho_{\lfloor\frac{n_1}{2} \rfloor},X}\right]\\
&&\left(\sum_{p=0}^{r-1}\sum_{[w_1,\ldots,w_p\in\{2,\ldots,r\}]}\mathcal{F}_{(m_{w_1},\frac{n_{w_1}}{2N})}\mathcal{F}_{(m_{w_2},\frac{n_{w_2}}{2N})}\ldots \mathcal{F}_{(m_{w_p},\frac{n_{w_p}}{2N})}\right.\\
&&\left.\left.\left(\sum_{P\in \mathcal{P}_{1,w_1,\ldots,w_p}^{r}}\prod_{(a,b)\in P}G_{a,b}+R_{1,w_1,\ldots,w_p}\right)\right)\right|_{(u_1,\ldots,u_N)=(x_1,\ldots,x_N)}.\\
&=& S_1+S_2+S_3;
\end{eqnarray*}
where by induction hypothesis the degree of $N$ in $R_{1,w_1,\ldots,w_p}$ is less than $\sum_{i=2}^{s}l_i-\sum_{i=1}^{p}l_{w_i}$; and 
\begin{eqnarray*}
S_1&=&\left\{\frac{1}{\hat{V}_{\lfloor \frac{n_1}{2}\rfloor}\mathcal{S}_{\rho_{\lfloor \frac{n_1}{2}\rfloor},X}}\sum_{[i_1\in\{N-\lfloor\frac{n_1}{2} \rfloor+1,\ldots,N\}]}\left(u_{i_1}\frac{\partial}{\partial u_{i_1}}\right)^{l_1}\left[\hat{V}_{\lfloor \frac{n_1}{2}\rfloor}\mathcal{S}_{\rho_{\lfloor\frac{n_1}{2} \rfloor},X}\right]\right\}\\
&&\left(\sum_{p=0}^{r-1}\sum_{[w_1,\ldots,w_p\in\{2,\ldots,r\}]}E_{m_{w_1},w_1}E_{m_{w_2},w_2}\ldots E_{m_{w_p},w_p}\right.\\
&&\left.\left.\left(\sum_{P\in \mathcal{P}_{1,w_1,\ldots,w_p}^{r}}\prod_{(a,b)\in P}G_{a,b}+R_{1,w_1,\ldots,w_p}\right)\right)\right|_{(u_1,\ldots,u_N)=(x_1,\ldots,x_N)}.
\end{eqnarray*}
and
\begin{eqnarray*}
S_2&=&\frac{\ell_1}{\hat{V}_{\lfloor \frac{n_1}{2}\rfloor}\mathcal{S}_{\rho_{\lfloor\frac{n_1}{2} \rfloor},X}}\sum_{[i_1\in\{N-\lfloor\frac{n_1}{2} \rfloor+1,\ldots,N\}]}\left\{\left(u_{i_1}\frac{\partial}{\partial u_{i_1}}\right)^{l_1-1}\left[\hat{V}_{\frac{n_1}{2}}\mathcal{S}_{\rho_{\lfloor\frac{n_1}{2} \rfloor},X}\right]\right\}\\
&&\times\left\{\left(u_{i_1}\frac{\partial}{\partial u_{i_1}}\right)\left(\sum_{p=0}^{r-1}\sum_{[w_1,\ldots,w_p\in\{2,\ldots,t\}]}\mathcal{F}_{(m_{w_1},\frac{n_{w_1}}{2N})}\mathcal{F}_{(m_{w_2},\frac{n_{w_2}}{2N})}\ldots \mathcal{F}_{(m_{w_p},\frac{w_p}{2N})}\right.\right.\\
&&\left.\left.\left(\sum_{P\in \mathcal{P}_{1,w_1,\ldots,w_p}^{r}}\prod_{(a,b)\in P}G_{a,b}+R_{1,w_1,\ldots,w_p}\right)\right)\right|_{(u_1,\ldots,u_N)=(x_1,\ldots,x_N)};
\end{eqnarray*}
and $S_3$ consists of all the other terms.

By the definition of $E_{m_i,i}$ we have
\begin{eqnarray*}
S_1&=&E_{m_1,1}\left(\sum_{p=0}^{r-1}\sum_{[w_1,\ldots,w_p\in\{2,\ldots,r\}]}E_{m_{w_1},w_1}E_{m_{w_2},w_2}\ldots E_{m_{w_p},w_p}\right.\\
&&\left.\left.\left(\sum_{P\in \mathcal{P}_{1,w_1,\ldots,w_p}^{r}}\prod_{(a,b)\in P}G_{a,b}+R_{1,w_1,\ldots,w_p}\right)\right)\right|_{(u_1,\ldots,u_N)=(x_1,\ldots,x_N)}.
\end{eqnarray*}

By Lemma \ref{ll66}, we have
\begin{eqnarray*}
S_2&=&\left\{\left(\sum_{p=0}^{r-1}\sum_{[w_1,\ldots,w_p\in\{2,\ldots,r\}]}\sum_{x=1}^{p}E_{m_{w_1,w_1}}\ldots E_{m_{w_{x-1}},w_{x-1}} E_{m_{w_{x+1}},w_{x+1}}\ldots E_{m_{w_p},w_p}\right.\right.\\
&&\left.\left.\left.\left[\left(G_{1,x}+R_{1,x}\right)\left(\sum_{P\in \mathcal{P}_{1,w_1,\ldots,w_p}^{r}}\prod_{(a,b)\in P}G_{a,b}+R_{1,w_1,\ldots,w_p}\right)\right]\right)\right\}\right|_{(u_1,\ldots,u_N)=(x_1,\ldots,x_N)},
\end{eqnarray*}
where the degree of $N$ in $R_{1,x}$ is less than $l_1+l_x$.

Hence we have
\begin{eqnarray*}
S_2&=&\left\{\left(\sum_{p=0}^{r-1}\sum_{[w_1,\ldots,w_p\in\{2,\ldots,r\}]}\sum_{x=1}^{p}E_{m_{w_1,w_1}}\ldots E_{m_{w_{x-1}},w_{x-1}} E_{m_{w_{x+1}},w_{x+1}}\ldots E_{m_{w_p},w_p}\right.\right.\\
&&\left.\left.\left.\left[G_{1,x}\sum_{P\in \mathcal{P}_{1,w_1,\ldots,w_p}^{r}}\prod_{(a,b)\in P}G_{a,b}+R_{w_1,\ldots,\hat{w}_x,\ldots,w_p}\right)\right]\right\}\right|_{(u_1,\ldots,u_N)=(x_1,\ldots,x_N)}
\end{eqnarray*}
where the degree of $N$ in $R_{w_1,\ldots,\hat{w}_x,\ldots,w_p}$ is less than $\sum_{i=2}^{t}l_i-\sum_{i=1}^{p}l_{w_i}$.

Note that 
\begin{eqnarray*}
S_1+S_2&=&\sum_{p=0}^{r}\sum_{[w_1,\ldots,w_p\in\{1,2,\ldots,r\}]}E_{m_{w_1},w_1}E_{m_{w_2},w_2}\ldots E_{m_{w_p},w_p}\\
&&\left.\left(\sum_{P\in \mathcal{P}_{w_1,\ldots,w_p}^{r}}\prod_{(a,b)\in P} G_{a,b}+R_{w_1,\ldots,w_p}\right)\right|_{(u_1,\ldots,u_N)=(x_1,\ldots,x_N)}
\end{eqnarray*}
where the degree of $N$ in $R_{w_1,\ldots,w_p}$ is less than $\sum_{i=1}^{t}l_i-\sum_{i=1}^{p}l_{w_i}$. 

It remains to show that $S_3$ does not contribute to the leading terms. Define
\begin{eqnarray*}
\mathcal{H}_{j_1,\ldots,j_p}=\left(\sum_{P\in \mathcal{P}_{1,w_1,\ldots,w_p}^{r}}\prod_{(a,b)\in P}G_{a,b}+R_{1,w_1,\ldots,w_p}\right);
\end{eqnarray*}
By Lemma \ref{ll66}, the degree of $N$ in $\left.\mathcal{H}_{j_1,\ldots,j_p}\right|_{U_{N,\kappa}=(1,\ldots,1)}$ is at most $\sum_{i=2}^r l_i-\sum_{j=1}^p l_{w_j}$. Moreover, by Lemma \ref{ll66}, for any index $i$, the degree of $N$ in $\left.\frac{\partial}{\partial u_i}\mathcal{H}_{j_1,\ldots,j_p}\right|_{(u_1,\ldots,u_N)=(x_1,\ldots,x_N)}$ is less than $\sum_{i=2}^t l_i-\sum_{j=1}^p l_{w_j}$.

We write
\begin{eqnarray*}
&&\frac{1}{\hat{V}_{\lfloor \frac{n_1}{2}\rfloor}\mathcal{S}_{\rho_{\lfloor\frac{n_1}{2} \rfloor},X}}\sum_{[i_1\in\{N-\lfloor\frac{n_1}{2}\rfloor+1,\ldots,N\}]}\left(u_{i_1}\frac{\partial}{\partial u_{i_1}}\right)^{m_1}\left[\hat{V}_{\lfloor \frac{n_1}{2}\rfloor}\mathcal{S}_{\rho_{\lfloor\frac{n_1}{2} \rfloor},X}\right]\\
&&\left.\left(\sum_{p=0}^{r-1}\sum_{[w_1,\ldots,w_p\in\{2,\ldots,r\}]}\mathcal{F}_{(m_{w_1},\frac{n_{w_1}}{2N})}\mathcal{F}_{(m_{w_2},\frac{n_{w_2}}{2N})}\ldots \mathcal{F}_{(m_{w_p},\frac{n_{w_p}}{2N})}\mathcal{H}_{j_1,\ldots,j_p}\right)\right|_{(u_1,\ldots,u_N)=(x_1,\ldots,x_N)}
\end{eqnarray*}
as a sum of terms of the following form
\begin{eqnarray}
\mathrm{Sym}_{a_1,\ldots,a_{q+1}}\left[\frac{u_{a_1}^{m_1-s_0}(\partial_{a_1}^{s_1}[\log\mathcal{S}_{\rho_{\lfloor \frac{n_1}{2}\rfloor},X}])^{d_1}\ldots (\partial_{a_t}^{s_t}[\log\mathcal{S}_{\rho_{\lfloor \frac{n_1}{2}\rfloor},X}])^{d_t} }{(u_{a_1}-u_{a_2})\ldots(u_{a_1}-u_{a_{q+1}})}\right.\label{ssa1}\\
\left.\frac{\partial_{a_1}^{f_1}\mathcal{F}_{(m_{w_1},\frac{n_{w_1}}{2N})} \ldots\partial_{a_1}^{f_p}\mathcal{F}_{(m_{w_p},\frac{n_{w_p}}{2N})}\partial_{a_1}^{h_0}\mathcal{H}_{j_1,\ldots,j_p} }{}\right]\notag
\end{eqnarray}
where
\begin{itemize}
\item $\{a_1,\ldots,a_{q+1}\}\subset\{N-\lfloor\frac{n_1}{2} \rfloor+1,\ldots,N\}$;
\item $s_1<s_2<\ldots<s_t$ are positive integers;
\item $f_1,\ldots,f_p,h_0$ are nonnegative integers;
\item 
\begin{eqnarray}
q+s_0+s_1d_1+\ldots+s_td_t+f_1+\ldots+f_p+h_0=m_1\label{rsdm}
\end{eqnarray}
\end{itemize}
By Lemma \ref{l33}, the degree of $N$ in $[\log\mathcal{S}_{\rho_{\lfloor \frac{n_1}{2}\rfloor},X}])^{d_1}\ldots (\partial_{a_t}^{s_t}[\log\mathcal{S}_{\rho_{\lfloor \frac{n_1}{2}\rfloor},X}])^{d_t}$ is at most $d_1+\ldots+d_t$; therefore, the terms in (\ref{ssa1}) with highest degree of $N$ has the form
\begin{eqnarray}
\mathrm{Sym}_{a_1,\ldots,a_{q+1}}\left[\frac{u_{a_1}^{m_1}(\partial_{a_1}[\log\mathcal{S}_{\rho_{\lfloor \frac{n_1}{2}\rfloor},X}])^{d_1} }{(u_{a_1}-u_{a_2})\ldots(u_{a_1}-u_{a_{q+1}})}\right.\\
\left.\frac{\partial_{a_1}^{f_1}\mathcal{F}_{(m_{w_1},\frac{n_{w_1}}{2N})} \ldots\partial_{a_1}^{f_p}\mathcal{F}_{(m_{w_p},\frac{n_{w_p}}{2N})}\partial_{a_1}^{h_0}\mathcal{H}_{j_1,\ldots,j_p} }{}\right]
\label{ssa2}
\end{eqnarray}
where
\begin{eqnarray}
s_0=d_2=\ldots=d_t=0;\ s_1=1.\label{sd0}
\end{eqnarray}
Let
\begin{eqnarray*}
B=\{i\in\{1,2,\ldots,p\}:f_i=0\}.
\end{eqnarray*}
Then 
\begin{eqnarray*}
(\ref{ssa2})=\left[\prod_{i\in B}\mathcal{F}_{(m_{w_i},\frac{n_{w_i}}{2N})}\right] S(u_1,\ldots,u_{N})
\end{eqnarray*}
It suffices to show that the degree of $N$ in $S$, except for $S_1$ and $S_2$, is less than $\sum_{i=1}^{s}l_i-\sum_{i\in B}l_i$. Note that the degree of $N$ in $\partial_{a_1}[\log\mathcal{S}_{\rho_{\lfloor \frac{n_1}{2}\rfloor},X}])^{d_1}$ is at most $d_1$ by Lemma \ref{l33}. The summation over $\{a_1,\ldots,a_{q+1}\}\subset\{N-\lfloor\frac{n_1}{2} \rfloor+1,\ldots,N\}$ gives $O(N^{q+1})$ terms.  By Lemma \ref{p36}, when $i\notin B$, the degree of $N$ in $\partial_{a_1}^{f_i}\mathcal{F}_{(m_{w_i},\frac{n_{w_i}}{2N})} $ is at most $m_{w_i}$. Therefore the degree of $N$ in $S(u_1,\ldots,u_N)$ is at most
\begin{eqnarray*}
\sum_{i=2}^{r}m_i-\sum_{i=1}^{p}m_{w_i}+d_1+\sum_{i\in \{1,2,\ldots,p\}\setminus B}m_{w_i}+q+1
\end{eqnarray*}

 By (\ref{rsd}) and (\ref{sd0}), if $|B|\leq p-2$, $q+d_1+1\leq m_1-1$, then the degree of $N$ in $S(u_1,\ldots,u_N)$ is at most
 \begin{eqnarray*}
 \sum_{i=1}^{r}m_i-\sum_{i\in B}m_{w_i}-1
 \end{eqnarray*}
 Therefore only the terms where at most one $f_i$ is nonzero contribute to the leading order. In these terms if $h_0>0$,  then by Lemma \ref{l67}, the degree of $N$ is less than $\sum_{i=1}^{r}m_i-\sum_{i\in B}m_{w_i}$. So only the terms where $h_0=0$ and at most one $f_i$ is nonzero contribute to the leading order. These terms are in $S_1$ and $S_2$. Then the proof is complete.

\end{proof}

\begin{theorem}\label{gff1}The collection of random variables 
\begin{eqnarray*}
\{N^{-l}\left[p_{l}^{((1-\kappa) N)}-\mathbb{E}p_{l}^{((1-\kappa) N)}\right]\}_{l\in \NN;\kappa=a_1,\ldots,a_m}
\end{eqnarray*}
converges, as $N\rightarrow\infty$, in the sense of moments, to the Gaussian vector with zero mean and covariance
\begin{eqnarray*}
&&\lim_{N\rightarrow\infty}\frac{\mathrm{cov}\left(p_{l_1}^{\lfloor(1-\kappa_1) N\rfloor},p_{l_2}^{\lfloor(1-\kappa_2)N\rfloor}\right)}{N^{l_1+l_2}}\\
&=&\frac{(1-\kappa_1)^{l_1}(1-\kappa_2)^{l_2}}{(2\pi\mathbf{i})^2}\sum_{i=1}^{n}\sum_{j=1}^{n}\oint_{|z-x_i|=\epsilon}\oint_{|w-x_j|=\epsilon}\left(\sum_{i=1}^{n}\frac{z}{n(z-x_i)}+\frac{z H(z)}{1-\kappa_1}\right)^{l_1}\\
&&\times\left(\sum_{i=1}^{n}\frac{w}{n(w-x_j)}+\frac{w H(w)}{1-\kappa_2}\right)^{l_2}Q(z,w)dzdw,
\end{eqnarray*}
where 
\begin{itemize}
\item $\epsilon>0$ is sufficiently small such that the disk centered at $x_i$ with radius $\epsilon$ contains exactly one singularity $x_i$ of the integrand.
\item the $z$- and $w$-contours of integration are counter-clockwise.
\end{itemize}
\end{theorem}

\begin{proof}The theorem follows from Lemma \ref{p510} similar arguments as in the single-level case. The only difference is in the expansion $\{b_1,\ldots,b_{r+1}\}\subset\{1,2,\ldots,\lfloor (1-t_1)N\rfloor\}$, while $\{a_1,\ldots,a_{q+1}\}\subset\{1,2,\ldots,\lfloor(1-t_2)N \rfloor\}$.
\end{proof}

\section{Piecewise Boundary Conditions}\label{pbr}

In this section, we introduce the piecewise boundary conditions on the bottom boundary of a contracting square-hexagon lattice, and review the limit shape results for perfect matchings on such a graph.

For $N\geq 1$, let $\lambda(N)\in \GT_N^+$. We consider the following special asymptotical case of $\lambda(N)$ as $N\rightarrow\infty$. Let
\begin{eqnarray*}
\Omega=(\Omega_1<\Omega_2<\ldots<\Omega_N)=(\lambda_N(N)+1,\lambda_{N-1}(N)+2,\ldots,\lambda_1(N)+N)
\end{eqnarray*}
Indeed, $\Omega_1,\ldots,\Omega_N$ are the locations of the $N$ remaining vertices on the bottom boundary of the contracting square-hexagon lattice.
Assume
\begin{eqnarray}
\Omega&=&(A_1,A_1+1,\ldots.B_1-1,B_1,\label{abt}\\
&&A_2,A_2+1,\ldots,B_2-1,B_2,\ldots,A_s,A_s+1,\ldots,B_s-1,B_s).\notag
\end{eqnarray}
where 
\begin{eqnarray*}
\sum_{i=1}^{s}(B_i-A_i+1)=N.
\end{eqnarray*}
and $s$ is a fixed positive integer independent of $N$. Suppose as $N\rightarrow\infty$,
\begin{eqnarray*}
A_i(N)=a_iN+o(N),\qquad B_i(N)=b_iN+o(N),\qquad \mathrm{for}\ 1\leq i\leq s,
\end{eqnarray*}
and $a_1<b_1<\ldots<a_s<b_s$ are fixed parameters independent of $N$ and satisfy $\sum_{i=1}^{s}(b_i-a_i)=1$. Assume the edge weights  $\{x_i\}_{i=1}^{N}$ $\{y_j\}_{j\in I_2\cap\{1,2,\ldots,N\}}$ satisfy (\ref{px}) and (\ref{py}).

Let $\Sigma_N$ be the permutation group of $N$ elements and let $\sigma\in \Sigma_N$. Let 
\begin{eqnarray*}
X=(x_1,\ldots,x_N).
\end{eqnarray*}
Let $x_1,\ldots, x_n$ be all the distinct elements in $\{x_1,\ldots,x_N\}$. Let $\Sigma_N^{X}$ be the subgroup of $\Sigma_N$ that preserves the value of $X$; more precisely
\begin{eqnarray*}
\Sigma_N^{X}=\{\sigma\in \Sigma_N: x_{\sigma(i)}=x_i,\ \mathrm{for}\ 1\leq i\leq N\}
\end{eqnarray*}
Let $[\Sigma_N/\Sigma_N^X]^r$ be the collection of all the right cosets of $\Sigma_N^X$ in $\Sigma_N$. More precisely,
\begin{eqnarray*}
[\Sigma_N/\Sigma_N^X]^r=\{\Sigma_N^X\sigma:\sigma\in \Sigma_N\},
\end{eqnarray*}
where for each $\sigma\in \Sigma_N$
\begin{eqnarray*}
\Sigma_N^X\sigma=\{\xi\sigma:\xi\in \Sigma_N^X\}
\end{eqnarray*}
and $\xi\sigma\in \Sigma_N$ is defined by
\begin{eqnarray*}
\xi\sigma(k)=\xi(\sigma(k)),\ \mathrm{for}\ 1\leq k\leq N.
\end{eqnarray*}

 For $1\leq j\leq N$, let
\begin{eqnarray}
\eta_j^{\sigma}(N)=|\{k:k>j,x_{\sigma(k)}\neq x_{\sigma(j)}\}|.\label{et}
\end{eqnarray}
For $1\leq i\leq n$, let
\begin{eqnarray}
\Phi^{(i,\sigma)}(N)=\{\lambda_j(N)+\eta_j^{\sigma}(N):x_{\sigma(j)}=x_i\}\label{pis}
\end{eqnarray}
and let $\phi^{(i,\sigma)}(N)$ be the partition obtained by decreasingly ordering all the elements in $\Phi^{(i,\sigma)}(N)$. Let $\Sigma_N^{X}$ be the subgroup $\Sigma_N$ that preserves the value of $X$; more precisely
\begin{eqnarray*}
\Sigma_N^{X}=\{\sigma\in \Sigma_N: x_{\sigma(i)}=x_i,\ \mathrm{for}\ 1\leq i\leq N\}
\end{eqnarray*}
Let $[\Sigma/\Sigma_N^X]^r$ be the collection of all the right cosets of $\Sigma_N^X$ in $\Sigma_N$. More precisely,
\begin{eqnarray*}
[\Sigma/\Sigma_N^X]^r=\{\Sigma_N^X\sigma:\sigma\in \Sigma_N\},
\end{eqnarray*}
where for each $\sigma\in \Sigma_N$
\begin{eqnarray*}
\Sigma_N^X\sigma=\{\xi\sigma:\xi\in \Sigma_N^X\}
\end{eqnarray*}
and $\xi\sigma\in \Sigma_N$ is defined by
\begin{eqnarray*}
\xi\sigma(k)=\xi(\sigma(k)),\ \mathrm{for}\ 1\leq k\leq N.
\end{eqnarray*}

For simplicity, we make the following assumptions.

\begin{assumption}\label{ap423}Let $(x_1,\ldots,x_N)$ be an $N$-tuple of real numbers at which we evaluate the Schur polynomial.
\begin{itemize}
\item $N$ is an integral multiple of $n$; and.
\item $\{x_i\}_{i=1}^{N}$ are periodic with period $n$, i.e., $x_{i}=x_{j}$ for $1\leq i,j\leq N$ and $[i\mod n]=[j\mod n]$; 
\item $x_1>x_2>\ldots>x_n>0$.
\end{itemize}
\end{assumption}

We may further make the assumptions below

\begin{assumption}\label{ap32}Assume $x_{1,N}=x_1>0$ and $(x_{2,N},\ldots,x_{n,N})$ changes with $N$. Assume that for each fixed $N$,  $(x_{1,N},\ldots,x_{n,N})$ satisfies Assumption \ref{ap423}. Moreover, assume that
\begin{eqnarray*}
\liminf_{N\rightarrow\infty} \frac{\log\left(\min_{1\leq i<j\leq n}\frac{x_{i,N}}{x_{j,N}}\right)}{\log N}\geq \alpha>0,
\end{eqnarray*}
where $\alpha$ is a sufficiently large positive constant independent of $N$.
\end{assumption}

Let $\ol{\sigma}_0\in [\Si_N/\Si_N^X]^r$ be the unique element in $[\Si_N/\Si_N^X]^r$ satisfying the condition that for any representative $\sigma_0\in\ol{\sigma}_0$, we have
\begin{eqnarray}
x_{\si_0(1)}\geq x_{\si_0(2)}\geq\ldots\geq x_{\si_0(N)}.\label{sz}
\end{eqnarray}
Let $\bm_{i}$ be the limit of the counting measures for $\phi^{(i,\si_0)}(N)$ as $N\rightarrow\infty$.

\begin{assumption}\label{ap428}Assume $x_1,\ldots,x_N$ satisfy Assumption \ref{ap423}.

Let $A_i$, $B_i$ be given as in (\ref{abt}). For $1\leq i\leq s$, let
\begin{eqnarray*}
B_i-A_i+1=K_i.
\end{eqnarray*}

By (\ref{abt}), we may assume
\begin{eqnarray*}
&&\lambda_1=\lambda_2=\ldots=\lambda_{K_s}=\mu_1;\\
&&\lambda_{K_s+1}=\lambda_{K_s+2}=\ldots=\lambda_{K_s+K_{s-1}}=\mu_2;\\
&&\ldots\\
&&\lambda_{\sum_{t=2}^{s}K_t}=\lambda_{1+\sum_{t=2}^{s}K_t}=\ldots=\lambda_{\sum_{t=1}^{s}K_t}=\mu_s;
\end{eqnarray*}
and note that
\begin{eqnarray}
\mu_1>\ldots>\mu_s\label{mi}
\end{eqnarray}
are all the distinct elements in $\{\lambda_1,\lambda_2,\ldots,\lambda_N\}$. Let 
\begin{eqnarray}
J_i=\{t:1\leq p\leq N, 1\leq t\leq s, [\si_0(p)\mod n]=i,\lambda_p=\mu_t\}\label{ji}.
\end{eqnarray}
Suppose that all the following conditions hold
\begin{itemize}
\item If $1\leq i<j\leq n$, $\ell\in J_i$, and $t\in J_j$, then $\ell<t$; and
\item For any $p,q$ satisfying $1\leq p\leq s$ and $1\leq q\leq s$, and $q>p$.
\begin{eqnarray*}
C_1N \leq \mu_p-\mu_q\leq C_2N
\end{eqnarray*}
where $C_1$, $C_2$ are constants independent of $N$.
\end{itemize}
\end{assumption}

Let
\begin{eqnarray}
H_{\mathbf{m}_i}(u)=\int_{0}^{\ln(u)}R_{\mathbf{m}_i}(t)dt+\ln\left(\frac{\ln(u)}{u-1}\right)\label{hmi}
\end{eqnarray}
and $\mathbf{R}_{\mathbf{m}_i}$ is the Voiculescu R-transform of $\mathbf{m}_i$ given by
\begin{eqnarray*}
R_{\bm_i}=\frac{1}{S_{\bm_i}^{(-1)}(z)}-\frac{1}{z};
\end{eqnarray*}
Where $S_{\bm_i}$ is the moment generating function for $\bm_i$ given by
\begin{eqnarray*}
S_{\bm_i}(z)=z+M_1(\bm_i)z^2+M_2(\bm_i)z^3+\ldots;
\end{eqnarray*}
$M_k(\bm_i)=\int_{\RR}x^k\bm_i(dx)$; and $S_{\bm_i}^{-1}(z)$ is the inverse series of $S_{\bm_i}(z)$.
See also Section 2.2 of \cite{bg} for details.

\begin{proposition}\label{tm1}Suppose Assumptions \ref{ap32} and \ref{ap428} hold. Let $\kappa\in(0,1)$ be a positive number. Let $\rho_{\lfloor (1-\kappa)N\rfloor}$ be a probability measure on $\GT_{\lfloor (1-\kappa)N\rfloor}^+$ as defined in Lemma \ref{l33} or Remark \ref{rm25}. Let $\bm[\rho_{\lfloor (1-\kappa)N\rfloor}]$ be the corresponding random counting measure. Then as $N\rightarrow\infty$, $\bm[\rho_{\lfloor (1-\kappa)N\rfloor}]$ converge in probability, in the sense of moments to a deterministic measure $\bm^{\kappa}$, whose moments are given by
\begin{eqnarray*}
\int_{\RR}x^{p}\textbf{m}^{\kappa}(dx)=
\frac{1}{2n(p+1)\pi \mathbf{i}}\sum_{i=1}^{n}\oint_{C_{1}}\frac{dz}{z}\left(zQ_{i,\kappa}'(z)+\frac{n-i}{n}+\frac{z}{n(z-1)}\right)^{p+1}
\end{eqnarray*}
where for $1\leq i\leq n$
\begin{eqnarray*}
Q_{i,\kappa}(z)=\left\{\begin{array}{cc}\frac{1}{(1-\kappa)n}\left[
H_{\bm_i}(z)-(n-i)\log z+\kappa\sum_{r\in\{1,2,\ldots,n\}\cap I_2}\log\frac{1+y_rzx_1}{1+y_rx_1}\right]&\mathrm{if}\ i=1\\\frac{1}{(1-\kappa)n}\left[H_{\bm_i}(z)-(n-i)\log z\right]
&\mathrm{otherwise}\end{array}\right.
\end{eqnarray*}
and for $i\geq n+1$,
\begin{eqnarray*}
Q_{i,\kappa}(z)=\left\{\begin{array}{cc}Q_{(i\mod n),\kappa}(z),&\mathrm{if}\ (i\mod n)\neq 0\\Q_{n,\kappa}(z),&\mathrm{if}\ (i\mod n)= 0 \end{array}\right\}.
\end{eqnarray*}
\end{proposition}
\begin{proof}See Theorem 2.18 of \cite{Li18}.
\end{proof}

\begin{lemma}\label{p437}Let $k$ be a positive integer such that $1\leq k\leq N$. Let 
\begin{eqnarray*}
w_{i}=\left\{\begin{array}{cc}u_i&\mathrm{if}\ 1\leq i\leq k\\x_i&\mathrm{if}\ k+1\leq i\leq N\end{array}\right.
\end{eqnarray*}
Assume 
\begin{eqnarray*}
k=qn+r,\qquad \mathrm{where}\ r<n,
\end{eqnarray*}
and $q,r$ are positive integers.
Then the Schur function can be computed by the following formula
\begin{eqnarray}
&&\label{s0s}s_{\lambda}(w_1,\ldots,w_N)\\
&=&\sum_{\ol{\sigma}\in[\Sigma_N/\Sigma_N^X]^r} \left(\prod_{i=1}^{n}x_i^{|\phi^{(i,\sigma)}(N)|}\right)\left(\prod_{i=1}^{r}s_{\phi^{(i,\sigma)}(N)}\left(\frac{u_i}{x_i},\frac{u_{n+i}}{x_i}\ldots,\frac{u_{qn+i}}{x_i},1,\ldots,1\right)\right)\notag\\
&&\times\left(\prod_{i=r+1}^{n}s_{\phi^{(i,\sigma)}(N)}\left(\frac{u_i}{x_i},\frac{u_{n+i}}{x_i}\ldots,\frac{u_{(q-1)n+i}}{x_i},1,\ldots,1\right)\right)\notag\\
&&\times\left(\prod_{i<j,x_{\sigma(i)}\neq x_{\sigma(j)}}\frac{1}{w_{\sigma(i)}-w_{\sigma(j)}}\right)\notag
\end{eqnarray}
where $\sigma\in \ol{\sigma}\cap \Sigma_N$ is a representative. 
\end{lemma}
\begin{proof}See Corollary 3.4 of \cite{Li18}.
\end{proof}

\begin{theorem}\label{tm2}Under Assumptions \ref{ap32} and \ref{ap428}, for each given $\{a_i,b_i\}_{i=1}^{n}$, when $\alpha$ in Assumption \ref{ap32} is sufficiently large and $k\leq n$ we have
\begin{eqnarray}
&&\lim_{N\rightarrow\infty}\frac{1}{N}\log \frac{s_{\lambda(N)}(u_1x_{1,N},\ldots,u_kx_{k,N},x_{k+1,N},\ldots,x_{N,N})}{s_{\lambda(N)}(x_{1,N},\ldots,x_{N,N})}=\sum_{i=1}^{k}[Q_i(u_i)]\label{fc}
\end{eqnarray}
where for $1\leq i\leq k$, 
\begin{enumerate}
\item if $[i\mod n]\neq 0$,
\begin{eqnarray*}
Q_i(u)=\frac{H_{\mathbf{m}_{i\mod n}}(u)}{n}-\frac{(n-[i\mod n])\log(u)}{n}.
\end{eqnarray*}
and the convergence of (\ref{fc}) is uniform when $u_1,\ldots,u_k$ are in an open complex neighborhood of $1$.
\item if $[i\mod n]=0$,
\begin{eqnarray*}
Q_i(u)=\frac{H_{\mathbf{m}_n}(u)}{n}.
\end{eqnarray*}
\end{enumerate}
\end{theorem}
\begin{proof}See Theorem 2.8 of \cite{Li18}.
\end{proof}

\begin{lemma}\label{l440}Let $\si_0$ satisfy (\ref{sz}), and let $\ol{\si}_0\in [\Si_N/ \Si_N^X]^r$. For $1\leq i\leq k$, assume $\frac{u_i}{x_i}$ is in an open complex neighborhood of $1$.  For any $\si\in \Si_N$, let
\begin{eqnarray*}
&&L_{\si}\left(\frac{u_1}{x_1},\ldots,\frac{u_k}{x_k}\right)\\
&=& \left(\prod_{i=1}^{n}x_i^{|\phi^{(i,\sigma)}(N)|}\right)\left(\prod_{i=1}^{r}s_{\phi^{(i,\sigma)}(N)}\left(\frac{u_i}{x_i},\frac{u_{n+i}}{x_i}\ldots,\frac{u_{qn+i}}{x_i},1,\ldots,1\right)\right)\notag\\
&&\times\left(\prod_{i=r+1}^{n}s_{\phi^{(i,\sigma)}(N)}\left(\frac{u_i}{x_i},\frac{u_{n+i}}{x_i}\ldots,\frac{u_{(q-1)n+i}}{x_i},1,\ldots,1\right)\right)\notag\\
&&\times\left(\prod_{i<j,x_{\sigma(i)}\neq x_{\sigma(j)}}\frac{1}{w_{\sigma(i)}-w_{\sigma(j)}}\right)
\end{eqnarray*}
Suppose that Assumption \ref{ap32} holds. When $\alpha$ in Assumption \ref{ap32} is sufficiently large, we have
\begin{eqnarray*}
\left|\frac{L_{\si_0}}{L_{\si}}\right|\geq e^{CN}
\end{eqnarray*}
where $C>0$ is a constant independent of $\si$, $N$ and $(u_1,\ldots,u_k)$. Moreover,
\begin{eqnarray*}
\lim_{\alpha\rightarrow+\infty} C=+\infty.
\end{eqnarray*}
\end{lemma}
\begin{proof}See Lemma 4.5 of \cite{Li18}.
\end{proof}

\section{Central Limit Theorem for Piecewise Boundary Conditions}\label{pbc}

In this section, we construct certain statistics from the (random) dimer configuration on a contracting square hexagon lattice with piecewise boundary conditions, and show that they converge in distribution to sum of $n$ independent Gaussian random variables in the scaling limit, where $1\times n$ is the size of a fundamental domain.  The main theorem proved in this section is Theorem \ref{clt2}.

\subsection{First order moments}
For the piecewise boundary conditions, the proof of Proposition \ref{pn41} still holds. For $\kappa\in(0,1)$, let 
\begin{eqnarray*}
X_N&=&(x_{1,N},\ldots,x_{N,N});\\
X_{N,\kappa}&=&(x_{1+N-\lfloor(1-\kappa) N\rfloor,N},\ldots,x_{N,N})\\
U_{N,\kappa}&=&(u_1,u_2,\ldots,u_{\lfloor (1-\kappa)N\rfloor})\\
U_{N,\kappa,X}&=&(u_1x_{1+N-\lfloor(1-\kappa) N\rfloor,N},u_2x_{2+N-\lfloor(1-\kappa) N\rfloor,N},\ldots,u_{\lfloor (1-\kappa)N\rfloor}x_{N,N})
\end{eqnarray*}

Let $\lambda\in\GT_{\lfloor (1-\kappa)N \rfloor}$, and $\rho_{\lfloor(1-\kappa)N \rfloor}$ be a probability measure on $\GT_{\lfloor(1-\kappa)N}\rfloor$ as defined in Proposition \ref{tm1}. Then we have
\begin{eqnarray*}
&&E_{k,\kappa,N}:=\mathbf{E}\sum_{i=1}^{\lfloor (1-\kappa)N\rfloor}(\lambda_i+\lfloor(1-\kappa)N\rfloor-i)^k\\
&=&\sum_{\lambda\in\GT_{\lfloor(1-\kappa) N\rfloor}}\rho_{\lfloor(1-\kappa)N \rfloor}(\lambda)\sum_{i=1}^{\lfloor(1-\kappa)N \rfloor}(\lambda+\lfloor(1-\kappa)N\rfloor-i)^k\\
&=&\left.\frac{1}{V_{\lfloor (1-\kappa)N\rfloor}(U_{N,\kappa,X})}\sum_{i=1}^{\lfloor(1-\kappa)N\rfloor}\left(u_i\frac{\partial}{\partial u_i}\right)^kV_{\lfloor(1-\kappa)N \rfloor}(U_{N,\kappa,X})\mathcal{S}_{\rho_{\lfloor(1-\kappa)N \rfloor},X_{N,\kappa}}(U_{N,\kappa,X})\right|_{U_{N,\kappa}=(1,\ldots,1)}\\
\end{eqnarray*}
For $1\leq i\leq n$, let
\begin{eqnarray*}
v_i=\begin{cases}x_{i,N},\ \mathrm{if}\ 1\leq i\leq N-\lfloor (1-\kappa)N\rfloor \\ x_{i,N}  u_{i-N+\lfloor (1-\kappa)N\rfloor},\ \mathrm{if}\ N-\lfloor (1-\kappa)N\rfloor+1\leq i\leq N \end{cases}
\end{eqnarray*}
Let $\lambda(N)$ be the partition corresponding to the boundary condition. For $1\leq i\leq n$, let 
\begin{eqnarray}
R(i)=\{1\leq j\leq \lfloor(1-\kappa)N \rfloor: [(j+N-\lfloor(1-\kappa)N \rfloor)\mod n]=[i\mod n]\}
\end{eqnarray}
 and for $1\leq i\leq \lfloor (1-\kappa)N\rfloor$ let 
\begin{eqnarray*}
j(i)=\begin{cases}[i+N-\lfloor (1-\kappa)N \rfloor]\mod n,\ \mathrm{if}\ \left([i+N-\lfloor (1-\kappa)N \rfloor]\mod n\right)\neq 0\\
n,\ \mathrm{if}\ \left([i+N-\lfloor (1-\kappa)N \rfloor]\mod n\right)=0
\end{cases}
\end{eqnarray*}
Assume 
\begin{eqnarray*}
\lfloor(1-\kappa)N\rfloor=q_{N,\kappa}n+r_{N,\kappa}
\end{eqnarray*}
where $q_{N,\kappa}$ and $r_{N,\kappa}$ are nonnegative integers satisfying $r_{N,\kappa}<n$.
 By Lemmas \ref{lmm212}, \ref{p437}, \ref{l440}, we obtain
\begin{eqnarray*}
&&E_{k,\kappa,N}\\
&=&\frac{1}{V_{\lfloor (1-\kappa)N\rfloor}(X_{N,\kappa})}\sum_{i=1}^{\lfloor(1-\kappa)N\rfloor}\left(u_i\frac{\partial}{\partial u_i}\right)^kV_{\lfloor(1-\kappa)N \rfloor}(U_{N,\kappa,X})
\left[\frac{s_{\lambda(N)}\left(U_{N,\kappa,X},x_{1,N},\ldots, x_{N-\lfloor (1-\kappa)N \rfloor,N}\right)}{s_{\lambda(N)}(X_N)}\right.\\
&&\left.\left.\prod_{l\in\{1,\ldots,N-\lfloor(1-\kappa)N \rfloor\}\cap I_2}\prod_{j=1}^{\lfloor(1-\kappa)N \rfloor}\left(\frac{1+y_{\ol{l}}u_jx_{\ol{N-\lfloor (1-\kappa)N\rfloor+j}}}{1+y_{\ol{l}}x_{\ol{N-\lfloor (1-\kappa)N\rfloor+j}}}\right)\right]\right|_{U_{N,\kappa}=(1,\ldots,1)}\\
&=&\left.\frac{1}{V_{\lfloor (1-\kappa)N\rfloor}(X_{N,\kappa})}\sum_{i=1}^{\lfloor(1-\kappa)N\rfloor}\left(u_i\frac{\partial}{\partial u_i}\right)^kV_{\lfloor(1-\kappa)N \rfloor}(U_{N,\kappa,X})\frac{T_N}{P_N}\right|_{U_{N,\kappa}=(1,\ldots,1)},
\end{eqnarray*}
where 
\begin{eqnarray*}
&&T_N=\prod_{l\in\{1,\ldots,N-\lfloor(1-\kappa)N \rfloor\}\cap I_2}\prod_{j=1}^{\lfloor(1-\kappa)N \rfloor}\left(\frac{1+y_{\ol{l}}x_{\ol{N-\lfloor (1-\kappa)N\rfloor+j}}u_j}{1+y_{\ol{l}}x_{\ol{N-\lfloor (1-\kappa)N\rfloor+j}}}\right)\\
&\times&\left(\prod_{i=1}^{r}s_{\phi^{(j(i),\sigma_0)}(N)}\left(u_{i},u_{n+i}\ldots,u_{q_{N,\kappa}n+i},1,\ldots,1\right)\right)\\
&&\times\left(\prod_{i=r+1}^{n}s_{\phi^{(j(i),\sigma_0)}(N)}\left(u_{i},u_{n+i}\ldots,u_{(q_{N,\kappa}-1)n+i},1,\ldots,1\right)\right)\\
&&\times\left(\prod_{i<j,x_{\sigma_0(i)}\neq x_{\sigma_0(j)}}\frac{1}{v_{\sigma_0(i)}-v_{\sigma_0(j)}}\right)\left(1+o(1)\right)
\end{eqnarray*}
and
\begin{eqnarray*}
P_N&=&\left(\prod_{i=1}^{n}s_{\phi^{(i,\sigma_0)}(N)}\left(1,\ldots,1\right)\right)\\
&&\left(\prod_{i<j,x_{\sigma_0(i)}\neq x_{\sigma_0(j)}}\frac{1}{x_{\sigma_0(i)}-x_{\sigma_0(j)}}\right)\left(1+o(1)\right)
\end{eqnarray*}
Then we have
\begin{eqnarray*}
E_{k,\kappa,N}=\sum_{j=1}^{n}E_{k,\kappa,N}^{(j)}.
\end{eqnarray*}
where
\begin{eqnarray*}
E_{k,\kappa,N}^{(j)}:&=&\lim_{[x_k\rightarrow (x_{k\mod n}),\mathrm{for}\ 1\leq k\leq N]}
\frac{1}{V_{\lfloor (1-\kappa)N\rfloor}(X_{N,\kappa})}\\
&&\left.\sum_{i\in\{1,\ldots,\lfloor(1-\kappa)N\rfloor\}\cap R(j)}\left(u_i\frac{\partial}{\partial u_i}\right)^kV_{\lfloor(1-\kappa)N \rfloor}(U_{N,\kappa,X})\frac{T_N}{P_N}\right|_{U_{N,\kappa}=(1,\ldots,1)}\\
&=&\lim_{[x_k\rightarrow (x_{k\mod n}),\mathrm{for}\ 1\leq k\leq N]}\left.\sum_{i\in\{1,\ldots,\lfloor(1-\kappa)N\rfloor\}\cap R(j)}\left(u_i\frac{\partial}{\partial u_i}\right)^k\frac{T_{N,i}}{P_{N,i}}\right|_{U_{N,\kappa}=(1,\ldots,1)};
\end{eqnarray*}

\begin{eqnarray*}
T_{N,i}&=&\left[\prod_{l\in\{1,\ldots,N-\lfloor(1-\kappa)N \rfloor\}\cap I_2}\left(\frac{1+y_{\ol{l}}x_{j(i)}u_i}{1+y_{\ol{l}}x_{j(i)}}\right)\right]s_{\phi^{(j(i),\sigma_0)}(N)}\left(u_{i},1,\ldots,1\right)\\
&&\times\left(\prod_{j(i)<k\leq n }\left(\frac{1}{u_ix_{j(i)}-x_k}\right)^{\frac{\kappa N}{n}}\right)\left(\prod_{k<j(i)\leq n}\left(\frac{1}{x_k-u_ix_{j(i)}}\right)^{\frac{\kappa N}{n}}\right)\\
&&\left(\prod_{N- \lfloor(1-\kappa)N\rfloor+1\leq k\leq N,\ k\mod n=j(i),\ k-N+\lfloor(1-\kappa)N \rfloor\neq i}\left[u_ix_{j(i)}-x_{k}\right]\right)e^{N o(1)}
\end{eqnarray*}
and
\begin{eqnarray*}
P_{N,i}&=&s_{\phi^{(j(i),\sigma_0)}(N)}\left(1,\ldots,1\right)\\
&&\times\left(\prod_{j(i)<k\leq n }\left(\frac{1}{x_{j(i)}-x_k}\right)^{\frac{\kappa N}{n}}\right)\left(\prod_{k<j(i)\leq n}\left(\frac{1}{x_k-x_{j(i)}}\right)^{\frac{\kappa N}{n}}\right)\\
&&\left(\prod_{N- \lfloor(1-\kappa)N\rfloor+1\leq k\leq N,\ k\mod n=j(i),\ k-N+\lfloor(1-\kappa)N \rfloor\neq i}\left[x_{j(i)}-x_{k}\right]\right)e^{No(1)}.
\end{eqnarray*}
In the expressions above, the $o(1)$ terms converge to 0 uniformly when $u_i$ is in a neighborhood of $1$.

If the edge weights satisfy Assumption \ref{ap32}, we obtain
\begin{eqnarray*}
E_{k,\kappa,N}^{(j)}:=\lim_{[x_k\rightarrow (x_{k\mod n}),\mathrm{for}\ 1\leq k\leq N]}\left.\sum_{i\in\{1,\ldots,\lfloor(1-\kappa)N\rfloor\}\cap R(j)}\left(u_i\frac{\partial}{\partial u_i}\right)^k\frac{\tilde{V}_{N,i}\tilde{T}_{N,i}}{\tilde{W}_{N,i}\tilde{P}_{N,i}}\right|_{U_{N,\kappa}=(1,\ldots,1)}
\end{eqnarray*}
where
\begin{eqnarray*}
\tilde{T}_{N,i}&=&\left[\prod_{l\in\{1,\ldots,N-\lfloor(1-\kappa)N \rfloor\}\cap I_2}\left(\frac{1+y_{\ol{l}}x_{j(i)}u_i}{1+y_{\ol{l}}x_{j(i)}}\right)\right]s_{\phi^{(j(i),\sigma_0)}(N)}\left(u_{i},1,\ldots,1\right)\\
&&\times\left(\prod_{j(i)<k\leq n }\left(\frac{1}{u_ix_{j(i)}-x_k}\right)^{\frac{\kappa N}{n}}\right)e^{No(1)};
\end{eqnarray*}
\begin{eqnarray*}
\tilde{P}_{N,i}&=&s_{\phi^{(j(i),\sigma_0)}(N)}\left(1,\ldots,1\right)\left(\prod_{j(i)<k \leq n}\left(\frac{1}{x_{j(i)}-x_k}\right)^{\frac{\kappa N}{n}}\right)e^{No(1)};
\end{eqnarray*}
\begin{eqnarray*}
\tilde{V}_{N,i}&=&\prod_{[N- \lfloor(1-\kappa)N\rfloor+1\leq k\leq N,\ (k\mod n)=(j(i)\mod n)],\ k-N+\lfloor(1-\kappa)N \rfloor\neq i}\left[u_ix_{j(i)}-x_{k}\right]\\
\tilde{W}_{N,i}&=&\prod_{[N- \lfloor(1-\kappa)N\rfloor+1\leq k\leq N,\ (k\mod n)=(j(i)\mod n)],\ k-N+\lfloor(1-\kappa)N \rfloor\neq i}\left[x_{j(i)}-x_{k}\right]
\end{eqnarray*}
Note that
\begin{eqnarray}
&&E_{k,\kappa,N}^{(j)}\label{ekj}\\
&=&\frac{1}{\tilde{W}_{N,i}\tilde{P}_{N,i}}\left.\sum_{i\in\{1,\ldots,\lfloor(1-\kappa)N\rfloor\}\cap R(j)}\left(u_i\frac{\partial}{\partial u_i}\right)^k\tilde{V}_{N,i}\exp\left[\log\left(\tilde{T}_{N,i}\right)\right]\right|_{U_{N,\kappa}=(1,\ldots,1)}\notag
\end{eqnarray}
and
\begin{eqnarray*}
\frac{\partial \tilde{T}_{N,i}}{\partial u_i}=\frac{\partial}{\partial u_i}\exp\left[\log\left(\tilde{T}_{N,i}\right)\right]=\exp\left[\log\left(\tilde{T}_{N,i}\right)\right]\frac{\partial}{\partial u_i}\left[\log\left(\tilde{T}_{N,i}\right)\right]
\end{eqnarray*}

\begin{lemma}\label{l51}Assume $\kappa \in (0,1)$ and the edge weights satisfy Assumption \ref{ap32}, then
\item For $1\leq i\leq \lfloor (1-\kappa)N\rfloor$ and $i\in R(j)$
\begin{eqnarray*}
&&\lim_{N\rightarrow\infty}\frac{1}{N}\frac{\partial [\log\tilde{T}_{N,i}]}{\partial u_i}=\frac{\kappa}{n}\sum_{l\in I_2\cap\{1,2,\ldots,n\}}\frac{y_lx_{j(i)}}{1+y_l x_{j(i)}u_i}-\frac{\kappa}{n}\frac{n-j(i)}{u_i}+\frac{1}{n} H_{\bm_{j(i)}}'(u_i);
\end{eqnarray*}
and the convergence is uniform when $u_i$ is in a neighborhood of $1$.
\end{lemma}
\begin{proof}The lemma follows from explicit computations and Theorem 3.6 of \cite{bk}; see also \cite{bg,GP15,GM05}.
\end{proof}

 \subsection{Second order moments}
 
Let
\begin{eqnarray*}
&&E_{k,\ell,\kappa,N}:=\mathbf{E}\left(\sum_{i=1}^{\lfloor (1-\kappa)N\rfloor}(\lambda_i+\lfloor(1-\kappa)N\rfloor-i)^k\sum_{j=1}^{\lfloor (1-\kappa)N\rfloor}(\lambda_j+\lfloor(1-\kappa)N\rfloor-j)^l\right)\\
&=&\sum_{\lambda\in\GT_{\lfloor(1-\kappa) N\rfloor}}\rho_{\lfloor(1-\kappa)N \rfloor}(\lambda)\sum_{i=1}^{\lfloor(1-\kappa)N \rfloor}(\lambda_i+\lfloor(1-\kappa)N\rfloor-i)^k\sum_{j=1}^{\lfloor (1-\kappa)N\rfloor}(\lambda_j+\lfloor(1-\kappa)N\rfloor-j)^l\\
&=&\frac{1}{V_{\lfloor (1-\kappa)N\rfloor}(U_{N,\kappa,X})}\sum_{i=1}^{\lfloor(1-\kappa)N\rfloor}\left(u_i\frac{\partial}{\partial u_i}\right)^k\sum_{j=1}^{\lfloor(1-\kappa)N\rfloor}\left(u_j\frac{\partial}{\partial u_j}\right)^l\\
&&\left.V_{\lfloor(1-\kappa)N \rfloor}(U_{N,\kappa,X})\mathcal{S}_{\rho_{\lfloor(1-\kappa)N \rfloor},X_{N,\kappa}}(U_{N,\kappa,X})\right|_{U_{N,\kappa}=(1,\ldots,1)}\\
\end{eqnarray*}
Again by Lemmas \ref{lmm212}, \ref{p437}, \ref{l440}, we obtain
\begin{eqnarray*}
&&E_{k,\ell,\kappa,N}\\
&=&\frac{1}{V_{\lfloor (1-\kappa)N\rfloor}(X_{N,\kappa})}\sum_{i=1}^{\lfloor(1-\kappa)N\rfloor}\left(u_i\frac{\partial}{\partial u_i}\right)^k\sum_{j=1}^{\lfloor(1-\kappa)N\rfloor}\left(u_j\frac{\partial}{\partial u_j}\right)^l V_{\lfloor(1-\kappa)N \rfloor}(U_{N,\kappa,X})\\
&&\left[\frac{s_{\lambda(N)}\left(U_{N,\kappa,X},x_{1,N},\ldots, x_{N-\lfloor (1-\kappa)N \rfloor,N}\right)}{s_{\lambda(N)}(X_N)}\right.\\
&&\left.\left.\prod_{l\in\{1,\ldots,N-\lfloor(1-\kappa)N \rfloor\}\cap I_2}\prod_{j=1}^{\lfloor(1-\kappa)N \rfloor}\left(\frac{1+y_{\ol{l}}u_jx_{\ol{N-\lfloor (1-\kappa)N\rfloor+j}}}{1+y_{\ol{l}}x_{\ol{N-\lfloor (1-\kappa)N\rfloor+j}}}\right)\right]\right|_{U_{N,\kappa}=(1,\ldots,1)}\\
&=&\left.\frac{1}{V_{\lfloor (1-\kappa)N\rfloor}(X_{N,\kappa})}\sum_{i=1}^{\lfloor(1-\kappa)N\rfloor}\left(u_i\frac{\partial}{\partial u_i}\right)^k\sum_{j=1}^{\lfloor(1-\kappa)N\rfloor}\left(u_j\frac{\partial}{\partial u_j}\right)^lV_{\lfloor(1-\kappa)N \rfloor}(U_{N,\kappa,X})\frac{T_N}{P_N}\right|_{U_{N,\kappa}=(1,\ldots,1)},
\end{eqnarray*}
Then we have
\begin{eqnarray*}
E_{k,l,\kappa,N}=\sum_{s=1}^{n}\sum_{t=1}^{n}E_{k,l,\kappa,N}^{(s,t)};
\end{eqnarray*}
where
\begin{eqnarray}
E_{k,l,\kappa,N}^{(s,t)}:&=&\lim_{[x_k\rightarrow (x_{k\mod n}),\mathrm{for}\ 1\leq k\leq N]}
\frac{1}{V_{\lfloor (1-\kappa)N\rfloor}(X_{N,\kappa})}\label{e2}\\
&&\sum_{[i\in\{1,\ldots,\lfloor(1-\kappa)N\rfloor\}\cap R(s)]}\sum_{[r\in\{1,\ldots,\lfloor(1-\kappa)N\rfloor\}\cap R(t)]}\left(u_i\frac{\partial}{\partial u_i}\right)^k\left(u_r\frac{\partial}{\partial u_r}\right)^l
\notag\\&&\left.V_{\lfloor(1-\kappa)N \rfloor}(U_{N,\kappa,X})\frac{T_N}{P_N}\right|_{U_{N,\kappa}=(1,\ldots,1)}\notag\\
&=&\lim_{[x_k\rightarrow (x_{k\mod n}),\mathrm{for}\ 1\leq k\leq N]}\sum_{i\in\{1,\ldots,\lfloor(1-\kappa)N\rfloor\}\cap R(s)}\left(u_i\frac{\partial}{\partial u_i}\right)^k\notag\\
&&\sum_{[r\in\{1,\ldots,\lfloor(1-\kappa)N\rfloor\}\cap R(t)]}\left(u_r\frac{\partial}{\partial u_r}\right)^l\left.\frac{V_{N,(i,r)}^{(s,t)}T_{N,(i,r)}^{(s,t)}}{W_{N,(i,r)}^{(s,t)}P_{N,(i,r)}^{(s,t)}}\right|_{U_{N,\kappa}=(1,\ldots,1)}\notag
\end{eqnarray}
and 
where (assume the edge weights satisfy Assumption \ref{ap32})
\begin{enumerate}
\item if $s=t$,
\begin{eqnarray*}
T_{N,(i,r)}^{(s,s)}&=&
\left[\prod_{l\in\{1,\ldots,N-\lfloor(1-\kappa)N \rfloor\}\cap I_2}\left(\frac{1+y_{\ol{l}}x_{s}u_i}{1+y_{\ol{l}}x_{s}}\right)\right]
\left[\prod_{l\in\{1,\ldots,N-\lfloor(1-\kappa)N \rfloor\}\cap I_2}\left(\frac{1+y_{\ol{l}}x_{s}u_r}{1+y_{\ol{l}}x_{s}}\right)\right]\\
&&s_{\phi^{(s,\sigma_0)}(N)}\left(u_{i},u_j,1,\ldots,1\right)\\
&&\times\left(\prod_{s<k\leq n }\left(\frac{1}{u_ix_{s}-x_k}\right)^{\frac{\kappa N}{n}}\right)\left(\prod_{s<k\leq n }\left(\frac{1}{u_rx_{s}-x_k}\right)^{\frac{\kappa N}{n}}\right)e^{No(1)};\\
P_{N,(i,r)}^{(s,s)}&=&s_{\phi^{(s,\sigma_0)}(N)}\left(1,\ldots,1\right)\left(\prod_{s<k \leq n}\left(\frac{1}{x_{s}-x_k}\right)^{\frac{2\kappa N}{n}}\right)e^{No(1)}\\
V_{N,(i,r)}^{(s,s)}&=&\left(\prod_{[N- \lfloor(1-\kappa)N\rfloor+1\leq k\leq N,\ (k\mod n)=(s\mod n)],\ k-N+\lfloor(1-\kappa)N \rfloor\notin \{i,r\}}\left[u_ix_{s}-x_{k}\right]\right)\\
&&\times  \left(\prod_{[N- \lfloor(1-\kappa)N\rfloor+1\leq k\leq N,\ (k\mod n)=(s\mod n)],\ k-N+\lfloor(1-\kappa)N \rfloor\notin \{i,r\}}\left[u_rx_{s}-x_{k}\right]\right)\\
&&\times (u_ix_s-u_rx_s)\\
W_{N,(i,r)}^{(s,s)}&=&\prod_{[N- \lfloor(1-\kappa)N\rfloor+1\leq k\leq N,\ (k\mod n)=(s\mod n)],\ k-N+\lfloor(1-\kappa)N \rfloor\notin \{i,r\}}\left[x_{s}-x_{k}\right]^2\\
&&\times \left(x_{i+N-\lfloor(1-\kappa)N\rfloor}-x_{r+N-\lfloor (1-\kappa)N\rfloor}\right)
\end{eqnarray*}
In the expressions above, the $o(1)$ terms converge to 0 uniformly when $u_i$, $u_r$ are in a neighborhood of $1$ as $N\rightarrow\infty$; moreover, the operator $\frac{\partial^2}{\partial u_r\partial u_s}$ acting on $\log o(1)$ is identically 0 (instead of converging to $0$ as $N\rightarrow\infty$).
\item if $s\neq t$,
\begin{eqnarray*}
T_{N,(i,r)}^{(s,t)}&=&
\left[\prod_{l\in\{1,\ldots,N-\lfloor(1-\kappa)N \rfloor\}\cap I_2}\left(\frac{1+y_{\ol{l}}x_{s}u_i}{1+y_{\ol{l}}x_{s}}\right)\right]
\left[\prod_{l\in\{1,\ldots,N-\lfloor(1-\kappa)N \rfloor\}\cap I_2}\left(\frac{1+y_{\ol{l}}x_{s}u_r}{1+y_{\ol{l}}x_{s}}\right)\right]\\
&&s_{\phi^{(s,\sigma_0)}(N)}\left(u_{i},1,\ldots,1\right)s_{\phi^{(t,\sigma_0)}(N)}\left(u_r,1,\ldots,1\right)\\
&&\times\left(\prod_{s<k\leq n}\left(\frac{1}{u_ix_{s}-x_k}\right)^{\frac{\kappa N}{n}}\right)\left(\prod_{t<k\leq n}\left(\frac{1}{u_rx_{t}-x_k}\right)^{\frac{\kappa N}{n}}\right) \left(\frac{1}{u_ix_{s}-u_rx_t}\right)e^{No(1)};\\
P_{N,(i,r)}^{(s,t)}&=&\tilde{P}_{N,i}\tilde{P}_{N,r}e^{No(1)}\\
V_{N,(i,r)}^{(s,t)}&=&\tilde{V}_{N,i}\tilde{V}_{N,r}\\
W_{N,(i,r)}^{(s,t)}&=&\tilde{W}_{N,i}\tilde{W}_{N,r}
\end{eqnarray*}
In the expressions above, the $o(1)$ terms converge to 0 uniformly when $u_i$, $u_r$ are in a neighborhood of $1$ as $N\rightarrow\infty$; moreover, the operator $\frac{\partial^2}{\partial u_r\partial u_s}$ acting on $\log o(1)$ is identically 0 (instead of converging to $0$ as $N\rightarrow\infty$).
\end{enumerate}

\begin{lemma}\label{l52}Assume $\kappa \in (0,1)$ and the edge weights satisfy Assumption \ref{ap32}, then
\begin{enumerate}
\item For $1\leq i<j\leq \lfloor (1-\kappa)N\rfloor$ and $i,j\in R(s)$
\begin{eqnarray*}
&&\lim_{N\rightarrow\infty}\frac{\partial^2 [\log T^{(s,s)}_{N,(i,r)}]}{\partial u_i\partial u_r}\\
&=&\frac{\partial^2}{\partial u_i u_r}\left[\log\left(1-(u_i-1)(u_r-1)\frac{u_i H_{\bm_s}'(u_i)-u_r H'_{\bm_s}(u_r)}{u_i-u_r}\right)\right]
\end{eqnarray*}
and the convergence is uniform when $u_i$ is in a neighborhood of $1$.
\item For $1\leq i<j\leq \lfloor (1-\kappa)N\rfloor$ and $i\in R(s)$, $j\in R(t)$ with $s\neq t$
\begin{eqnarray*}
&&\lim_{N\rightarrow\infty}\frac{\partial^2 [\log T^{(s,t)}_{N,(i,r)}]}{\partial u_i\partial u_s}=0.
\end{eqnarray*}
and the convergence is uniform when $u_i$ is in a neighborhood of $1$.
\end{enumerate}
\end{lemma}
\begin{proof}First we consider the case that $i,j\in R(s)$, we have
\begin{eqnarray*}
&&\lim_{N\rightarrow\infty}\frac{\partial^2 [\log T^{(s,s)}_{N,(i,r)}]}{\partial u_i\partial u_r}
=\lim_{N\rightarrow\infty}\frac{\partial^2 [\log s_{\phi^{(s,\sigma_0)}(N)}\left(u_{i},u_j,1,\ldots,1\right)]}{\partial u_i\partial u_r}
\end{eqnarray*}
Then Part (1) of the lemma follows from Theorem 6.8 of \cite{bk}; see also \cite{GP15,bg16}.

Now we consider the case that $i\in R(s)$, $j\in R(t)$ and $s\neq t$. In this case
\begin{eqnarray*}
\lim_{N\rightarrow\infty}\frac{\partial^2 [\log T^{(s,s)}_{N,(i,r)}]}{\partial u_i\partial u_r}
&=&\lim_{N\rightarrow\infty}\frac{\partial^2 [-\log(u_ix_{s,N}-u_rx_{t,N})]}{\partial u_i\partial u_r}\\
&=&\lim_{N\rightarrow\infty}\frac{x_{s,N}x_{t,N}}{(u_ix_{s,N}-u_rx_{t,N})^2};
\end{eqnarray*}
and the limit is 0 by Assumption \ref{ap32}.
\end{proof}

\subsection{Asymptotical analysis}

Let
\begin{eqnarray}
F_{k,\kappa,N}^{(s)}&=&\lim_{[x_k\rightarrow (x_{k\mod n}),\mathrm{for}\ 1\leq k\leq N]}
\frac{1}{V_{\lfloor (1-\kappa)N\rfloor}(U_{N,\kappa,X})T_N}\label{fkn}\\
&&\sum_{[i\in\{1,\ldots,\lfloor(1-\kappa)N\rfloor\}\cap R(s)]}\left(u_i\frac{\partial}{\partial u_i}\right)^kV_{\lfloor(1-\kappa)N \rfloor}(U_{N,\kappa,X})T_N\notag
\end{eqnarray}

For simplicity, we use the notation $\partial_i$ to denote $\frac{\partial}{\partial u_i}$.
Expanding the right hand side of (\ref{ekj}), we can write $E_{k,\kappa,N}^{(j)}$ as a sum of terms with the following form
\begin{eqnarray}
\frac{c_0 x_{j(i)}^r (\partial_i^{s_1}[\log\tilde{T}_{N,i}])^{d_1}\ldots (\partial_i^{s_t}[\log\tilde{T}_{N,i}])^{d_t}}{(x_{j(i)}-x_{a_1})\ldots(x_{j(i)}-x_{a_r})}\label{rt}
\end{eqnarray}
Similarly, we can write the right hand side of (\ref{fkn}) as a large sum of the following form
\begin{eqnarray*}
\frac{c_0 x_{j(i)}^r u_i^{k-s_0}(\partial_i^{s_1}[\log T_N])^{d_1}\ldots (\partial_i^{s_t}[\log T_N])^{d_t}}{(x_{j(i)}u_i-x_{a_1}u_{a_1-N+\lfloor (1-\kappa)N\rfloor})\ldots(x_{j(i)}u_i-x_{a_r}u_{a_r-N+\lfloor (1-\kappa)N\rfloor})}
\end{eqnarray*}
such that
\begin{itemize}
\item for $1\leq s\leq r$, $a_s$ is a positive integer satisfying $N-\lfloor (1-\kappa)N\rfloor+1\leq a_s\leq N$; and
\item In (\ref{rt}), we have $(a_s\mod n)=(j(i)\mod n)$ for all $1\leq s\leq r$; and
\item $N-\lfloor (1-\kappa)N\rfloor+i, a_1,\ldots,a_r$ are distinct; and
\item $\{s_j\}_{j=0}^{t}$ and $\{d_j\}_{j=1}^{t}$ are nonnegative integers satisfying $s_1<s_2<\ldots <s_t$; and
\item 
\begin{eqnarray}
r+s_0+s_1d_1+\ldots+s_td_t=k; \label{rsd}
\end{eqnarray}
and
\item $c_0$ is a constant independent of $N$ and $a_1,\ldots,a_r$.
 \end{itemize}
 From the expression (\ref{ekj}) we see that any term obtained by permuting $i,a_1-N+\lfloor(1-\kappa)N\rfloor,\ldots,a_r-N+\lfloor (1-\kappa)N\rfloor$ of (\ref{rt}) with in $\{1,2,\ldots,\lfloor(1-\kappa)N \rfloor\}\cap R(j)$ are still present in the sum. Let
 \begin{eqnarray*}
 \tilde{a}_1=a_1-N+\lfloor (1-\kappa)N\rfloor
\end{eqnarray*} 
 Hence we have
 \begin{eqnarray*}
E_{k,\kappa,N}^{(j)}&=&\sum_{r,\{s_j\}_{j=0}^{t},\{d_j\}_{j=1}^{t}\ \mathrm{satisfy}\ (\ref{rsd})}(r+1)!
 \sum_{[\{a_1-N+\lfloor(1-\kappa)N \rfloor,\ldots,a_{r+1}-N+\lfloor(1-\kappa)N\rfloor\}\in \{1,2,\ldots,\lfloor(1-\kappa)N \rfloor\}\cap R(j)]}\\
 &&\lim_{[x_{a_1},\ldots,x_{a_{r+1}}\longrightarrow x_j]}\mathrm{Sym}_{a_1,\ldots,a_{r+1}}\\
 &&\left.\left[\frac{c_0 x_{a_1}^r u_{\tilde{a}_1}^{k-s_0}(\partial_{\tilde{a}_1}^{s_1}[\log\tilde{T}_{N,\tilde{a}_1}])^{d_1}\ldots (\partial_{\tilde{a}_1}^{s_t}[\log\tilde{T}_{N,\tilde{a}_1}])^{d_t}}{(x_{a_1}-x_{a_2})\ldots(x_{a_1}-x_{a_{r+1}})}\right]\right|_{U_{N,\kappa}=\{1,\ldots,1\}}.
 \end{eqnarray*}
 where the constant $c_0$ may depend on $r$, $\{s_j\}_{j=0}^t$ and $\{d_j\}_{j=1}^{t}$.

\begin{lemma}\label{le53} Let
\begin{eqnarray*}
&&\tilde{T}_N=\prod_{l\in\{1,\ldots,N-\lfloor(1-\kappa)N \rfloor\}\cap I_2}\prod_{j=1}^{\lfloor(1-\kappa)N \rfloor}\left(\frac{1+y_{\ol{l}}x_{\ol{N-\lfloor (1-\kappa)N\rfloor+j}}u_j}{1+y_{\ol{l}}x_{\ol{N-\lfloor (1-\kappa)N\rfloor+j}}}\right)\\
&\times&\left(\prod_{i=1}^{r}s_{\phi^{(j(i),\sigma_0)}(N)}\left(u_{i},u_{n+i}\ldots,u_{q_{N,\kappa}n+i},1,\ldots,1\right)\right)\\
&&\times\left(\prod_{i=r+1}^{n}s_{\phi^{(j(i),\sigma_0)}(N)}\left(u_{i},u_{n+i}\ldots,u_{(q_{N,\kappa}-1)n+i},1,\ldots,1\right)\right)\\
&&\times\left(\prod_{N-\lfloor(1-\kappa)N\rfloor+1\leq i\leq N,\ 1\leq j\leq N-\lfloor(1-\kappa)N\rfloor,\ \tilde{i}\in R(p),\tilde{j}\in R(q),p<q}\frac{1}{x_iu_{\tilde{i}}-x_j}\right)\left(1+o(1)\right)
\end{eqnarray*}
and
\begin{eqnarray*}
\tilde{V}_{\lfloor (1-\kappa)N\rfloor}=\prod_{k=1}^{n}\prod_{[1+N-\lfloor(1-\kappa)N \rfloor\leq i<j\leq N, (i\mod n)=(j\mod n)=(k\mod n)]}(x_iu_{\tilde{i}}-x_ju_{\tilde{j}})
\end{eqnarray*}
Assume $a,b,c\in R(j)$, and $a,b,c$ are distinct positive integers. Assume that the $o(1)$ in the definition of $\tilde{T}_N$ converges to $0$ uniformly when $u_a,u_b,u_c$ are in a neighborhood of 1. Then
\begin{eqnarray*}
\left.\lim_{N\rightarrow\infty}\frac{\partial^3\log[\tilde{T}_N]}{\partial u_a\partial u_b\partial u_c}\right|_{U_{N,\kappa}=(1,\ldots,1)}=0.
\end{eqnarray*}
\end{lemma} 

\begin{proof}We have
\begin{eqnarray*}
\left.\lim_{N\rightarrow\infty}\frac{\partial^3\log[\tilde{T}_N]}{\partial u_a\partial u_b\partial u_c}\right|_{U_{N,\kappa}=(1,\ldots,1)}=\lim_{N\rightarrow\infty}\frac{\partial^3\log[s_{\phi^{(j,\sigma_0)}(N)}\left(u_a,u_b,u_c,1,\ldots,1\right)]}{\partial u_a\partial u_b\partial u_c}
\end{eqnarray*}
To compute the derivative on the right hand side, note that 
\begin{eqnarray*}
\lim_{N\rightarrow \infty}\bm[\phi^{(j,\sigma_0)}(N)]=\bm_j
\end{eqnarray*}
By Theorem 6.8 of \cite{bk} (see also \cite{GP15}, \cite{bg16})，
\begin{eqnarray*}
&&\lim_{N\rightarrow\infty}\frac{\partial^2\log[s_{\phi^{(j,\sigma_0)}(N)}\left(u_a,u_b,u_c,1,\ldots,1\right)]}{\partial u_a\partial u_b}\\
&=&\frac{\partial^2}{\partial u_a\partial u_b}\log\left(1-(u_a-1)(u_b-1)\frac{u_a H_{\bm_j}'(u_a)-u_bH_{\bm_j}'(u_b)}{u_a-u_b}\right);
\end{eqnarray*}
and the convergence is uniform when $(u_a,u_b,u_c)$ is in a neighborhood of $(1,1,1)$. Therefore we can take the derivative with respect to $u_c$ on both sides. The right hand side is independent of $u_c$, and the derivative with respect to $u_c$ is 0. Then the lemma follows.
\end{proof}

 \begin{lemma}\label{l53}Let $\kappa\in (0,1)$ and the edge weights satisfy Assumption \ref{ap32}. Then 
  \begin{enumerate}
  \item  the degree of $N$ in $E_{k,\kappa,N}^{(j)}$ is at most $k+1$.
 \item For any integer $i$ satisfying $1\leq i\leq \lfloor (1-\kappa)N\rfloor$ and $i\in R(j)$, the degree of $N$ in $\left.\frac{\partial}{\partial u_i} F_{k,\kappa,N}^{(j)}\right|_{U_{N,\kappa}=(1,\ldots,1)}$ is at most $k$;
   \item For any integer $i$ satisfying $1\leq i\leq \lfloor (1-\kappa)N\rfloor$ and $i\notin R(j)$, the degree of $N$ in $\left.\frac{\partial}{\partial u_i} F_{k,\kappa,N}^{(j)}\right|_{U_{N,\kappa}=(1,\ldots,1)}$ is less than $k$;
 \item For any integers $i_1,i_2$ satisfying $1\leq i_1<i_2\leq \lfloor (1-\kappa)N\rfloor$ and $i_1,i_2\in R(j)$, the degree of $N$ in $\left.\frac{\partial^2}{\partial u_{i_1}\partial u_{i_2}} F_{k,\kappa,N}^{(j)}\right|_{U_{N,\kappa}=(1,\ldots,1)}$ is less than $k$.
 \end{enumerate}

 \end{lemma}
 
 \begin{proof}We first consider the asymptotics of 
 \begin{eqnarray}
 &&\lim_{[x_{a_1},\ldots,x_{a_{r+1}}\longrightarrow x_j]}\mathrm{Sym}_{a_1,\ldots,a_{r+1}}\label{sls}\\
&& \left.\left[\frac{c_0 x_{a_1}^r u_{\tilde{a}_1}^{k-s_0}(\partial_{\tilde{a}_1}^{s_1}[\log\tilde{T}_{N,\tilde{a}_1}])^{d_1}\ldots (\partial_{\tilde{a}_1}^{s_t}[\log\tilde{T}_{N,\tilde{a}_1}])^{d_t}}{(x_{a_1}-x_{a_2})\ldots(x_{a_1}-x_{a_{r+1}})}\right]\right|_{U_{N,\kappa}=\{1,\ldots,1\}}\notag
 \end{eqnarray}
 By Lemma \ref{l51}, the degree of $N$ in each factor $(\partial_{\tilde{a}_1}^{s_l}[\log\tilde{T}_{N,\tilde{a}_1}])^{d_1}$ is at most $d_1$. By Lemma \ref{l52} and identity (\ref{rsd}), the degree of $N$ in (\ref{sls}) is at most $k-r$. Summing over the permutations, we obtain that
 \begin{eqnarray*}
 &&\sum_{[\{a_1-N+\lfloor(1-\kappa)N \rfloor,\ldots,a_{r+1}-N+\lfloor(1-\kappa)N\rfloor\}\in \{1,2,\ldots,\lfloor(1-\kappa)N \rfloor\}\cap R(j)]}\\
 &&\lim_{[x_{a_1},\ldots,x_{a_{r+1}}\longrightarrow x_j]}\mathrm{Sym}_{a_1,\ldots,a_{r+1}}\frac{c_0 x_{a_1}^r u_{\tilde{a}_1}^{k-s_0}(\partial_{\tilde{a}_1}^{s_1}[\log\tilde{T}_{N,\tilde{a}_1}])^{d_1}\ldots (\partial_{\tilde{a}_1}^{s_t}[\log\tilde{T}_{N,\tilde{a}_1}])^{d_t}}{(x_{a_1}-x_{a_2})\ldots(x_{a_1}-x_{a_{r+1}})}
 \end{eqnarray*}
 is the sum of $O(N^{r+1})$ terms, the degree of $N$ in each of which is at most $k-r$. Therefore, the degree of $N$ in $E_{k,\kappa,N}^{(j)}$ is at most $k+1$, and we complete the proof of Part (1).
 
Now we prove Parts (2) and (3). 
 We consider the following two cases.
 \begin{itemize}
\item Consider
\begin{eqnarray*}
\mathcal{D}_i:&=&\left.\frac{\partial F_{k,\kappa,N}^{(j)}}{\partial u_i}\right|_{U_{N,\kappa}=(1,\ldots,1)}\\
&=&\frac{\partial}{\partial u_i}\left[\sum_{[r\geq 0,\{s_i\geq0\}_{i=0}^{t},\{d_i\geq 0\}_{i=1}^{t}:s_0+s_1d_1+\ldots s_tdt+r=k]}\right.\\&&\sum_{[\{a_1-N+\lfloor(1-\kappa)N \rfloor,\ldots,a_{r+1}-N+\lfloor(1-\kappa)N\rfloor\}\subset \{1,2,\ldots,\lfloor(1-\kappa)N \rfloor\}\cap R(j)]}(r+1)!\\
 &&\lim_{[x_{a_w}\longrightarrow x_{j}],\ 1\leq w\leq r+1}\mathrm{Sym}_{a_1,\ldots,a_{r+1}}\\
 &&\left.\left.\left(\frac{c_0 x_{a_1}^r u_{\tilde{a}_1}^{k-s_0}(\partial_{\tilde{a}_1}^{s_1}[\log T_{N,(\tilde{a}_1,i)}^{(j,r)}])^{d_1}\ldots (\partial_{\tilde{a}_1}^{s_t}[\log T_{N,(\tilde{a}_1,i)}^{(j,r)}])^{d_t}}{(x_{a_1}u_{\tilde{a}_1}-x_{a_2}u_{\tilde{a}_2})\ldots(x_{a_1}u_{\tilde{a}_1}-x_{a_{r+1}}u_{\tilde{a}_{r+1}})}\right)\right]\right|_{U_{N,\kappa}=(1,\ldots,1)}
\end{eqnarray*}
with 
\begin{eqnarray*}
i\notin \{a_1-N+\lfloor(1-\kappa)N \rfloor,\ldots,a_{r+1}-N+\lfloor(1-\kappa)N\rfloor\}\subset \{1,2,\ldots,\lfloor(1-\kappa)N \rfloor\}
\end{eqnarray*}

When $\kappa\in(0,1)$, there are $O(N^{r+1})$ such terms. Notice that if $i\in R(r)$
\begin{eqnarray*}
&&\left.\frac{\partial}{\partial u_i}\left[\frac{\partial^{s_q}(\log T_{N,(\tilde{a}_1,i)}^{(j,r)})}{\partial \tilde{a}_1^{s_q}}\right]^{d_q}\right|_{U_{N,\kappa}=(1,\ldots,1)}\\
&=&\left.d_q\left[\frac{\partial^{s_q}(\log T_{N,(\tilde{a}_1,i)}^{(j,r)}}{\partial \tilde{a}_1^{s_q}}\right]^{d_q-1}\frac{\partial}{\partial u_i}\left[\frac{\partial^{s_q}(\log T_{N,(\tilde{a}_1,i)}^{(j,r)})}{\partial \tilde{a}_1^{s_q}}\right]\right|_{U_{N,\kappa}=(1,\ldots,1)}\\
&=&\left.d_q\left[\frac{\partial^{s_q}\left[\log \tilde{T}_{N,\tilde{a}_1}\right]}{\partial \tilde{a}_1^{s_q}}\right]^{d_q-1}\frac{\partial}{\partial u_i}\left[\frac{\partial^{s_q}\left[\log T_{N,(\tilde{a}_1,i)}^{(j,r)}\right]}{\partial \tilde{a}_1^{s_q}}\right]\right|_{U_{N,\kappa}=(1,\ldots,1)}
\end{eqnarray*}
By Lemmas \ref{l51} and \ref{l52}, the degree of $N$ in the expressions above is at most $d_q-1$ if $r=j$, and is strictly less than $d_q-1$ when $r\neq j$. Note that $\mathcal{D}_i$ can be written as a sum of terms of the form
\begin{eqnarray*}
&&\lim_{[x_{a_1},\ldots,x_{a_{r+1}}\longrightarrow x_j]}\mathrm{Sym}_{a_1,\ldots,a_{r+1}}\\
&&\left.\left[\frac{c_0 x_{a_1}^r u_{\tilde{a}_1}^{k-s_0}\left(\partial_i\left[\left(\partial_{\tilde{a}_1}^{s_1}\left[\log T_{N,(\tilde{a}_1,i)}^{(j,r)}\right]\right)^{d_1}\right]\right)\ldots \left(\partial_{\tilde{a}_1}^{s_t}\left[\log \tilde{T}_{N,\tilde{a}_1}\right]\right)^{d_t}}{(x_{a_1}-x_{a_2})\ldots(x_{a_1}-x_{a_{r+1}})}\right]\right|_{U_{N,\kappa}=(1,\ldots,1)}
\end{eqnarray*}
By Lemmas \ref{l51} and \ref{l52}, the degree of $N$ in the expressions above is at most $d_1+\ldots+d_t-1$ when $r=j$; and is less than $d_1+\ldots+d_t-1$ when $r\neq j$. Since $r+s_0+s_1d_1+\ldots+s_td_t=k$, we have
\begin{eqnarray*}
d_1+\ldots+d_t-1\leq k-r-1;
\end{eqnarray*}
and in the sum over $\sum_{[\{a_1-N+\lfloor(1-\kappa)N \rfloor,\ldots,a_{r+1}-N+\lfloor(1-\kappa)N\rfloor\}\in \{1,2,\ldots,\lfloor(1-\kappa)N \rfloor\}\cap R(j)]}$, it is the sum of  $O(N^{r+1})$ of such terms, we obtain that the degree of $N$ in this sum is at most $k$ when $r=j$; and is less than $k$ when $r\neq j$. This completes the proof of Part (3).

\item Consider the case when
\begin{eqnarray*}
i\in \{a_1-N+\lfloor(1-\kappa)N \rfloor,\ldots,a_{r+1}-N+\lfloor(1-\kappa)N\rfloor\}\subset \{1,2,\ldots,\lfloor(1-\kappa)N \rfloor\}
\end{eqnarray*}
Note that there are $O(N^r)$ such terms in total, since $i$ is fixed. By Lemmas \ref{l51} and \ref{l52}, the degree of $N$ in
\begin{eqnarray*}
\left.\mathrm{Sym}_{a_1,\ldots,a_{r+1}}\left[\frac{c_0 x_{a_1}^r u_{\tilde{a}_1}^{k-s_0}\left(\partial_i\left[\left(\partial_{\tilde{a}_1}^{s_1}\left[\log T_{N,(\tilde{a}_1,i)}^{(j,r)}\right]\right)^{d_1}\right]\right)\ldots \left(\partial_{\tilde{a}_1}^{s_t}\left[\log \tilde{T}_{N,\tilde{a}_1}\right]\right)^{d_t}}{(x_{a_1}-x_{a_2})\ldots(x_{a_1}-x_{a_{r+1}})}\right]\right|_{U_{N,\kappa}=(1,\ldots,1)}
\end{eqnarray*}
is at most $l-r$.  This completes the proof of Part (2).
\end{itemize} 

Now we prove Part (4).
Note that
\begin{eqnarray*}
&&\left.\frac{\partial^2 F_{k,\kappa,N}^{(j)}}{\partial u_{i_1}\partial u_{i_2}}\right|_{U_{N,\kappa}=(1,\ldots,1)}\\
&=&\frac{\partial^2 }{\partial u_{i_1}\partial u_{i_2}}\left[\lim_{[x_k\rightarrow (x_{k\mod n}),\mathrm{for}\ 1\leq k\leq N]}
\frac{1}{\tilde{V}_{\lfloor (1-\kappa)N\rfloor}(U_{N,\kappa,X})\tilde{T}_N}\right.\\
&&\left.\left.\sum_{[i\in\{1,\ldots,\lfloor(1-\kappa)N\rfloor\}\cap R(j)]}\left(u_i\frac{\partial}{\partial u_i}\right)^k\tilde{V}_{\lfloor(1-\kappa)N \rfloor}(U_{N,\kappa,X})\tilde{T}_N\right]\right|_{U_{N,\kappa}=(1,\ldots,1)}\\
&=&\frac{\partial^2}{\partial u_{i_1}\partial u_{i_2}}\left[\sum_{[r\geq 0,\{s_i\geq0\}_{i=0}^{t},\{d_i\geq 0\}_{i=1}^{t}:s_0+s_1d_1+\ldots s_td_t+r=k]}\right.\\
&&\sum_{[\{a_1-N+\lfloor(1-\kappa)N \rfloor,\ldots,a_{r+1}-N+\lfloor(1-\kappa)N\rfloor\}\subset \{1,2,\ldots,\lfloor(1-\kappa)N \rfloor\}\cap R(j)]}(r+1)!\\
 &&\left.\left.\lim_{[x_{a_w}\longrightarrow x_{j}],\ 1\leq w\leq r+1}\mathrm{Sym}_{a_1,\ldots,a_{r+1}}\frac{c_0 x_{a_1}^r u_{\tilde{a}_1}^{k-s_0}(\partial_{\tilde{a}_1}^{s_1}[\log \tilde{T}_{N}])^{d_1}\ldots (\partial_{\tilde{a}_1}^{s_t}[\log \tilde{T}_{N}])^{d_t}}{(x_{a_1}u_{\tilde{a}_1}-x_{a_2}u_{\tilde{a}_2})\ldots(x_{a_1}u_{\tilde{a}_1}-x_{a_{r+1}}u_{\tilde{a}_{r+1}})}\right]\right|_{U_{N,\kappa}=(1,\ldots,1)}
\end{eqnarray*}
The following cases might occur
\begin{itemize}
\item $\{i_1,i_2\}\cap\{\tilde{a}_1,\ldots,\tilde{a}_{r+1}\}=\emptyset$, and both $\partial_{i_1}$ and $\partial_{i_2}$ are applied to the same $\log\tilde{T}_N$. We have
\begin{eqnarray*}
&&\frac{\partial^2}{\partial u_{i_1}\partial u_{i_2}}\left[\frac{\partial^{s_w}[\log \tilde{T}_N]}{[\partial u_{\tilde{a}_1}]^{s_w}}\right]^{d_w}\\
&=&d_w(d_w-1)\left[\frac{\partial^{s_w}[\log \tilde{T}_N]}{\partial u_{\tilde{a}_1}^{s_w}}\right]^{d_w-2}\frac{\partial^{s_w+1}[\log \tilde{T}_N]}{\partial u_{i_1}\partial u_{\tilde{a}_1}^{s_w}}\frac{\partial^{s_w+1}[\log \tilde{T}_N]}{\partial u_{i_2}\partial u_{\tilde{a}_1}^{s_w}}\\
&&+d_w \left[\frac{\partial^{s_w}[\log \tilde{T}_N]}{\partial u_{\tilde{a}_1}^{s_w}}\right]^{d_w-1}\frac{\partial^{s_w+2}[\log \tilde{T}_N]}{\partial u_{i_1}\partial u_{i_2}\partial u_{\tilde{a}_1}^{s_w}}
\end{eqnarray*}
 By Lemma \ref{le53}, the degree of $N$ in the expression above is less than $d_{w}-1$. Taking into account all the other factors, as well as the sum of $O(N^{r+1})$ terms, in this case the degree of $N$ in $\left.\frac{\partial^2 F_{k,\kappa,N}^{(j)}}{\partial u_{i_1}\partial u_{i_2}}\right|_{U_{N,\kappa}=(1,\ldots,1)}$ is less than
 \begin{eqnarray*}
 d_1+d_2+\ldots+d_t-1+r+1\leq k.
 \end{eqnarray*}
 \item $\{i_1,i_2\}\cap\{\tilde{a}_1,\ldots,\tilde{a}_{r+1}\}=\emptyset$, and $\partial_{i_1}$ and $\partial_{i_2}$ are applied to different $\log\tilde{T}_N$. In this case, the degree of $N$ is at most
  \begin{eqnarray*}
 d_1+d_2+\ldots+d_t-2+r+1\leq k-1.
 \end{eqnarray*}
 \item $i_1\in\{a_1,\ldots,a_{r+1}\}$ and $i_2\notin\{a_1,\ldots,a_{r+1}\}$. In this case, we take the sum over $O(N^r)$ terms, since one element in $\{a_1,\ldots,a_{r+1}\}$ is fixed to be $i_1$. Then the degree of $N$ in $\left.\frac{\partial^2 F_{k,\kappa,N}^{(j)}}{\partial u_{i_1}\partial u_{i_2}}\right|_{U_{N,\kappa}=(1,\ldots,1)}$ is at most 
 \begin{eqnarray*}
  d_1+d_2+\ldots+d_t-1+r\leq k-1.
 \end{eqnarray*}
 \item $\{i_1,i_2\}\subset \{a_1,\ldots,a_{r+1}\}$. In this case, we take the sum over $O(N^{r-1})$ terms, since two elements in $\{a_1,\ldots,a_{r+1}\}$ are fixed to be $i_1$ and $i_2$. Then the degree of $N$ in $\left.\frac{\partial^2 F_{k,\kappa,N}^{(j)}}{\partial u_{i_1}\partial u_{i_2}}\right|_{U_{N,\kappa}=(1,\ldots,1)}$ is at most 
 \begin{eqnarray*}
  d_1+d_2+\ldots+d_t+r-1\leq k-1.
 \end{eqnarray*}
\end{itemize}
\end{proof}

\subsection{Covariance} 

Let
\begin{eqnarray*}
\tilde{G}_{\kappa,N,(k,l)}^{(j,s)}&=&k\sum_{r=0}^{k-1}{{k-1}\choose{r}}\sum_{[\{a_1-N+\lfloor(1-\kappa)N \rfloor,\ldots,a_{r+1}-N+\lfloor(1-\kappa)N\rfloor\}\in \{1,2,\ldots,\lfloor(1-\kappa)N \rfloor\}\cap R(j)]}(r+1)!\\
&&\lim_{[x_{a_1},\ldots,x_{a_{r+1}}\longrightarrow x_j]}\mathrm{Sym}_{a_1,\ldots,a_{r+1}}\\
&&\left[\frac{ x_{a_1}^r u_{\tilde{a}_1}^{k-s_0} \partial_{\tilde{a}_1}[F_{l,\kappa,N}^{(s)}](\partial_{\tilde{a}_1}[\log \tilde{T}_{N}])^{k-1-r}}{(x_{a_1}u_{\tilde{a}_1}-x_{a_2}u_{\tilde{a}_2})\ldots(x_{a_1}u_{\tilde{a}_1}-x_{a_{r+1}}u_{\tilde{a}_{r+1}}))}\right]
\end{eqnarray*}

\begin{lemma}\label{l55}Let $l,k$ be arbitrary positive integers. Then
\begin{eqnarray*}
E_{k,l,\kappa,N}^{(j,s)}:=E_{k,\kappa,N}^{(j)}E_{l,\kappa,N}^{(s)}+\left.\tilde{G}_{\kappa,N,(l,k)}^{(j,s)}\right|_{U_{N,\kappa}=(1,\ldots,1)}+R
\end{eqnarray*}
\end{lemma}
where
\begin{itemize}
\item if $j=s$, the degree of $N$ in $\left.\tilde{G}_{\kappa,N,(l,k)}^{(j,s)}\right|_{U_{N,\kappa}=(1,\ldots,1)}$ is at most $l+k$;
 \item if $j\neq s$, the degree of $N$ in $\left.\tilde{G}_{\kappa,N,(l,k)}^{(j,s)}\right|_{U_{N,\kappa}=(1,\ldots,1)}$ is less than $l+k$;
\item the degree of $N$ in $R$ is less than $l+k$. 
\end{itemize}

\begin{proof}By (\ref{e2}) and (\ref{fkn}), we obtain
\begin{eqnarray*}
E_{k,l,\kappa,N}^{(j,s)}&=&\lim_{[x_k\rightarrow (x_{k\mod n}),\mathrm{for}\ 1\leq k\leq N]}\sum_{i\in\{1,\ldots,\lfloor(1-\kappa)N\rfloor\}\cap R(j)}\left(u_i\frac{\partial}{\partial u_i}\right)^k\\
&&\sum_{[r\in\{1,\ldots,\lfloor(1-\kappa)N\rfloor\}\cap R(s)]}\left(u_r\frac{\partial}{\partial u_r}\right)^l\left.\frac{V_{N,(i,r)}^{(j,s)}T_{N,(i,r)}^{(j,s)}}{W_{N,(i,r)}^{(j,s)}P_{N,(i,r)}^{(j,s)}}\right|_{U_{N,\kappa}=(1,\ldots,1)}\\
&=&\frac{1}{{V_{N,(i,r)}^{(j,s)}T_{N,(i,r)}^{(j,s)}}}\lim_{[x_k\rightarrow (x_{k\mod n}),\mathrm{for}\ 1\leq k\leq N]}\sum_{i\in\{1,\ldots,\lfloor(1-\kappa)N\rfloor\}\cap R(j)}\left(u_i\frac{\partial}{\partial u_i}\right)^k{V_{N,(i,r)}^{(j,s)}T_{N,(i,r)}^{(j,s)}}\\
&&\frac{1}{{V_{N,(i,r)}^{(j,s)}T_{N,(i,r)}^{(j,s)}}}\sum_{[r\in\{1,\ldots,\lfloor(1-\kappa)N\rfloor\}\cap R(s)]}\left(u_r\frac{\partial}{\partial u_r}\right)^l\left.{V_{N,(i,r)}^{(j,s)}T_{N,(i,r)}^{(j,s)}}\right|_{U_{N,\kappa}=(1,\ldots,1)}.
\end{eqnarray*}
Note that for any integer $w$, we have
\begin{eqnarray*}
&&\left.\frac{\partial^w}{\partial u_i^w}\left[\frac{1}{{V_{N,(i,r)}^{(j,s)}T_{N,(i,r)}^{(j,s)}}}\sum_{[r\in\{1,\ldots,\lfloor(1-\kappa)N\rfloor\}\cap R(s)]}\left(u_r\frac{\partial}{\partial u_r}\right)^l{V_{N,(i,r)}^{(j,s)}T_{N,(i,r)}^{(j,s)}}\right]\right|_{U_{N,\kappa}=(1,\ldots,1)}\\
&=&\left.\frac{\partial^w}{\partial u_i^w} F_{l,\kappa,N}^{(s)}\right|_{U_{N,\kappa}=(1,\ldots,1)}.\\
\end{eqnarray*}
Then $E_{k,l,\kappa,N}^{(j,s)}$ can be written as a sum of following terms
\begin{eqnarray*}
&&\left.\mathrm{Sym}_{a_1,\ldots,a_{r+1}}\left[\frac{c_0x_{a_1}^ru_{\tilde{a}_1}^{k-s_0}\partial_{\tilde{a}_1}^{s_1}[F_{l,\kappa,N}^{(s)}][\partial_{\tilde{a}_1}^{s_2}(\log T_{N,(\tilde{a}_1,r)}^{(j,s)})]^{d_2}\ldots [\partial_{\tilde{a}_1}^{s_t}(\log T_{N,(\tilde{a}_1,r)}^{(j,s)})]^{d_t}}{(x_{a_1}u_{\tilde{a}_1}-x_{a_2}u_{\tilde{a}_2})\ldots (x_{a_1}u_{\tilde{a}_1}-x_{a_{r+1}}u_{\tilde{a}_{r+1}})}\right]\right|_{U_{N,\kappa}=(1,\ldots,1)}\\
&=&\left.\mathrm{Sym}_{a_1,\ldots,a_{r+1}}\left[\frac{c_0x_{a_1}^ru_{\tilde{a}_1}^{k-s_0}\partial_{\tilde{a}_1}^{s_1}[F_{l,\kappa,N}^{(s)}][\partial_{\tilde{a}_1}^{s_2}(\log \tilde{T}_N]^{d_2}\ldots [\partial_{\tilde{a}_1}^{s_t}(\log \tilde{T}_N)^{(j,s)})]^{d_t}}{(x_{a_1}u_{\tilde{a}_1}-x_{a_2}u_{\tilde{a}_2})\ldots (x_{a_1}u_{\tilde{a}_1}-x_{a_{r+1}}u_{\tilde{a}_{r+1}})}\right]\right|_{U_{N,\kappa}=(1,\ldots,1)}.
\end{eqnarray*}
such that 
\begin{itemize}
\item  $r$, $s_0,s_1,\ldots,s_t$, $d_2,\ldots,d_t$ are nonnegative integers; and
\item  $s_2<s_3<\ldots<s_t$; and
\item 
\begin{eqnarray}
s_0+s_1+s_2d_2+\ldots+s_td_t+r=k; \label{sdk}
\end{eqnarray}
and
\item $\{a_1,\ldots,a_{r+1}\}\subset\{N-\lfloor(1-\kappa)N \rfloor+1,N-\lfloor(1-\kappa)N \rfloor+2,\ldots,N\}\cap R(j)$
\end{itemize}
When $s_1=0$, we obtain $\left.F_{l,\kappa,N}^{(s)}F_{k,\kappa,N}^{(j)}\right|_{U_{N,\kappa}=(1,\ldots,1)}$.

Now we consider the terms corresponding to $s_1\geq 1$. By Lemma \ref{l53}, the degree of $N$ in $\partial_{\tilde{a}_1}^{s_1}[F_{l,\kappa,N}^{(s)}]$ is at most $l$ when $j=s$; and is less than $l$ when $j\neq s$. Therefore, the total degree of $N$ in these terms is at most $l+d_2+\ldots+d_t+r+1$ when $j=s$; and is less than $l+d_2+\ldots+d_t+r+1$ when $j\neq s$. By (\ref{sdk}) and the assumption that $s_1\geq 1$, we have
\begin{eqnarray*}
l+d_2+\ldots+d_t+r+1\leq l+k；
\end{eqnarray*}
and the equality holds when $s_0=d_3=\ldots=d_t=0$, $s_1=s_2=1$; $d_2=k-1-r$; this corresponds to $G_{\kappa,N,(k,l)}^{(j,s)}$, in which the degree of $N$ is $l+k$ when $j=s$, and the degree of $N$ is less than $l+k$ when $j\neq s$. The degree of $N$ is less than $l+k$ in all the other terms. This completes the proof.
\end{proof}

\begin{lemma}\label{l56}
\begin{eqnarray*}
&&\left.\frac{1}{\tilde{V}_{\lfloor(1-\kappa)N\rfloor} \tilde{T}_N }k\sum_{i\in\{1,2,\ldots,\lfloor(1-\kappa)N \rfloor\cap R(j)}\left(u_i\frac{\partial}{\partial u_i}\left[F_{l,\kappa,N}^{(s)}\right]\right)\left(u_i\frac{\partial}{\partial u_i}\right)^{k-1}\left[\tilde{V}_{\lfloor(1-\kappa)N\rfloor} \tilde{T}_N \right]\right|_{U_{N,\kappa}=(1,\ldots,1)}\\
&=&\left.\tilde{G}_{\kappa,N,(k,l)}^{(j,s)}\right|_{U_{N,\kappa}=(1,\ldots,1)}+R
\end{eqnarray*}
where the degree of $N$ in $R$ is less than $l+k$.
\end{lemma}

\begin{proof}This follows from the proof of Lemma \ref{l55}.
\end{proof}

\begin{lemma}\label{l67}Let $i\in\{1,2,\ldots,\lfloor(1-\kappa)N\rfloor\}$. Then the degree of $N$ in 
\begin{eqnarray*}
\left.\frac{\partial}{\partial u_i}\tilde{G}_{\kappa,N,(k,l)}^{(j,s)}\right|_{U_{N,\kappa}=(1,\ldots,1)}
\end{eqnarray*}
 is less than $k+l$.
\end{lemma}

\begin{proof}Note that $\tilde{G}_{\kappa,N,(k,l)}^{(j,s)}$ is the sum of terms 
\begin{eqnarray*}
\mathrm{Sym}_{a_1,\ldots,a_{r+1}}
\left[\frac{ x_{a_1}^r u_{\tilde{a}_1}^{k-s_0} \partial_{\tilde{a}_1}[F_{l,\kappa,N}^{(s)}](\partial_{\tilde{a}_1}[\log \tilde{T}_{N}])^{k-1-r}}{(x_{a_1}u_{\tilde{a}_1}-x_{a_2}u_{\tilde{a}_2})\ldots(x_{a_1}u_{\tilde{a}_1}-x_{a_{r+1}}u_{\tilde{a}_{r+1}}))}\right]
\end{eqnarray*}
If we take derivatives $\frac{\partial}{\partial u_i}$, the following cases might occur
\begin{enumerate}
\item $i\in \{a_1,\ldots,a_{r+1}\}$. Since one element in $\{a_1,\ldots,a_{r+1}\}$ is fixed to be $i$, we take the sum over $O(N^r)$ terms. By Lemma \ref{le53}, and Lemma \ref{l53}, the degree of $N$ is at most
\begin{eqnarray*}
l+(k-1-r)+r=l+k-1
\end{eqnarray*}
\item $i\notin\{a_1,\ldots,a_{r+1}\}$. In this case, we take the sum over $O(N^r)$ terms. Again by Lemmas \ref{le53} and \ref{l53}, the degree of $N$ is less than
\begin{eqnarray*}
l+k-1-r+r+1=l+k.
\end{eqnarray*}
Then the lemma follows.
\end{enumerate}
\end{proof}

\subsection{Products of Moments}

Recall that $\mathcal{P}_{w_1,\ldots,w_p}^{s}$ is the set of all pairings of the set $\{1,2,\ldots,s\}\setminus \{w_1,\ldots,w_p\}$.
We have the following lemma concerning the products of moments.

\begin{lemma}\label{l58}Let $s,l_1,\ldots,l_s$ be positive integers, and let $j_1,\ldots,j_s\in\{1,2,\ldots,n\}$. Then
\begin{eqnarray*}
&&\frac{1}{\tilde{V}_{\lfloor (1-\kappa)N\rfloor}\tilde{T}_N}\sum_{[i_1\in\{1,2,\ldots,\lfloor(1-\kappa)N\rfloor\}\cap R(j_1)]}\left(u_{i_1}\frac{\partial}{\partial u_{i_1}}\right)^{l_1}\sum_{[i_2\in\{1,2,\ldots,\lfloor(1-\kappa)N\rfloor\}\cap R(j_2)]}\left(u_{i_2}\frac{\partial}{\partial u_{i_2}}\right)^{l_2}\cdots\\&&\left.\sum_{[i_s\in\{1,2,\ldots,\lfloor(1-\kappa)N\rfloor\}\cap R(j_s)]}\left(u_{i_s}\frac{\partial}{\partial u_{i_s}}\right)^{l_s}\left[\tilde{V}_{\lfloor (1-\kappa)N\rfloor}\tilde{T}_N\right]\right|_{U_{N,\kappa}=(1,\ldots,1)}\\
&=&\left.\sum_{p=0}^{s}\sum_{[w_1,\ldots,w_p\in\{1,2,\ldots,s\}]}F_{l_{w_1},\kappa,N}^{(j_{w_1})}F_{l_{w_2},\kappa,N}^{(j_{w_2})}\ldots F_{l_{w_s},\kappa,N}^{(j_{w_p})}\left(\sum_{P\in \mathcal{P}_{w_1,\ldots,w_p}^{s}}\prod_{(a,b)\in P}\tilde{G}_{\kappa,N,(l_a,l_b)}^{(j_a,j_b)}+R\right)\right|_{U_{N,\kappa}=(1,\ldots,1)},
\end{eqnarray*}
where the degree of $N$ in $R$ is less than $\sum_{i=1}^{s}l_i-\sum_{i=1}^{p}l_{w_i}$.
\end{lemma}

\begin{proof} The lemma can be proved by induction on $s$, similar to the proof of proposition 5.10 in \cite{bg16}. We shall now sketch the proof. When $s=1$ the lemma follows from the definition of $F_{l,\kappa,N}^{(j)}$. When $s=2$, the lemma follows from Lemma \ref{l55}. Assume that the lemma holds for $s=t-1$, where $t\geq 2$ is a positive integer. When $s=t$, by induction hypothesis, we have 
\begin{eqnarray*}
&&\frac{1}{\tilde{V}_{\lfloor (1-\kappa)N\rfloor}\tilde{T}_N}\sum_{[i_1\in\{1,2,\ldots,\lfloor(1-\kappa)N\rfloor\}\cap R(j_1)]}\left(u_{i_1}\frac{\partial}{\partial u_{i_1}}\right)^{l_1}\sum_{[i_2\in\{1,2,\ldots,\lfloor(1-\kappa)N\rfloor\}\cap R(j_2)]}\left(u_{i_2}\frac{\partial}{\partial u_{i_2}}\right)^{l_2}\cdots\\&&\left.\sum_{[i_t\in\{1,2,\ldots,\lfloor(1-\kappa)N\rfloor\}\cap R(j_t)]}\left(u_{i_t}\frac{\partial}{\partial u_{i_t}}\right)^{l_t}\left[\tilde{V}_{\lfloor (1-\kappa)N\rfloor}\tilde{T}_N\right]\right|_{U_{N,\kappa}=(1,\ldots,1)}\\
&=&\frac{1}{\tilde{V}_{\lfloor (1-\kappa)N\rfloor}\tilde{T}_N}\sum_{[i_1\in\{1,2,\ldots,\lfloor(1-\kappa)N\rfloor\}\cap R(j_1)]}\left(u_{i_1}\frac{\partial}{\partial u_{i_1}}\right)^{l_1}\left[\tilde{V}_{\lfloor (1-\kappa)N\rfloor}\tilde{T}_N\right]\\
&&\left(\sum_{p=0}^{t-1}\sum_{[w_1,\ldots,w_p\in\{2,\ldots,t\}]}F_{l_{w_1},\kappa,N}^{(j_{w_1})}F_{l_{w_2},\kappa,N}^{(j_{w_2})}\ldots F_{l_{w_p},\kappa,N}^{(j_{w_p})}\right.\\
&&\left.\left.\left(\sum_{P\in \mathcal{P}_{1,w_1,\ldots,w_p}^{t}}\prod_{(a,b)\in P}\tilde{G}_{\kappa,N,(l_a,l_b)}^{(j_a,j_b)}+R_{1,w_1,\ldots,w_p}\right)\right)\right|_{U_{N,\kappa}=(1,\ldots,1)}.\\
&=& S_1+S_2+S_3;
\end{eqnarray*}
where by induction hypothesis the degree of $N$ in $R_{1,w_1,\ldots,w_p}$ is less than $\sum_{i=2}^{s}l_i-\sum_{i=1}^{p}l_{w_i}$.
\begin{eqnarray*}
S_1&=&\left\{\frac{1}{\tilde{V}_{\lfloor (1-\kappa)N\rfloor}\tilde{T}_N}\sum_{[i_1\in\{1,2,\ldots,\lfloor(1-\kappa)N\rfloor\}\cap R(j_1)]}\left(u_{i_1}\frac{\partial}{\partial u_{i_1}}\right)^{l_1}\left[\tilde{V}_{\lfloor (1-\kappa)N\rfloor}\tilde{T}_N\right]\right\}\\
&&\times\left(\sum_{p=0}^{t-1}\sum_{[w_1,\ldots,w_p\in\{2,\ldots,t\}]}F_{l_{w_1},\kappa,N}^{(j_{w_1})}F_{l_{w_2},\kappa,N}^{(j_{w_2})}\ldots F_{l_{w_p},\kappa,N}^{(j_{w_p})}\right.\\
&&\left.\left.\left(\sum_{P\in \mathcal{P}_{1,w_1,\ldots,w_p}^{t}}\prod_{(a,b)\in P}\tilde{G}_{\kappa,N,(l_a,l_b)}^{(j_a,j_b)}+R_{1,w_1,\ldots,w_p}\right)\right)\right|_{U_{N,\kappa}=(1,\ldots,1)}.
\end{eqnarray*}
and
\begin{eqnarray*}
S_2&=&\frac{\ell_1}{\tilde{V}_{\lfloor (1-\kappa)N\rfloor}\tilde{T}_N}\sum_{[i_1\in\{1,2,\ldots,\lfloor(1-\kappa)N\rfloor\}\cap R(j_1)]}\left\{\left(u_{i_1}\frac{\partial}{\partial u_{i_1}}\right)^{l_1-1}\left[\tilde{V}_{\lfloor (1-\kappa)N\rfloor}\tilde{T}_N\right]\right\}\\
&&\times\left\{\left(u_{i_1}\frac{\partial}{\partial u_{i_1}}\right)\left(\sum_{p=0}^{t-1}\sum_{[w_1,\ldots,w_p\in\{2,\ldots,t\}]}F_{l_{w_1},\kappa,N}^{(j_{w_1})}F_{l_{w_2},\kappa,N}^{(j_{w_2})}\ldots F_{l_{w_p},\kappa,N}^{(j_{w_p})}\right.\right.\\
&&\left.\left.\left.\left(\sum_{P\in \mathcal{P}_{1,w_1,\ldots,w_p}^{t}}\prod_{(a,b)\in P}\tilde{G}_{\kappa,N,(l_a,l_b)}^{(j_a,j_b)}+R_{1,w_1,\ldots,w_p}\right)\right)\right\}\right|_{U_{N,\kappa}=(1,\ldots,1)}.
\end{eqnarray*}
Indeed, $S_1$ corresponds to the terms where all the differentiations $\frac{\partial}{\partial u_{i_1}}$ are applied to $\tilde{V}_{\lfloor (1-\kappa)N\rfloor}\tilde{T}_N$ or $u_{i_1}$; $S_2$ corresponds to the terms where all the differentiations $\frac{\partial}{\partial u_{i_1}}$ except one are applied to $\tilde{V}_{\lfloor (1-\kappa)N\rfloor}\tilde{T}_N$ or $u_{i_1}$, and $S_3$ are all the other terms.

By the definition of $F_{l,\kappa,N}^{(j)}$ we have
\begin{eqnarray*}
S_1&=&F_{l_1,\kappa,N}^{(j_1)}\left(\sum_{p=0}^{t-1}\sum_{[w_1,\ldots,w_p\in\{2,\ldots,t\}]}F_{l_{w_1},\kappa,N}^{(j_{w_1})}F_{l_{w_2},\kappa,N}^{(j_{w_2})}\ldots F_{l_{w_s},\kappa,N}^{(j_{w_s})}\right.\\
&&\left.\left.\left(\sum_{P\in \mathcal{P}_{1,w_1,\ldots,w_p}^{s}}\prod_{(a,b)\in P}\tilde{G}_{\kappa,N,(l_a,l_b)}^{(j_a,j_b)}+R_{1,w_1,\ldots,w_p}\right)\right)\right|_{U_{N,\kappa}=(1,\ldots,1)}.
\end{eqnarray*}

By Lemma \ref{l56}, we have
\begin{eqnarray*}
S_2&=&\left\{\left(\sum_{p=0}^{t-1}\sum_{[w_1,\ldots,w_p\in\{2,\ldots,t\}]}\sum_{x=1}^{p}F_{l_{w_1},\kappa,N}^{(j_{w_1})}\ldots F_{l_{w_{x-1}},\kappa,N}^{(j_{w_{x-1}})} F_{l_{w_{x+1}},\kappa,N}^{(j_{w_{x+1}})}\ldots F_{l_{w_s},\kappa,N}^{(j_{w_s})}\right.\right.\\
&&\left.\left.\left.\left[\left(\tilde{G}_{\kappa,N,(l_1,l_x)}^{(j_1,j_x)}+R_{1,x}\right)\left(\sum_{P\in \mathcal{P}_{1,w_1,\ldots,w_p}^{s}}\prod_{(a,b)\in P}\tilde{G}_{\kappa,N,(l_a,l_b)}^{(j_a,j_b)}+R_{1,w_1,\ldots,w_p}\right)\right)\right]\right\}\right|_{U_{N,\kappa}=(1,\ldots,1)},
\end{eqnarray*}
where the degree of $N$ in $R_{1,x}$ is less than $l_1+l_x$.

Hence we have
\begin{eqnarray*}
S_2&=&\left\{\left(\sum_{p=0}^{t-1}\sum_{[w_1,\ldots,w_p\in\{2,\ldots,t\}]}\sum_{x=1}^{p}F_{l_{w_1},\kappa,N}^{(j_{w_1})}\ldots F_{l_{w_{x-1}},\kappa,N}^{(j_{w_{x-1}})} F_{l_{w_{x+1}},\kappa,N}^{(j_{w_{x+1}})}\ldots F_{l_{w_s},\kappa,N}^{(j_{w_s})}\right.\right.\\
&&\left.\left.\left.\left[\tilde{G}_{\kappa,N,(l_1,l_x)}^{(j_1,j_x)}\sum_{P\in \mathcal{P}_{1,w_1,\ldots,w_p}^{t}}\prod_{(a,b)\in P}\tilde{G}_{\kappa,N,(l_a,l_b)}^{(j_a,j_b)}+R_{w_1,\ldots,\hat{w}_x,\ldots,w_p}\right)\right]\right\}\right|_{U_{N,\kappa}=(1,\ldots,1)}
\end{eqnarray*}
where the degree of $N$ in $R_{w_1,\ldots,\hat{w}_x,\ldots,w_p}$ is less than $\sum_{i=2}^{t}l_i-\sum_{i=1}^{p}l_{w_i}$.

Note that 
\begin{eqnarray*}
S_1+S_2&=&\sum_{p=0}^{s}\sum_{[w_1,\ldots,w_p\in\{1,2,\ldots,s\}]}F_{l_{w_1},\kappa,N}^{(j_{w_1})}F_{l_{w_2},\kappa,N}^{(j_{w_2})}\ldots F_{l_{w_t},\kappa,N}^{(j_{w_t})}\\
&&\left.\left(\sum_{P\in \mathcal{P}_{w_1,\ldots,w_p}^{t}}\prod_{(a,b)\in P}\tilde{G}_{\kappa,N,(l_a,l_b)}^{(j_a,j_b)}+R_{w_1,\ldots,w_p}\right)\right|_{U_{N,\kappa}=(1,\ldots,1)}
\end{eqnarray*}
where the degree of $N$ in $R_{w_1,\ldots,w_p}$ is less than $\sum_{i=1}^{t}l_i-\sum_{i=1}^{p}l_{w_i}$.

It remains to show that $S_3$ does not contribute to the leading terms. Define
\begin{eqnarray*}
\tilde{H}_{j_1,\ldots,j_p}=\left(\sum_{P\in \mathcal{P}_{1,w_1,\ldots,w_p}^{t}}\prod_{(a,b)\in P}\tilde{G}_{\kappa,N,(l_a,l_b)}^{(j_a,j_b)}+R_{1,w_1,\ldots,w_p}\right);
\end{eqnarray*}
By Lemma \ref{l55}, the degree of $N$ in $\left.\tilde{H}_{j_1,\ldots,j_p}\right|_{U_{N,\kappa}=(1,\ldots,1)}$ is at most $\sum_{i=2}^t l_i-\sum_{j=1}^p l_{w_j}$. Moreover, by Lemma \ref{l67}, for any index $i$, the degree of $N$ in $\left.\frac{\partial}{\partial u_i}\tilde{H}_{j_1,\ldots,j_p}\right|_{U_{N,\kappa}=(1,\ldots,1)}$ is less than $\sum_{i=2}^t l_i-\sum_{j=1}^p l_{w_j}$.

We write
\begin{eqnarray*}
&&\frac{1}{\tilde{V}_{\lfloor (1-\kappa)N\rfloor}\tilde{T}_N}\sum_{[i_1\in\{1,2,\ldots,\lfloor(1-\kappa)N\rfloor\}\cap R(j_1)]}\left(u_{i_1}\frac{\partial}{\partial u_{i_1}}\right)^{l_1}\left[\tilde{V}_{\lfloor (1-\kappa)N\rfloor}\tilde{T}_N\right]\\
&&\left.\left(\sum_{p=0}^{t-1}\sum_{[w_1,\ldots,w_p\in\{2,\ldots,t\}]}F_{l_{w_1},\kappa,N}^{(j_{w_1})}F_{l_{w_2},\kappa,N}^{(j_{w_2})}\ldots F_{l_{w_p},\kappa,N}^{(j_{w_p})}\tilde{H}_{j_1,\ldots,j_p}\right)\right|_{U_{N,\kappa}=(1,\ldots,1)}.
\end{eqnarray*}
as a sum of terms of the following form
\begin{eqnarray}
&&\lim_{x_{a_1},\ldots,x_{a_{r+1}}\rightarrow\ x_{j_1}}\mathrm{Sym}_{a_1,\ldots,a_{r+1}}\notag\\
&&\left[\frac{x_{a_1}^ru_{\tilde{a}_1}^{l_1-s_0}(\partial_{\tilde{a}_1}^{s_1}[\log\tilde{T}_N])^{d_1}\ldots (\partial_{\tilde{a}_t}^{s_t}[\log\tilde{T}_N])^{d_t}\partial_{\tilde{a}_1}^{f_1}F_{l_{w_1},\kappa,N}^{(j_{w_1})} \ldots \partial_{\tilde{a}_1}^{f_p}F_{l_{w_p},\kappa,N}^{(j_{w_p})}\partial_{\tilde{a}_1}^{h_0}\tilde{H}_{j_1,\ldots,j_p} }{(x_{a_1}u_{\tilde{a}_1}-x_{a_2}u_{\tilde{a}_2})\ldots(x_{a_1}u_{\tilde{a}_1}-x_{a_{r+1}}u_{\tilde{a}_{r+1}})}\right]\label{sa1}
\end{eqnarray}
where
\begin{itemize}
\item $\{\tilde{a}_1,\ldots,\tilde{a}_{r+1}\}\subset\{1,2,\ldots,\lfloor(1-\kappa)N\}\cap R(j_1)$;
\item $s_1<s_2<\ldots<s_t$ are positive integers;
\item $f_1,\ldots,f_p,h_0$ are nonnegative integers;
\item 
\begin{eqnarray}
r+s_0+s_1d_1+\ldots+s_td_t+f_1+\ldots+f_p+h_0=l_1\label{rsdl}
\end{eqnarray}
\end{itemize}
By Lemma \ref{l51}, the degree of $N$ in $(\partial_{\tilde{a}_1}^{s_1}[\log\tilde{T}_N])^{d_1}\ldots (\partial_{\tilde{a}_t}^{s_t}[\log\tilde{T}_N])^{d_t}$ is at most $d_1+\ldots+d_t$; therefore, the terms in (\ref{sa1}) with highest degree of $N$ has the form
\begin{eqnarray}
\mathrm{Sym}_{a_1,\ldots,a_{r+1}}\left[\frac{x_{a_1}^ru_{\tilde{a}_1}^{l_1}(\partial_{\tilde{a}_1}[\log\tilde{T}_N])^{d_1}\partial_{\tilde{a}_1}^{f_1}F_{l_{w_1},\kappa,N}^{(j_{w_1})} \ldots \partial_{\tilde{a}_1}^{f_p}F_{l_{w_p},\kappa,N}^{(j_{w_p})}\partial_{\tilde{a}_1}^{h_0}\tilde{H}_{j_1,\ldots,j_p} }{(x_{a_1}u_{\tilde{a}_1}-x_{a_2}u_{\tilde{a}_2})\ldots(x_{a_1}u_{\tilde{a}_1}-x_{a_{r+1}}u_{\tilde{a}_{r+1}})}\right]\label{sa2}
\end{eqnarray}
where
\begin{eqnarray}
s_0=d_2=\ldots=d_t=0;\ s_1=1.\label{sd0}
\end{eqnarray}
Let
\begin{eqnarray*}
B=\{i\in\{1,2,\ldots,p\}:f_i=0\}.
\end{eqnarray*}
Then 
\begin{eqnarray*}
(\ref{sa2})=\left[\prod_{i\in B}F_{l_{w_i},\kappa,N}^{(j_{w_i})}\right] S(u_1,\ldots,u_{\lfloor(1-\kappa) N\rfloor})
\end{eqnarray*}
where $S(u_1,\ldots,u_{\lfloor1-\kappa N\rfloor})$ is a symmetric function. It suffices to show that the degree of $N$ in $S$, except for $S_1$ and $S_2$, is less than $\sum_{i=1}^{s}l_i-\sum_{i\in B}l_i$. Note that the degree of $N$ in $\partial_{\tilde{a}_1}[\log\tilde{T}_N])^{d_1}$ is at most $d_1$ by Lemma \ref{l51}. The summation over $\{\tilde{a}_1,\ldots,\tilde{a}_{r+1}\}\subset\{1,2,\ldots,\lfloor(1-\kappa)N\}\cap R(j_1)$ gives $O(N^{r+1})$ terms.  By Lemma \ref{l53}, when $i\notin B$, the degree of $N$ in $\partial_{\tilde{a}_1}^{f_i}F_{l_{w_i},\kappa,N}^{(j_{w_i})} $ is at most $l_{w_i}$. Therefore the degree of $N$ in $S(u_1,\ldots,u_{\lfloor(1-\kappa) N\rfloor})$ is at most
\begin{eqnarray*}
\sum_{i=2}^{s}l_i-\sum_{i=1}^{p}l_{w_i}+d_1+\sum_{i\in \{1,2,\ldots,p\}\setminus B}l_{w_i}+r+1
\end{eqnarray*}

 By (\ref{rsd}) and (\ref{sd0}), if $|B|\leq p-2$, $r_1+d_1+1\leq l_1-1$, then the degree of $N$ in $S(u_1,\ldots,u_{\lfloor(1-\kappa) N\rfloor})$ is at most
 \begin{eqnarray*}
 \sum_{i=1}^{s}l_i-\sum_{i\in B}l_i-1
 \end{eqnarray*}
 Therefore only the terms where at most one $f_i$ is nonzero contribute to the leading order. In these terms if $h_0>0$,  then by Lemma \ref{l67}, the degree of $N$ is less than $\sum_{i=1}^{s}l_i-\sum_{i\in B}l_i$. So only the terms where $h_0=0$ and at most one $f_i$ is nonzero contribute to the leading order. These terms are in $S_1$ and $S_2$. Then the proof is complete.
\end{proof}

\begin{lemma}\label{le59}Let $s,l_1,\ldots,l_s$ be positive integers, and let $j_1,\ldots,j_s\in\{1,2,\ldots,n\}$. Then
\begin{eqnarray*}
&&\frac{1}{\tilde{V}_{\lfloor (1-\kappa)N\rfloor}\tilde{T}_N}\left[\sum_{[i_1\in\{1,2,\ldots,\lfloor(1-\kappa)N\rfloor\}\cap R(j_1)]}\left(u_{i_1}\frac{\partial}{\partial u_{i_1}}\right)^{l_1}-E_{l_1,\kappa,N}^{(j_1)}\right]\cdots\\&&\left.\left[\sum_{[i_s\in\{1,2,\ldots,\lfloor(1-\kappa)N\rfloor\}\cap R(j_s)]}\left(u_{i_s}\frac{\partial}{\partial u_{i_s}}\right)^{l_s}\left[\tilde{V}_{\lfloor (1-\kappa)N\rfloor}\tilde{T}_N\right]-E_{l_s,\kappa,N}^{(j_s)}\right]\right|_{U_{N,\kappa}=(1,\ldots,1)}\\
&=&\left.\sum_{P\in \mathcal{P}_{\emptyset}^{s}}\prod_{(a,b)\in P}\tilde{G}_{\kappa,N,(l_a,l_b)}^{(j_a,j_b)}+R\right|_{U_{N,\kappa}=(1,\ldots,1)},
\end{eqnarray*}
where the degree of $N$ in $R$ is less than $\sum_{i=1}^{s}l_i$.
\end{lemma}

\begin{proof}The lemma follows from Lemma \ref{l58} by explicit computations. See also the proof of Lemma 5.11 in \cite{bg16}.
\end{proof}

\begin{lemma}\label{le510}Let $s,l_1,\ldots,l_s$ be positive integers, and let $j_1,\ldots,j_s\in\{1,2,\ldots,n\}$. Let $\tilde{\mathcal{P}}_{\emptyset}^{s}\subset \mathcal{P}_{\emptyset}^{s}$ consisting of all the pairings of $\{1,2,\ldots,s\}$ such that in each pair $(a,b)$ in the pairing, $j_a=j_b$.
 Then
\begin{eqnarray*}
&&\lim_{N\rightarrow\infty}\frac{1}{N^{l_1+\ldots+l_s}}\frac{1}{\tilde{V}_{\lfloor (1-\kappa)N\rfloor}\tilde{T}_N}\left[\sum_{[i_1\in\{1,2,\ldots,\lfloor(1-\kappa)N\rfloor\}\cap R(j_1)]}\left(u_{i_1}\frac{\partial}{\partial u_{i_1}}\right)^{l_1}-E_{l_1,\kappa,N}^{(j_1)}\right]\cdots\\&&\left.\left[\sum_{[i_s\in\{1,2,\ldots,\lfloor(1-\kappa)N\rfloor\}\cap R(j_s)]}\left(u_{i_s}\frac{\partial}{\partial u_{i_s}}\right)^{l_s}\left[\tilde{V}_{\lfloor (1-\kappa)N\rfloor}\tilde{T}_N\right]-E_{l_s,\kappa,N}^{(j_s)}\right]\right|_{U_{N,\kappa}=(1,\ldots,1)}\\
&=&\lim_{N\rightarrow\infty}\frac{1}{N^{l_1+\ldots+l_s}}\left.\sum_{P\in \tilde{\mathcal{P}}_{\emptyset}^{s}}\prod_{(a,b)\in P}\tilde{G}_{\kappa,N,(l_a,l_b)}^{(j_a,j_a)}\right|_{U_{N,\kappa}=(1,\ldots,1)},
\end{eqnarray*}
where the degree of $N$ in $R$ is less than $\sum_{i=1}^{s}l_i$.
\end{lemma}

\begin{proof}The lemma follows from Lemma \ref{le59} and the fact that the degree of $N$ in $R$, in Lemma \ref{le59}, is less than $l_1+\ldots+l_s$, and that the degree of $N$ in $\tilde{G}_{\kappa,N,(l_a,l_b)}^{(j_a,j_b)}$ is less than $l_a+l_b$ if $j_a\neq j_b$.

\end{proof}

\subsection{Integral formula for covariance}

Assume that $\kappa\in (0,1)$ and $k$ is a positive integer. Let 
\begin{eqnarray*} p_k^{(\lfloor (1-\kappa)N \rfloor)}=\sum_{i=1}^{\lfloor (1-\kappa)N\rfloor}(\lambda_i+\lfloor (1-\kappa)N\rfloor-i)^k
\end{eqnarray*}
where $\lambda=(\lambda_1,\ldots,\lambda_{\lfloor (1-\kappa)N \rfloor})\in \GT_{\lfloor (1-\kappa)N \rfloor}$ has the distribution $\rho_{\lfloor (1-\kappa)N\rfloor }$ as defined in Lemma \ref{l33}. Explicit computations show that
\begin{eqnarray}
&&\mathbf{E}\left(p_{l_1}^{(\lfloor (1-\kappa)N\rfloor)}-\mathbf{E}p_{l_1}^{(\lfloor (1-\kappa)N}\right)\left(p_{l_2}^{(\lfloor (1-\kappa)N\rfloor)}-\mathbf{E}p_{l_2}^{(\lfloor (1-\kappa)N\rfloor)}\right)\label{mts}\\
&&\cdots \left(p_{l_s}^{(\lfloor (1-\kappa)N\rfloor)}-\mathbf{E}p_{l_s}^{(\lfloor (1-\kappa)N\rfloor)}\right)\notag\\
&=&\sum_{j_1,\ldots,j_s\in \{1,2,\ldots,n\}}\frac{1}{\tilde{V}_{\lfloor (1-\kappa)N\rfloor}\tilde{T}_N}\left[\sum_{[i_1\in\{1,2,\ldots,\lfloor(1-\kappa)N\rfloor\}\cap R(j_1)]}\left(u_{i_1}\frac{\partial}{\partial u_{i_1}}\right)^{l_1}-E_{l_1,\kappa,N}^{(j_1)}\right]\cdots
\notag
\\&&\left.\left[\sum_{[i_s\in\{1,2,\ldots,\lfloor(1-\kappa)N\rfloor\}\cap R(j_s)]}\left(u_{i_s}\frac{\partial}{\partial u_{i_s}}\right)^{l_s}\left[\tilde{V}_{\lfloor (1-\kappa)N\rfloor}\tilde{T}_N\right]-E_{l_s,\kappa,N}^{(j_s)}\right]\right|_{U_{N,\kappa}=(1,\ldots,1)}\notag
\end{eqnarray}

\begin{lemma}
\begin{eqnarray*}
&&\lim_{N\rightarrow\infty}\frac{1}{N^{l_1+l_2}}\mathbf{E}\left(p_{l_1}^{(\lfloor (1-\kappa)N\rfloor)}-\mathbf{E}p_{l_1}^{(\lfloor (1-\kappa)N\rfloor)}\right) \left(p_{l_2}^{(\lfloor (1-\kappa)N\rfloor)}-\mathbf{E}p_{l_2}^{(\lfloor (1-\kappa)N\rfloor)}\right)\\
&=&\left.\sum_{j=1}^{n}\lim_{N\rightarrow\infty}\frac{1}{N^{l_1+l_2}}\tilde{G}_{\kappa,N,(l_1,l_2)}^{(j,j)}\right|_{U_{N,\kappa}=(1,\ldots,1)}
\end{eqnarray*}
\end{lemma}

\begin{proof}The lemma follows from (\ref{mts}), Lemma \ref{le510}, Lemma \ref{l55}, and the fact that
\begin{eqnarray*}
\left.\tilde{G}_{\kappa,N,(l_1,l_2)}^{(j,j)}\right|_{U_{N,\kappa}=(1,\ldots,1)}=G_{\kappa,N,(l_1,l_2)}^{(j,j)}
\end{eqnarray*}
\end{proof}

Therefore, in order to obtain an explicit integral formula for the covariance 
\begin{eqnarray*}
\lim_{N\rightarrow\infty}\frac{1}{N^{l_1+l_2}}\mathbf{E}\left(p_{l_1}^{(\lfloor (1-\kappa)N\rfloor)}-\mathbf{E}p_{l_1}^{(\lfloor (1-\kappa)N\rfloor)}\right) \left(p_{l_2}^{(\lfloor (1-\kappa)N\rfloor)}-\mathbf{E}p_{l_2}^{(\lfloor (1-\kappa)N\rfloor)}\right);
\end{eqnarray*}
It suffices to obtain an explicit integral formula for 
\begin{eqnarray*}
\left.\lim_{N\rightarrow\infty}\frac{1}{N^{l_1+l_2}}\tilde{G}_{\kappa,N,(l_1,l_2)}^{(j,j)}\right|_{U_{N,\kappa}=(1,\ldots,1)},
\end{eqnarray*}
where $1\leq j\leq N$.

We have
\begin{eqnarray*}
&&\left.\lim_{N\rightarrow\infty}\frac{1}{N^{l_1+l_2}}\tilde{G}_{\kappa,N,(l_1,l_2)}^{(j,j)}\right|_{U_{N,\kappa}=(1,\ldots,1)}\\
&=&\lim_{N\rightarrow\infty}\frac{1}{N^{l_1+l_2}}l_1\sum_{r=0}^{l_1-1}{{l_1-1}\choose{r}}\sum_{[\{a_1-N+\lfloor(1-\kappa)N \rfloor,\ldots,a_{r+1}-N+\lfloor(1-\kappa)N\rfloor\}\in \{1,2,\ldots,\lfloor(1-\kappa)N \rfloor\}\cap R(j)]}(r+1)!\\
&&\lim_{[x_{a_1},\ldots,x_{a_{r+1}}\longrightarrow x_j]}\mathrm{Sym}_{a_1,\ldots,a_{r+1}}\left[\frac{ x_{a_1}^r u_{\tilde{a}_1}^{l_1} (\partial_{\tilde{a}_1}[\log \tilde{T}_{N}])^{l_1-1-r}}{(x_{a_1}u_{\tilde{a}_1}-x_{a_2}u_{\tilde{a}_2})\ldots(x_{a_1}u_{\tilde{a}_1}-x_{a_{r+1}}u_{\tilde{a}_{r+1}}))}\right]\\
&&\partial_{\tilde{a}_1}\left[\sum_{s=0}^{l_2}{{l_2}\choose{s}}\sum_{[\{b_1-N+\lfloor(1-\kappa)N \rfloor,\ldots,b_{s+1}-N+\lfloor(1-\kappa)N\rfloor\}\subset \{1,2,\ldots,\lfloor(1-\kappa)N \rfloor\}\cap R(j)]}\right.(s+1)!\\
 &&\left.\left.\lim_{[x_{b_w}\longrightarrow x_{j}],\ 1\leq w\leq s+1}\mathrm{Sym}_{b_1,\ldots,b_{s+1}}\frac{c_0 x_{b_1}^s u_{\tilde{b}_1}^{l_2}(\partial_{\tilde{b}_1}[\log \tilde{T}_N])^{l_2-s}}{(x_{b_1}u_{\tilde{b}_1}-x_{b_2}u_{\tilde{b}_2})\ldots(x_{b_1}u_{\tilde{b}_1}-x_{b_{s+1}}u_{\tilde{b}_{s+1}})}\right]\right|_{U_{N,\kappa}=(1,\ldots,1)}
\end{eqnarray*}

We consider the following cases
\begin{itemize}
\item If $\{a_1,\ldots,a_{r+1}\}\cap \{b_1,\ldots,b_{s+1}\}=\emptyset$, we have
\begin{eqnarray*}
\partial_{\tilde{a}_1}(\partial_{\tilde{b}_1}[\log \tilde{T}_N])^{l_2-s}=(l_2-s)(\partial_{\tilde{b}_1}[\log \tilde{T}_N])^{l_2-s-1}\partial_{\tilde{a}_1}(\partial_{\tilde{b}_1}[\log \tilde{T}_N])
\end{eqnarray*}
where the degree of $N$, when $U_{N,\kappa}=(1,\ldots,1)$, is at most $l_2-s-1$. By Lemma \ref{l51} and Lemma \ref{l52}, and note that
\begin{eqnarray*}
\end{eqnarray*}
we have
\begin{eqnarray*}
I_1:&=&\lim_{N\rightarrow\infty}\frac{1}{N^{l_1+l_2}}l_1\sum_{r=0}^{l_1-1}{{l_1-1}\choose{r}}\sum_{[\{a_1-N+\lfloor(1-\kappa)N \rfloor,\ldots,a_{r+1}-N+\lfloor(1-\kappa)N\rfloor\}\in \{1,2,\ldots,\lfloor(1-\kappa)N \rfloor\}\cap R(j)]}(r+1)!\\
&&\lim_{[x_{a_1},\ldots,x_{a_{r+1}}\longrightarrow x_j]}\mathrm{Sym}_{a_1,\ldots,a_{r+1}}\left[\frac{ x_{a_1}^r u_{\tilde{a}_1}^{l_1} (\partial_{\tilde{a}_1}[\log \tilde{T}_{N}])^{l_1-1-r}}{(x_{a_1}u_{\tilde{a}_1}-x_{a_2}u_{\tilde{a}_2})\ldots(x_{a_1}u_{\tilde{a}_1}-x_{a_{r+1}}u_{\tilde{a}_{r+1}}))}\right]\\
&&\partial_{\tilde{a}_1}\left[\sum_{s=0}^{l_2}{{l_2}\choose{s}}\sum_{[\{b_1-N+\lfloor(1-\kappa)N \rfloor,\ldots,b_{s+1}-N+\lfloor(1-\kappa)N\rfloor\}\subset \{1,2,\ldots,\lfloor(1-\kappa)N \rfloor\}\cap R(j),\{a_1,\ldots,a_{r+1}\}\cap\{b_1,\ldots,b_{s+1}\}=\emptyset]}\right.\\
 &&(s+1)!\left.\left.\lim_{[x_{b_w}\longrightarrow x_{j}],\ 1\leq w\leq s+1}\mathrm{Sym}_{b_1,\ldots,b_{s+1}}\frac{c_0 x_{b_1}^s u_{\tilde{b}_1}^{l_2}(\partial_{\tilde{b}_1}[\log \tilde{T}_N])^{l_2-s}}{(x_{b_1}u_{\tilde{b}_1}-x_{b_2}u_{\tilde{b}_2})\ldots(x_{b_1}u_{\tilde{b}_1}-x_{b_{s+1}}u_{\tilde{b}_{s+1}})}\right]\right|_{U_{N,\kappa}=(1,\ldots,1)}\\
 &=&\lim_{N\rightarrow\infty}\frac{1}{N^{l_1+l_2}}l_1\sum_{r=0}^{l_1-1}{{l_1-1}\choose{r}}\sum_{[\{a_1-N+\lfloor(1-\kappa)N \rfloor,\ldots,a_{r+1}-N+\lfloor(1-\kappa)N\rfloor\}\in \{1,2,\ldots,\lfloor(1-\kappa)N \rfloor\}\cap R(j)]}(r+1)!\\
&&\lim_{[x_{a_1},\ldots,x_{a_{r+1}}\longrightarrow x_j]}\mathrm{Sym}_{a_1,\ldots,a_{r+1}}\left[\frac{ x_{a_1}^r u_{\tilde{a}_1}^{l_1} [A_j(u_{\tilde{a}_1})N]^{l_1-1-r}}{(x_{a_1}u_{\tilde{a}_1}-x_{a_2}u_{\tilde{a}_2})\ldots(x_{a_1}u_{\tilde{a}_1}-x_{a_{r+1}}u_{\tilde{a}_{r+1}}))}\right]\\
&&\left[\sum_{s=0}^{l_2}{{l_2}\choose{s}}(l_2-s)\sum_{[\{b_1-N+\lfloor(1-\kappa)N \rfloor,\ldots,b_{s+1}-N+\lfloor(1-\kappa)N\rfloor\}\subset \{1,2,\ldots,\lfloor(1-\kappa)N \rfloor\}\cap R(j),\{a_1,\ldots,a_{r+1}\}\cap\{b_1,\ldots,b_{s+1}\}=\emptyset]}\right.\\
 &&(s+1)!\left.\left.\lim_{[x_{b_w}\longrightarrow x_{j}],\ 1\leq w\leq s+1}\mathrm{Sym}_{b_1,\ldots,b_{s+1}}\frac{ x_{b_1}^s u_{\tilde{b}_1}^{l_2}([A_j(u_{\tilde{b}_1})N])^{l_2-s-1}B_j(u_{\tilde{a}_1},u_{\tilde{b}_1})}{(x_{b_1}u_{\tilde{b}_1}-x_{b_2}u_{\tilde{b}_2})\ldots(x_{b_1}u_{\tilde{b}_1}-x_{b_{s+1}}u_{\tilde{b}_{s+1}})}\right]\right|_{U_{N,\kappa}=(1,\ldots,1)}
\end{eqnarray*}
where
\begin{eqnarray}
A_j(z)&=&\begin{cases}\frac{\kappa}{n}\sum_{l\in I_2\cap\{1,2,\ldots,n\}}\frac{y_l x_1}{1+y_lx_1 z}-\frac{\kappa}{n}\frac{n-1}{z}+\frac{1}{n}H'_{\bm_i}(z),\ \mathrm{if}\ j=1;\\-\frac{\kappa}{n}\frac{n-j}{z}+\frac{1}{n}H'_{\bm_j}(z),\ \mathrm{if}\ 2\leq j\leq n \end{cases}\label{aj}
\end{eqnarray}
and
\begin{eqnarray}
B_j(z,w)&=&\frac{\partial^2}{\partial z\partial w}\left[\log\left(1-(z-1)(w-1)\frac{z H_{\bm_j}'(z)-wH'_{\bm_j}(w)}{z-w}\right)\right]\label{bj}
\end{eqnarray}

By Lemma \ref{l552}, we obtain
\begin{eqnarray*}
I_1&\approx&\lim_{N\rightarrow\infty}\frac{1}{N^{l_1+l_2}}\sum_{r=0}^{l_1-1}\frac{l_1!}{(l_1-1-r)! r!}\sum_{[\{a_1-N+\lfloor(1-\kappa)N \rfloor,\ldots,a_{r+1}-N+\lfloor(1-\kappa)N\rfloor\}\in \{1,2,\ldots,\lfloor(1-\kappa)N \rfloor\}\cap R(j)]}\\
&&\frac{\partial^r}{\partial z^r}\left[z^{l_1} [NA_j(z)]^{l_1-1-r}\right]\left[\sum_{s=0}^{l_2}\frac{l_2!}{(l_2-s-1)! s!}\right.\\
&&\left.\sum_{[\{b_1-N+\lfloor(1-\kappa)N \rfloor,\ldots,b_{s+1}-N+\lfloor(1-\kappa)N\rfloor\}\subset \{1,2,\ldots,\lfloor(1-\kappa)N \rfloor\}\cap R(j),\{a_1,\ldots,a_{r+1}\}\cap\{b_1,\ldots,b_{s+1}\}=\emptyset]}\right.\\
 &&\frac{\partial^s}{\partial w^s}\left.\left[w^{l_2}[NA_j(w)]^{l_2-s-1}B_j(z,w)\right]\right|_{(z,w)=(1,1)}
\end{eqnarray*}
By the residue theorem, we deduce that
\begin{eqnarray*}
I_1&\approx&\lim_{N\rightarrow\infty}\frac{1}{N^{l_1+l_2}}\sum_{r=0}^{l_1-1}\frac{l_1!}{(l_1-1-r)!}\sum_{[\{a_1-N+\lfloor(1-\kappa)N \rfloor,\ldots,a_{r+1}-N+\lfloor(1-\kappa)N\rfloor\}\in \{1,2,\ldots,\lfloor(1-\kappa)N \rfloor\}\cap R(j)]}\\
&&\mathrm{Res}_{z=1}\left(\frac{z^{l_1} [NA_j(z)]^{l_1-1-r}}{(z-1)^{r+1}}\right.\\
&&\left[\sum_{s=0}^{l_2-1}\frac{l_2!}{(l_2-s-1)!}\sum_{[\{b_1-N+\lfloor(1-\kappa)N \rfloor,\ldots,b_{s+1}-N+\lfloor(1-\kappa)N\rfloor\}\subset \{1,2,\ldots,\lfloor(1-\kappa)N \rfloor\}\cap R(j),\{a_1,\ldots,a_{r+1}\}\cap\{b_1,\ldots,b_{s+1}\}=\emptyset]}\right.\\
 &&\mathrm{Res}_{w=1}\left.\left[\frac{w^{l_2}[NA_j(w)]^{l_2-s-1}B_j(z,w)}{(w-1)^{s+1}}\right]\right)\\
 &\approx&\frac{1}{(2\pi\mathbf{i})^2}\lim_{N\rightarrow\infty}\frac{1}{N^{l_1+l_2}}\oint_{|z-1|=\epsilon}
\left(\frac{\lfloor (1-\kappa)N\rfloor}{n}\frac{z}{z-1}+zNA_j(z)\right)^{l_1}\\
 &&\oint_{|w-1|=\epsilon}\left(\frac{\lfloor (1-\kappa)N\rfloor}{n}\frac{w}{w-1}+wNA_j(w)\right)B_j(z,w)dwdz\\
 &=&\frac{1}{(2\pi\mathbf{i})^2}\oint_{|z-1|=\epsilon}
\left(\frac{1-\kappa}{n}\frac{z}{z-1}+zA_j(z)\right)^{l_1}\oint_{|w-1|=\epsilon}\left(\frac{1-\kappa}{n}\frac{w}{w-1}+wA_j(w)\right)^{l_2}B_j(z,w)dwdz
\end{eqnarray*}

\item If $|\{a_1,\ldots,a_{r+1}\}\cap \{b_1,\ldots,b_{s+1}\}|\geq 2$, then the degree of $N$ in these terms is at most
\begin{eqnarray*}
l_2-s+l_1-1-r+r+1+s+1-2=l_1+l_2-1<l_1+l_2, 
\end{eqnarray*}
therefore the contribution of these terms to the limit is 0.

\item If $|\{a_1,\ldots,a_{r+1}\}\cap \{b_1,\ldots,b_{s+1}\}|=1$.  Then
\begin{eqnarray*}
I_2:&=&\lim_{N\rightarrow\infty}\frac{1}{N^{l_1+l_2}}l_1\sum_{r=0}^{l_1-1}{{l_1-1}\choose{r}}\sum_{[\{a_1-N+\lfloor(1-\kappa)N \rfloor,\ldots,a_{r+1}-N+\lfloor(1-\kappa)N\rfloor\}\in \{1,2,\ldots,\lfloor(1-\kappa)N \rfloor\}\cap R(j)]}(r+1)!\\
&&\lim_{[x_{a_1},\ldots,x_{a_{r+1}}\longrightarrow x_j]}\mathrm{Sym}_{a_1,\ldots,a_{r+1}}\left[\frac{ x_{a_1}^r u_{\tilde{a}_1}^{l_1} (\partial_{\tilde{a}_1}[\log \tilde{T}_{N}])^{l_1-1-r}}{(x_{a_1}u_{\tilde{a}_1}-x_{a_2}u_{\tilde{a}_2})\ldots(x_{a_1}u_{\tilde{a}_1}-x_{a_{r+1}}u_{\tilde{a}_{r+1}}))}\right]\\
&&\partial_{\tilde{a}_1}\left[\sum_{s=0}^{l_2}{{l_2}\choose{s}}\sum_{[\{b_1-N+\lfloor(1-\kappa)N \rfloor,\ldots,b_{s+1}-N+\lfloor(1-\kappa)N\rfloor\}\subset \{1,2,\ldots,\lfloor(1-\kappa)N \rfloor\}\cap R(j),|\{a_1,\ldots,a_{r+1}\}\cap\{b_1,\ldots,b_{s+1}\}|=1]}\right.\\
 &&(s+1)!\left.\left.\lim_{[x_{b_w}\longrightarrow x_{j}],\ 1\leq w\leq s+1}\mathrm{Sym}_{b_1,\ldots,b_{s+1}}\frac{c_0 x_{b_1}^s u_{\tilde{b}_1}^{l_2}(\partial_{\tilde{b}_1}[\log \tilde{T}_N])^{l_2-s}}{(x_{b_1}u_{\tilde{b}_1}-x_{b_2}u_{\tilde{b}_2})\ldots(x_{b_1}u_{\tilde{b}_1}-x_{b_{s+1}}u_{\tilde{b}_{s+1}})}\right]\right|_{U_{N,\kappa}=(1,\ldots,1)}\\
  &=&\lim_{N\rightarrow\infty}\frac{1}{N^{l_1+l_2}}l_1\sum_{r=0}^{l_1-1}{{l_1-1}\choose{r}}\sum_{[\{a_1-N+\lfloor(1-\kappa)N \rfloor,\ldots,a_{r+1}-N+\lfloor(1-\kappa)N\rfloor\}\in \{1,2,\ldots,\lfloor(1-\kappa)N \rfloor\}\cap R(j)]}(r+1)!\\
&&\lim_{[x_{a_1},\ldots,x_{a_{r+1}}\longrightarrow x_j]}\mathrm{Sym}_{a_1,\ldots,a_{r+1}}\left[\frac{ x_{a_1}^r u_{\tilde{a}_1}^{l_1} [A_j(u_{\tilde{a}_1})N]^{l_1-1-r}}{(x_{a_1}u_{\tilde{a}_1}-x_{a_2}u_{\tilde{a}_2})\ldots(x_{a_1}u_{\tilde{a}_1}-x_{a_{r+1}}u_{\tilde{a}_{r+1}}))}\right]\\
&&\partial_{\tilde{a}_1}\left[\sum_{s=0}^{l_2}{{l_2}\choose{s}}\sum_{[\{b_1-N+\lfloor(1-\kappa)N \rfloor,\ldots,b_{s+1}-N+\lfloor(1-\kappa)N\rfloor\}\subset \{1,2,\ldots,\lfloor(1-\kappa)N \rfloor\}\cap R(j),|\{a_1,\ldots,a_{r+1}\}\cap\{b_1,\ldots,b_{s+1}\}|=1]}\right.\\
 &&(s+1)!\left.\left.\lim_{[x_{b_w}\longrightarrow x_{j}],\ 1\leq w\leq s+1}\mathrm{Sym}_{b_1,\ldots,b_{s+1}}\frac{ x_{b_1}^s u_{\tilde{b}_1}^{l_2}([A_j(u_{\tilde{b}_1})N])^{l_2-s}}{(x_{b_1}u_{\tilde{b}_1}-x_{b_2}u_{\tilde{b}_2})\ldots(x_{b_1}u_{\tilde{b}_1}-x_{b_{s+1}}u_{\tilde{b}_{s+1}})}\right]\right|_{U_{N,\kappa}=(1,\ldots,1)}
\end{eqnarray*}
By Lemma \ref{l552}, we deduce
\begin{eqnarray*}
I_2:&=&I_3+I_4
\end{eqnarray*}
where
\begin{eqnarray*}
I_3:&=&
\lim_{N\rightarrow\infty}\frac{1}{N^{l_1+l_2}}l_1\sum_{r=0}^{l_1-1}{{l_1-1}\choose{r}}\sum_{[\{a_1-N+\lfloor(1-\kappa)N \rfloor,\ldots,a_{r+1}-N+\lfloor(1-\kappa)N\rfloor\}\in \{1,2,\ldots,\lfloor(1-\kappa)N \rfloor\}\cap R(j)]}\\
&&\frac{\partial^r}{\partial u_{\tilde{a}_1}^r}\left[ u_{\tilde{a}_1}^{l_1} [A_j(u_{\tilde{a}_1})N]^{l_1-1-r}\right]\\
&&\partial_{\tilde{a}_1}\left[\sum_{s=0}^{l_2}{{l_2}\choose{s}}\sum_{[\{b_1-N+\lfloor(1-\kappa)N \rfloor,\ldots,b_{s+1}-N+\lfloor(1-\kappa)N\rfloor\}\subset \{1,2,\ldots,\lfloor(1-\kappa)N \rfloor\}\cap R(j),\{a_1,\ldots,a_{r+1}\}\cap\{b_1,\ldots,b_{s+1}\}=\{b_1\}]}\right.\\
 &&\left.\left.\frac{\partial^s}{\partial u_{\tilde{b}_1}^s}\left( u_{\tilde{b}_1}^{l_2}[A_j(u_{\tilde{b}_1})N]^{l_2-s}\right)\right]\right|_{U_{N,\kappa}=(1,\ldots,1)};\\
 &=&
\lim_{N\rightarrow\infty}\frac{1}{N^{l_1+l_2}}l_1\sum_{r=0}^{l_1-1}{{l_1-1}\choose{r}}\sum_{[\{a_1-N+\lfloor(1-\kappa)N \rfloor,\ldots,a_{r+1}-N+\lfloor(1-\kappa)N\rfloor\}\in \{1,2,\ldots,\lfloor(1-\kappa)N \rfloor\}\cap R(j)]}\\
&&\frac{\partial^r}{\partial u_{\tilde{a}_1}^r}\left[u_{\tilde{a}_1}^{l_1} [A_j(u_{\tilde{a}_1})N]^{l_1-1-r}\right]\\
&&\left[\sum_{s=0}^{l_2}{{l_2}\choose{s}}\sum_{[\{b_1-N+\lfloor(1-\kappa)N \rfloor,\ldots,b_{s+1}-N+\lfloor(1-\kappa)N\rfloor\}\subset \{1,2,\ldots,\lfloor(1-\kappa)N \rfloor\}\cap R(j),\{a_1,\ldots,a_{r+1}\}\cap\{b_1,\ldots,b_{s+1}\}=\{b_1\}]}\right.\\
 &&\left.\left.\frac{\partial^{s+1}}{\partial u_{\tilde{b}_1}^{s+1}}\left( u_{\tilde{b}_1}^{l_2}[A_j(u_{\tilde{b}_1})N]^{l_2-s}\right)\right]\right|_{U_{N,\kappa}=(1,\ldots,1)};
\end{eqnarray*}
and 
\begin{eqnarray*}
I_4&=&
\lim_{N\rightarrow\infty}\frac{1}{N^{l_1+l_2}}l_1\sum_{r=0}^{l_1-1}{{l_1-1}\choose{r}}\sum_{[\{a_1-N+\lfloor(1-\kappa)N \rfloor,\ldots,a_{r+1}-N+\lfloor(1-\kappa)N\rfloor\}\in \{1,2,\ldots,\lfloor(1-\kappa)N \rfloor\}\cap R(j)]}\\
&&\frac{\partial^r}{\partial u_{\tilde{a}_1}^r}\left[u_{\tilde{a}_1}^{l_1}[A_j(u_{\tilde{a}_1})N]^{l_1-1-r}\right]\\
&&\partial_{\tilde{a}_1}\left[\sum_{s=0}^{l_2}{{l_2}\choose{s}}\sum_{[\{b_1-N+\lfloor(1-\kappa)N \rfloor,\ldots,b_{s+1}-N+\lfloor(1-\kappa)N\rfloor\}\subset \{1,2,\ldots,\lfloor(1-\kappa)N \rfloor\}\cap R(j),\{a_1,\ldots,a_{r+1}\}\cap\{b_1,\ldots,b_{s+1}\}=\{b_j\},j\neq 1]}\right.\\
 &&\left.\left.\frac{\partial^s}{\partial u_{\tilde{b}_1}^s}\left( u_{\tilde{b}_1}^{l_2}[A_j(u_{\tilde{b}_1})N]^{l_2-s}\right)\right]\right|_{U_{N,\kappa}=(1,\ldots,1)};\\
 &=&0.
\end{eqnarray*}
By the residue theorem, we infer
\begin{eqnarray*}
I_2&=&\lim_{N\rightarrow\infty}\frac{1}{N^{l_1+l_2}}\sum_{r=0}^{l_1-1}\frac{l_1!}{(l_1-1-r)!}\sum_{[\{a_1-N+\lfloor(1-\kappa)N \rfloor,\ldots,a_{r+1}-N+\lfloor(1-\kappa)N\rfloor\}\in \{1,2,\ldots,\lfloor(1-\kappa)N \rfloor\}\cap R(j)]}\\
&&\mathrm{Res}_{z=u_{\tilde{a}_1}}\left[ \frac{z^{l_1} [A_j(z)N]^{l_1-1-r}}{(z-1)^{r+1}}\right]\\
&&\left[\sum_{s=0}^{l_2}\frac{l_2!}{(l_2-s)!}\sum_{[\{b_1-N+\lfloor(1-\kappa)N \rfloor,\ldots,b_{s+1}-N+\lfloor(1-\kappa)N\rfloor\}\subset \{1,2,\ldots,\lfloor(1-\kappa)N \rfloor\}\cap R(j),|\{a_1,\ldots,a_{r+1}\}\cap\{b_1,\ldots,b_{s+1}\}|=1]}\right.\\
 &&\left.\left.(s+1)\mathrm{Res}_{w=u_{\tilde{b}_1}}\left( \frac{w^{l_2}[A_j(w)N]^{l_2-s}}{(w-1)^{s+2}}\right)\right]\right|_{U_{N,\kappa}=(1,\ldots,1)}\\
 &=&\lim_{N\rightarrow\infty}\frac{1}{N^{l_1+l_2}}\frac{1}{(2\pi\mathbf{i})^2}\oint_{|z-1|=\epsilon}\left(\frac{\lfloor (1-\kappa)N\rfloor}{n}\frac{z}{z-1}+NzA_j(z)\right)^{l_1}\\
&&\oint_{|w-1|=\epsilon}\left(\frac{\lfloor (1-\kappa)N\rfloor}{n}\frac{w}{w-1}+NzA_j(w)\right)^{l_2}\frac{1}{(z-w)^2}dwdz\\
 &=&\frac{1}{(2\pi\mathbf{i})^2}\oint_{|z-1|=\epsilon}\left(\frac{ (1-\kappa)}{n}\frac{z}{z-1}+zA_j(z)\right)^{l_1}\\
&&\oint_{|w-1|=\epsilon}\left(\frac{ (1-\kappa)}{n}\frac{w}{w-1}+wA_j(w)\right)^{l_2}\frac{1}{(z-w)^2}dwdz
\end{eqnarray*}
\end{itemize}

Then we have the following proposition
\begin{proposition}\label{p612}
\begin{eqnarray*}
&&\lim_{N\rightarrow\infty}\frac{1}{N^{l_1+l_2}}\mathbf{E}\left(p_{l_1}^{(\lfloor (1-\kappa)N\rfloor)}-\mathbf{E}p_{l_1}^{(\lfloor (1-\kappa)N\rfloor)}\right) \left(p_{l_2}^{(\lfloor (1-\kappa)N\rfloor)}-\mathbf{E}p_{l_2}^{(\lfloor (1-\kappa)N\rfloor)}\right)\\
&=&\sum_{j=1}^n\frac{1}{(2\pi\mathbf{i})^2}\oint_{|z-1|=\epsilon}\left(\frac{ (1-\kappa)}{n}\frac{z}{z-1}+zA_j(z)\right)^{l_1}\\
&&\oint_{|w-1|=\epsilon}\left(\frac{ (1-\kappa)}{n}\frac{w}{w-1}+wA_j(w)\right)^{l_2}\left[B_j(z,w)+\frac{1}{(z-w)^2}\right]dwdz
\end{eqnarray*}
where for $1\leq j\leq n$, $A_j(z)$ and $B_j(z,w)$ are given by (\ref{aj}), (\ref{bj}).
\end{proposition}

\subsection{Central limit theorem in multiple levels}

Let 
\begin{eqnarray*}
&&1\geq\kappa_1\geq \kappa_2\geq\ldots\geq\kappa_k>0\\
&&1\leq n_1\leq n_2\ldots\leq n_k\leq 2N+1;
\end{eqnarray*}
such that for $1\leq i\leq k$,
\begin{eqnarray*}
\lfloor\frac{n_i}{2}\rfloor =\lfloor (1-\kappa_i)N\rfloor.
\end{eqnarray*}
and
\begin{eqnarray*}
U_{N,\kappa_1,\ldots,\kappa_k,X}=&&\left(u_{1,1}x_{1+N-\lfloor(1-\kappa_1) N\rfloor,N},u_{2,1}x_{2+N-\lfloor(1-\kappa_1) N\rfloor,N},\ldots,u_{\lfloor (1-\kappa_1)N\rfloor,1}x_{N,N}\right.;\\
&&u_{1,2}x_{1+N-\lfloor(1-\kappa_2) N\rfloor,N},u_{2,2}x_{2+N-\lfloor(1-\kappa_2) N\rfloor,N},\ldots,u_{\lfloor (1-\kappa_2)N\rfloor,2}x_{N,N}\\
&&\ldots\\
&&\left.u_{1,k}x_{1+N-\lfloor(1-\kappa_k) N\rfloor,N},u_{2,k}x_{2+N-\lfloor(1-\kappa_k) N\rfloor,N},\ldots,u_{\lfloor (1-\kappa_k)N\rfloor,k}x_{N,N}\right)
\end{eqnarray*}
Let $\mathcal{S}_{\rho,X}(U_{N,\kappa_1,\ldots,\kappa_s,X})$ be the multi-dimensional Schur generating function as defined in \ref{df59}, where $\rho$ is the joint distribution of partitions on the $n_1$th, $n_2$th, \ldots, $n_k$th row of the square-hexagon lattice, counting from the top. Then explicit computations show that
\begin{eqnarray*}
&&\mathbf{E}p_{l_1}^{(\lfloor (1-\kappa_1)N\rfloor)} p_{l_2}^{(\lfloor (1-\kappa_2)N\rfloor)}\cdots p_{l_k}^{(\lfloor (1-\kappa_k)N\rfloor)}\\
&=&\left.\mathcal{D}_{l_1}^{(n_1)}\mathcal{D}_{l_2}^{(n_2)}\ldots\mathcal{D}_{l_k}^{(n_k)}\mathcal{S}_{\rho,X}(U_{N,\kappa_1,\ldots,\kappa_k,X})\right|_{(u_{1,s},\ldots,u_{\lfloor \frac{n_s}{2}\rfloor,s})=(1,\ldots,1),\ \forall 1\leq s\leq k}
\end{eqnarray*}
where $\mathcal{D}_{l_i}^{(n_i)}$ is defined in (\ref{dd}).

\begin{lemma}\label{ll613}Suppose the assumptions in Lemma \ref{l59} hold. For $1\leq s\leq k$, let
\begin{eqnarray*}
t_s=N-\lfloor\frac{n_s}{2} \rfloor.
\end{eqnarray*}
Let
\begin{eqnarray*}
\mathbf{D}_{l_s}^{(n_s)}:=\frac{1}{\hat{V}_{\lfloor \frac{n_s}{2}\rfloor}}\left(\sum_{i=t_s+1}^{N}\left(u_{i}\frac{\partial}{\partial u_{i}}\right)^{l_s}\right)\hat{V}_{\lfloor \frac{n_s}{2}\rfloor},
\end{eqnarray*}
where $\hat{V}_{\lfloor \frac{n_s}{2}\rfloor}$ is the Vandermonde determinant on $\lfloor \frac{n_s}{2}\rfloor$ variables $u_1x_{t_s+1}, u_2x_{t_s+2},\ldots,u_{\lfloor\frac{n_s}{2} \rfloor}x_{N}$. Then
\begin{eqnarray*}
&&\left.\mathcal{D}_{l_1}^{(n_1)}\mathcal{D}_{l_2}^{(n_2)}\ldots\mathcal{D}_{l_k}^{(n_k)}\mathcal{S}_{\rho,X}(U_{N,\kappa_1,\ldots,\kappa_k,X})\right|_{(u_{1,s},\ldots,u_{\lfloor \frac{n_s}{2}\rfloor,s})=(1,\ldots,1),\ \forall 1\leq s\leq k}\\
&=&\frac{1}{\mathcal{S}{\rho_{\lfloor\frac{n_1}{2} \rfloor},X}(U_{N,\kappa_1,X})}\mathbf{D}_{l_1}^{(n_1)}\frac{\mathcal{S}_{\rho_{\lfloor\frac{n_1}{2} \rfloor},X}(U_{N,\kappa_1,X})}{\mathcal{S}_{\rho_{\lfloor\frac{n_2}{2} \rfloor},X}(U_{N,\kappa_2,X})}\mathbf{D}_{l_2}^{(n_2)}\ldots\\
&&\frac{\mathcal{S}_{\rho_{\lfloor\frac{n_{k-1}}{2} \rfloor},X}(U_{N,\kappa_{k-1},X})}{\mathcal{S}_{\rho_{\lfloor\frac{n_k}{2} \rfloor},X}(U_{N,\kappa_k,X})} \mathbf{D}_{l_k}^{(n_k)}\left\{
\left.\mathcal{S}_{\rho_{\lfloor\frac{n_k}{2} \rfloor},X}(U_{N,\kappa_k,X})\right\}\right|_{(u_{1},\ldots,u_N)=(1,\ldots,1)}
\end{eqnarray*}
where $\mathcal{S}_{\rho_{\lfloor\frac{n_k}{2} \rfloor},X}(U_{N,\kappa_k,X})$ is the one-dimensional Schur generating function defined as in Definition \ref{df33}, and $\rho_{\frac{n_k}{2}}$ is a probability measure on $\GT_{\lfloor \frac{n_k}{2}\rfloor}^+$ defined as in Lemma \ref{l33}.
\end{lemma}

\begin{proof}The lemma follows from same arguments as the proof of Lemma \ref{ll511}.
\end{proof}

For $1\leq s\leq N$, let $F_{k,\kappa,N}^{(s)}$ be defined as in (\ref{fkn}). Let
\begin{eqnarray*}
\tilde{G}_{\kappa_1,\kappa_2,N,(k,l)}^{(j,s)}&=&k\sum_{r=0}^{k-1}{{k-1}\choose{r}}\sum_{[\{a_1-N+\lfloor(1-\kappa_1)N \rfloor,\ldots,a_{r+1}-N+\lfloor(1-\kappa_1)N\rfloor\}\in \{1,2,\ldots,\lfloor(1-\kappa_1)N \rfloor\}\cap R(j)]}(r+1)!\\
&&\lim_{[x_{a_1},\ldots,x_{a_{r+1}}\longrightarrow x_j]}\mathrm{Sym}_{a_1,\ldots,a_{r+1}}\\
&&\left[\frac{ x_{a_1}^r u_{\tilde{a}_1}^{k-s_0} \partial_{\tilde{a}_1}[F_{l,\kappa_2,N}^{(s)}](\partial_{\tilde{a}_1}[\log \tilde{T}_{N,\kappa_1}])^{k-1-r}}{(x_{a_1}u_{\tilde{a}_1}-x_{a_2}u_{\tilde{a}_2})\ldots(x_{a_1}u_{\tilde{a}_1}-x_{a_{r+1}}u_{\tilde{a}_{r+1}}))}\right]
\end{eqnarray*}

where 
\begin{eqnarray*}
&&\tilde{T}_N=\prod_{l\in\{1,\ldots,N-\lfloor(1-\kappa_1)N \rfloor\}\cap I_2}\prod_{j=1}^{\lfloor(1-\kappa_1)N \rfloor}\left(\frac{1+y_{\ol{l}}x_{\ol{N-\lfloor (1-\kappa)N\rfloor+j}}u_j}{1+y_{\ol{l}}x_{\ol{N-\lfloor (1-\kappa)N\rfloor+j}}}\right)\\
&\times&\left(\prod_{i=1}^{r}s_{\phi^{(j(i),\sigma_0)}(N)}\left(u_{i},u_{n+i}\ldots,u_{q_{N,\kappa}n+i},1,\ldots,1\right)\right)\\
&&\times\left(\prod_{i=r+1}^{n}s_{\phi^{(j(i),\sigma_0)}(N)}\left(u_{i},u_{n+i}\ldots,u_{(q_{N,\kappa}-1)n+i},1,\ldots,1\right)\right)\\
&&\times\left(\prod_{N-\lfloor(1-\kappa)N\rfloor+1\leq i\leq N,\ 1\leq j\leq N-\lfloor(1-\kappa)N\rfloor,\ \tilde{i}\in R(p),\tilde{j}\in R(q),p<q}\frac{1}{x_iu_{\tilde{i}}-x_j}\right)\left(1+o(1)\right)
\end{eqnarray*}

\begin{lemma}
\begin{eqnarray*}
&&\lim_{N\rightarrow\infty}\frac{1}{N^{l_1+\ldots+l_s}}\mathbf{E}\left(p_{l_1}^{(\lfloor (1-\kappa_1)N\rfloor)}-\mathbf{E}p_{l_1}^{(\lfloor (1-\kappa_1)N}\right)\left(p_{l_2}^{(\lfloor (1-\kappa_2)N\rfloor)}-\mathbf{E}p_{l_2}^{(\lfloor (1-\kappa_2)N\rfloor)}\right)\\
&&\cdots \left(p_{l_s}^{(\lfloor (1-\kappa_k)N\rfloor)}-\mathbf{E}p_{l_s}^{(\lfloor (1-\kappa_k)N\rfloor)}\right)\notag\\
&=&\lim_{N\rightarrow\infty}\frac{1}{N^{l_1+\ldots+l_s}}\left.\sum_{P\in \tilde{\mathcal{P}}_{\emptyset}^{s}}\prod_{(a,b)\in P}\tilde{G}_{\kappa_a,\kappa_b,N,(l_a,l_b)}^{(j_a,j_a)}\right|_{U_{N,\kappa}=(1,\ldots,1)}
\end{eqnarray*}
\end{lemma}

\begin{proof}The lemma follows from similar arguments as the proof of Lemma \ref{le510}.
\end{proof}

\begin{proposition}\label{p615} Assume $\kappa_1,\kappa_2\in (0,1)$
\begin{eqnarray*}
&&\lim_{N\rightarrow\infty}\frac{1}{N^{l_1+l_2}}\mathbf{E}\left(p_{l_1}^{(\lfloor (1-\kappa_1)N\rfloor)}-\mathbf{E}p_{l_1}^{(\lfloor (1-\kappa_1)N\rfloor)}\right) \left(p_{l_2}^{(\lfloor (1-\kappa_2)N\rfloor)}-\mathbf{E}p_{l_2}^{(\lfloor (1-\kappa_2)N\rfloor)}\right)\\
&=&\sum_{j=1}^n\frac{1}{(2\pi\mathbf{i})^2}\oint_{|z-1|=\epsilon}\left(\frac{ (1-\kappa_1)}{n}\frac{z}{z-1}+zA_j(z)\right)^{l_1}\\
&&\oint_{|w-1|=\epsilon}\left(\frac{ (1-\kappa_2)}{n}\frac{w}{w-1}+zA_j(w)\right)^{l_2}\left[B_j(z,w)+\frac{1}{(z-w)^2}\right]dwdz
\end{eqnarray*}
where for $1\leq j\leq n$, $A_j(z)$ and $B_j(z,w)$ are given by (\ref{aj}), (\ref{bj}).
\end{proposition}

\begin{proof}The proposition follows from similar arguments as the proof of Proposition \ref{p612}.
\end{proof}

Then we have the following theorem:

\begin{theorem}\label{clt2}Assume $\kappa_1,\kappa_2\in (0,1)$. Then the random variables $\left\{\frac{1}{N^{l}}[p_{l}^{\lfloor(1-\kappa)N \rfloor}-\mathbf{E}p_{l}^{\lfloor(1-\kappa)N \rfloor}]\right\}_{l,\kappa}$  converge in distribution to a mean 0 Gaussian vector with covariance given by Proposition \ref{p615}. Moreover, each $\frac{1}{N^{l}}[p_{l}^{\lfloor(1-\kappa)N\rfloor}-\mathbf{E}p_{l}^{\lfloor(1-\kappa)N \rfloor}]$ converge in distribution to the sum of $n$ independent mean 0 Gaussian random variables.

\end{theorem}

\section{Gaussian free field in uniform boundary condition}\label{gffu}

\subsection{Height function}

Let $\mathcal{R}(\Omega(N),\check{c})$ be a contracting square-hexagon lattice. Assume that the edge weights of $\SH(\check{c})$ satisfy Assumption \ref{apew}.

The planar dual graph $\mathrm{SH}^*(\check{c})$ of $\mathrm{SH}(\check{c})$ is obtained by placing a vertex of $\mathrm{SH}^*(\check{c})$ inside each face of $\mathrm{SH}(\check{c})$; two vertices of $\SH^*(\check{c})$ are adjacent, or joined by an edge in $\SH^*(\check{c})$, if and only if the two corresponding faces of $\SH(\check{c})$ share an edge of $\SH(\check{c})$.

We place a vertex of $\SH^*(\check{c})$ at the center of each face of $\SH(\check{c})$, and obtain an embedding of $\SH^*(\check{c})$ into the plane. Each face of $\SH^*(\check{c})$ is either a triangle or a square, depending on whether the corresponding vertex of $\SH(\check{c})$ inside the dual face in $\SH^*(\check{c})$ is degree-3 or degree-4. 

For a contracting square-hexagon lattice $\mathcal{R}(\Omega,\check{c})$, let $\mathcal{R}^*(\Omega,\check{c})$ be a finite triangle-square lattice such that
\begin{itemize}
\item $\mathcal{R}^*(\Omega,\check{c})$ is a finite subgraph of $\SH^*(\check{c})$ as constructed above; and
\item $\mathcal{R}(\Omega,\check{c})$ is the interior dual graph of $\mathcal{R}^*(\Omega,\check{c})$. 
\end{itemize}
In other words, $\mathcal{R}^*(\Omega,\check{c})$ is the subgraph of $\SH^*(\check{c})$ consisting of all the faces of $\SH^*(\check{c})$ corresponding to vertices of $\mathcal{R}(\Omega,\check{c})$; see Figure \ref{fig:SHht}.

\begin{definition}\label{dfn35}Let $M\in\mathcal{M}(\Omega,\check{c})$ be a perfect matching of a contracting square-hexagon lattice $\mathcal{R}(\Omega,\check{c})$. We color the vertices of $\mathcal{R}(\Omega,\check{c})$ by black and white such that vertices of the same color cannot be adjacent and the boundary row of $\mathcal{R}(\Omega,\check{c})$ on the bottom consists of white vertices. A height function $h_M$ is an integer-valued function on vertices of $\mathcal{R}^*(\Omega,\check{c})$ that satisfies the following property. 

Let $f_1,f_2$ be a pair of adjacent vertices of $\mathcal{R}^*(\Omega,\check{c})$. Let $(f_1,f_2)$ denote the non-oriented edge of $\mathcal{R}^*(\Omega,\check{c})$ with endpoints $f_1$ and $f_2$; and let $[f_1,f_2\rangle$ (resp. $[f_2,f_1\rangle$) denote the oriented edge starting from $f_1$ (resp.\ $f_2$) and ending in $f_2$ (resp.\ $f_1$).
\begin{itemize}
\item if $(f_1,f_2)$ is a dual edge crossing a NW-SE edge or a NE-SW edge of $\SH(\check{c})$, 
\begin{itemize}
\item if the oriented dual edge $[f_1,f_2\rangle$ crosses an absent edge $e$ of $\SH(\check{c})$ in $M$ then $h_M(f_2)=h_M(f_1)+1$ if $[f_1,f_2\rangle$ has the white vertex or $e$ on the left, and $h_M(f_2)=h_M(f_1)-1$ otherwise.
\item if an oriented dual edge $[f_1,f_2\rangle$ crosses a present edge $e$ of $\SH(\check{c})$ in $M$ then $h_M(f_2)=h_M(f_1)-3$ if $[f_1,f_2\rangle$ has the white vertex of $e$ on the left, and $h_M(f_2)=h_M(f_1)+3$ otherwise.
\end{itemize}
\item if $(f_1,f_2)$ is a dual edge crossing a vertical edge of $\SH(\check{a})$.
\begin{itemize}
\item If an oriented dual edge $[f_1,f_2\rangle$ crosses an absent edge $e$ of $\SH(\check{c})$ in $M$, then $h_M(f_2)=h_M(f_1)+2$ if $[f_1,f_2\rangle$ has the white vertex of $e$ on the left, and $h_M(f_2)=h_M(f_1)-2$ otherwise.
\item If an oriented dual edge $[f_1,f_2\rangle$ crosses a present edge $e$ of $\SH(\check{c})$ in $M$ then $h_M(f_2)=h_M(f_1)-2$ if $[f_1,f_2\rangle$ has the white vertex of $e$ on the left, and $h_M(f_2)=h_M(f_1)+2$ otherwise.
\end{itemize}
\item $h_{M}(f_0)=0$, where $f_0$ is the lexicographic smallest vertex of $\mathcal{R}^*(\Omega,\check{c})$.
\end{itemize}
\end{definition}

It is straightforward to verify that the height function above is well-defined, by checking that around either a degree-3 vertex or a degree-4 vertex, the total height change is 0.
Moreover, since none of the boundary edges of $\mathcal{R}(\Omega,\check{c})$ (by boundary edges we mean edges of $\SH(\check{c})$ joining exactly one vertex of $\mathcal{R}(\Omega,\check{c})$ and one vertex outside $\mathcal{R}(\Omega,\check{c})$) are present in any perfect matching of $\mathcal{R}(\Omega,\check{c})$, the height function restricted on the boundary vertices of $\mathcal{R}^*(\Omega,\check{c})$ is fixed and independent of the random perfect matching; see Figure \ref{fig:SHht}.

\begin{figure}
\includegraphics{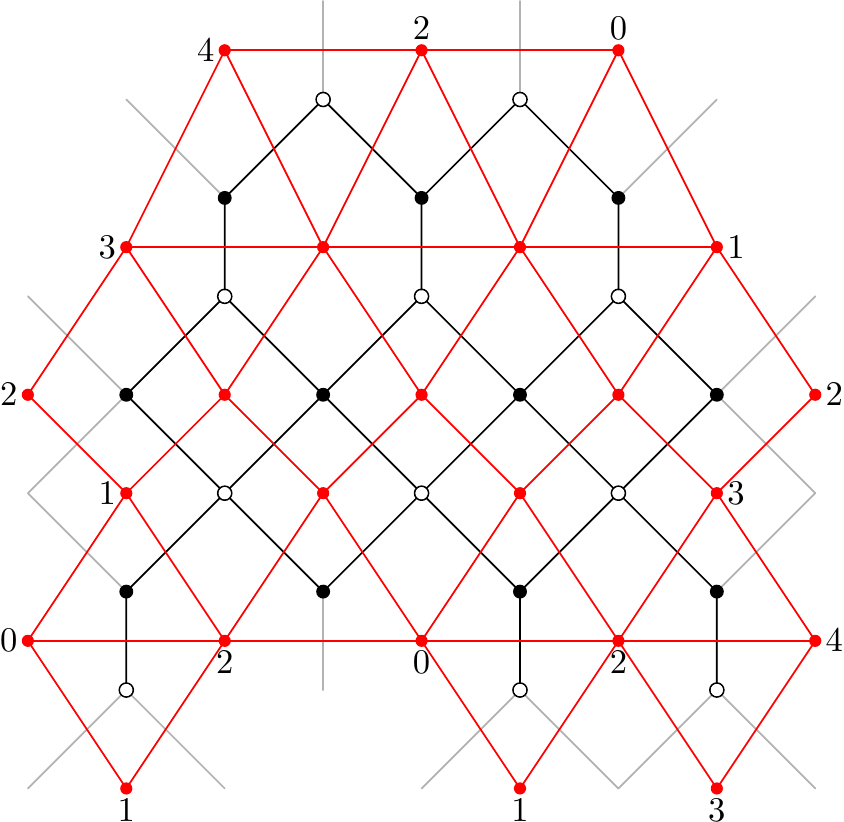}
\caption{Contracting square-hexagon lattice $\mathcal{R}(\Omega,\check{c})$, dual graph $\mathcal{R}^*(\Omega,\check{c})$ and height function on the boundary. The black lines represent the graph $\mathcal{R}(\Omega,\check{c})$, the gray lines represent boundary edges of $\mathcal{R}(\Omega,\check{c})$, the red lines represent the dual graph $\mathcal{R}^*(\Omega,\check{c})$, and the height function is defined on vertices of the dual graph. The values of the height function on the boundary vertices of $\mathcal{R}^*(\Omega,\check{c})$ are also shown in the figure.}
\label{fig:SHht}
\end{figure}

\begin{theorem}(Law of large numbers for the height function.) Assume that the assumptions of Proposition \ref{plm} hold.

Let $\rho_N^k$ be the measure on the configurations of the $k$th row, and let $\kappa\in (0,1)$, such that $k=[2\kappa N]$. Let  $\mathbf{m}^{\kappa}$ be the limit of the counting measures $m(\rho_N^k)$ in probability as $N\rightarrow\infty$ with moments given by (\ref{mtl})

Define
\begin{eqnarray}
\mathbf{h}(\chi,\kappa):=2\left(2(1-\kappa)\int_0^{\frac{\chi-\frac{\kappa r}{2n}}{1-\kappa}}d\mathbf{m}^{\kappa}-2\chi+2\kappa\right)\label{lh}
\end{eqnarray}
Then the random height function $h_M$ associate to a random perfect matching $M$, as defined by Definition \ref{dfn35}, has the following law of large numbers
\begin{eqnarray*}
\frac{h_{M}([\chi N],[\kappa N])}{N}\rightarrow\mathbf{h}(\chi,\kappa), \mathrm{when}\ N\rightarrow\infty
\end{eqnarray*}
where $\chi, \kappa$ are new continuous parameters of the domain.
\end{theorem}

\begin{proof}Same arguments as in the proof of Theorem 3.24 in \cite{BL17}.
\end{proof}

\subsection{Variational principle and complex Burger's equation}

Let $\mathrm{SH}(\check{c})$ be the whole plane square hexagon lattice with edge weights assigned as in Assumption \ref{apew} and periodically with period $n$ such that (\ref{px}) and (\ref{py}) hold. Then $\mathbb{Z}^2$ acts on $\mathrm{SH}(\check{c})$ by vertex-color-preserving and edge-weight-preserving isomorphisms of $\mathrm{SH}(\check{c})$. Let $\mathrm{SH}_1(\check{c})$ be the quotient graph of $\mathrm{SH}(\check{c})$ under the action of $\ZZ^2$. The graph $\mathrm{SH}_1(\check{c})$ is called a \textbf{fundamental domain} of $\mathrm{SH}(\check{c})$, which a finite graph that can be embedded into a torus.

Let $\gamma_x$ and $\gamma_y$ be two directed simple cycles winding once around the two homology generators of the torus where $\mathrm{SH}_1(\check{c})$ is embedded. Assume the edge weights of the square-hexagon lattice satisfy Assumption \ref{apew}. We shall modify the edge weights of the graph and construct a modified weighted adjacency matrix (Kasteleyn matrix) for $\mathrm{SH}_1(\check{c})$, which plays an essential role in the analysis of periodic dimer models, see \cite{Ka61,TF61,RK01}. 
\begin{itemize}
\item Multiply all the edge weights $x_i$ by $-1$. This way around each face of degree 4, there are an odd number of ``$-$'' signs multiplied by edge weights; while around each face of degree 6, there are an even number of ``$-$'' signs multiplied by edge weights.
\item Multiply the weight of each edge crossed by $\gamma_x$ with $w$ (resp.\ $w^{-1}$) if the black vertex of the edge is on the left (resp.\ right) of the path; then multiply the weight of each edge crossed by $\gamma_y$ with $z$ (resp.\ $z^{-1}$) if the black vertex of the edge is on the left (resp.\ right) of the path.
 \end{itemize}
  Let $K(z,w)$ be the weighted adjacency matrix of $\mathrm{SH}_1(\check{c})$ with respect to the modified edge weights after the multiplication above. More precisely, the rows of $K(z,w)$ are labeled by white vertices of $\mathrm{SH}_1(\check{c})$, while the columns of $K(z,w)$ are labeled by black vertices of $\mathrm{SH}_1(\check{c})$. For a black vertex $B$ and a white vertex $W$ of $G_1$; the entry $K_{BW}(z,w)=0$ if $B$ and $W$ are not adjacent; if $B$ and $W$ are joined by an edge $e_{BW}$ in $\mathrm{SH}_1(\check{c})$, then the entry $K_{BW}(z,w)$ is the modified weight of the edge $e_{BW}$. Let $P(z,w)=\det K(z,w)$ be the \textbf{characteristic polynomial}. See \cite{KOS} for more results about the characteristic polynomial and the phase transitions of the dimer model on a bipartite, periodic graph.
 
 \begin{example}Consider a fundamental domain of a square-hexagon lattice as illustrated in Figure \ref{fig:SHfd}. 
 \begin{figure}
\includegraphics{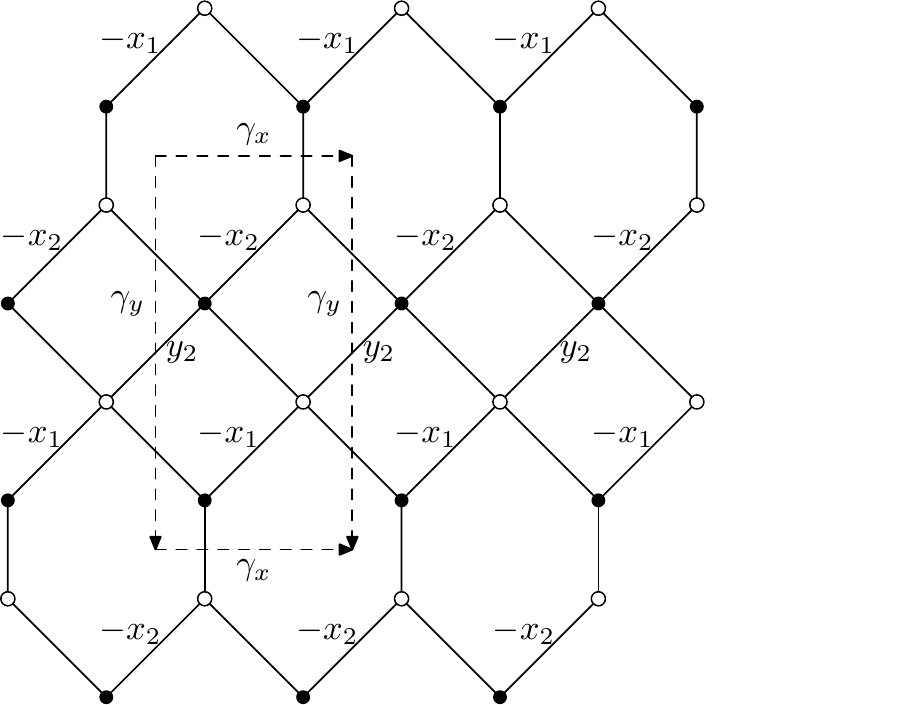}
\caption{A fundamental domain in a periodic square-hexagon lattice. The subgraph bounded by the dashed lines is a fundamental domain.}
\label{fig:SHfd}
\end{figure}
 We have
 \begin{eqnarray*}
 K(z,w)=\left(\begin{array}{cc}z-x_2&w\\1+y_2z& z-x_1\end{array}\right)
 \end{eqnarray*}
 and
 \begin{eqnarray*}
 P(z,w)=\det K(z,w)=(z-x_1)(z-x_2)-w(1+y_2z).
 \end{eqnarray*}
 \end{example}

\begin{proposition}Let $\mathbf{h}$ be the limit height function as given by (\ref{lh}).
In the liquid region, we have
\begin{eqnarray*}
\left(\frac{1}{4}\frac{\partial \mathbf{h}}{\partial x}+\frac{1}{2},\frac{1}{4}\frac{\partial \mathbf{h}}{\partial y}+\frac{n}{2}\right)=\frac{1}{\pi}(\mathrm{arg}z,-\mathrm{arg}w),
\end{eqnarray*}
where the functions $z$ and $w$ solve the differential equation
\begin{eqnarray}
\frac{z_y}{z}+\frac{w_x}{w}=0;\label{bge}
\end{eqnarray}
and the algebraic equation $P(z,w)=0$.
\end{proposition}

\begin{proof}Same arguments as the proof of Theorem 1 in \cite{KO}; see also \cite{CKP}.
\end{proof}

 When the edge weights satisfy Assumptions \ref{apew} and \ref{pw}, we can choose a fundamental domain such that $P(z,w)$ is linear in $w$. More precisely, when the period of the graph is $1\times n$, each row of the weighted adjacency matrix $K(z,w)$ has exactly two non-vanishing entries. Choose a fundamental domain consisting of $2n$ rows and $2$ columns such that the topmost row is a row of white vertices, and the rightmost column is a column of white vertices. Assume $\gamma_x$ is oriented from the left to the right and $\gamma_y$ is oriented from the top to the bottom.  Let $\mathrm{SH}_1(\check{c})$ be the toroidal graph constructed from the fundamental domain above by identifying the left and right boundary as well as the top and bottom boundary. Let $v_b$ be a white vertex of $\mathrm{SH}_1(\check{c})$. The following cases might occur
 \begin{itemize}
\item  If $v_b$ has degree 3, 
\begin{itemize}
\item when the vertex is not incident to an edge crossed by $\gamma_x$, the two non-vanishing entries of $K(z,w)$ on the row corresponding to $v_b$ are $z-x_i$ and $1$;
\item when the vertex is incident to an edge crossed by $\gamma_x$;, the two non-vanishing entries of $K(z,w)$ on the row corresponding to $v_b$ are $z-x_n$ and $w$;
\end{itemize}
\item If $v_b$ has degree 4
\begin{itemize}
\item when the vertex is not incident to an edge crossed by $\gamma_x$, the two non-vanishing entries of $K(z,w)$ on the row corresponding to $v_b$ are $z-x_i$ and $1+y_{i+1}z$;
\item when the vertex is incident to an edge crossed by $\gamma_x$;, the two non-vanishing entries of $K(z,w)$ on the row corresponding to $v_b$ are $z-x_{n-1}$ and $w(1+y_{n}z)$;
\end{itemize}
  \end{itemize}

In the toroidal graph $\mathrm{SH}_1(\check{c})$, each vertex is adjacent to exactly two vertices, with possible multiple edges joining two adjacent vertices. We may consider $\mathrm{SH}_1(\check{c})$ with vertices located on a circle, then $P(z,w)=\det K(z,w)$ counts the (signed) partition function of dimer configurations on the circle.
 
  Solving the equations $P(z,w)=0$ for $w$, we have
 \begin{eqnarray*}
 w=R(z),
 \end{eqnarray*}
 where $R(z)$ is a rational function of $z$ (quotient of two polynomials in $w$). By the explanations above, $R(z)$ can be written down explicitly as
 \begin{eqnarray*}
 R(z)=\frac{\prod_{i=1}^{n}(z-x_i)}{\prod_{j=\{1,2,\ldots,n\}\cap I_2}(1+y_jz)}.
 \end{eqnarray*}

Therefore given $P(z,w)=0$, (\ref{bge}) becomes
\begin{eqnarray}
z_x+\frac{R(z)}{zR'(z)}z_y=0.\label{ws}
\end{eqnarray}

\begin{lemma}\label{l43}Let $z$ be a solution of the following equation
\begin{eqnarray}
F_{\kappa,m}(z)=\frac{\chi}{1-\kappa},\label{fk}
\end{eqnarray}
in the upper half plane, then $z$ also satisfies the differential equation (\ref{ws}).
\end{lemma}
\begin{proof}Differentiate (\ref{fk}) with respect to $\chi$ and $\kappa$, we have
\begin{eqnarray}
\frac{\frac{\partial z}{\partial \chi}}{\frac{\partial z}{\partial \kappa}}=\frac{1}{-\frac{z}{n}\sum_{i\in I_2\cap\{1,2,\ldots,n\}}\frac{y_i}{1+y_iz}+\sum_{j=1}^n\frac{z}{n(z-x_j)}}\label{rd}
\end{eqnarray}
Given the different scalings of $(x,y)$ and $(\chi,\kappa)$, (more precisely, in the $(x,y)$-system, we assume each fundamental domain has height 1 and width 1; while in the $(\chi,\kappa)$ system, we assume each fundamental domain has width 1 and height $n$). Then we have
\begin{eqnarray*}
\frac{\partial}{\partial x}&=&\frac{\partial }{\partial \chi}\ \mathrm{and}\ \frac{\partial}{\partial y}=n\frac{\partial}{\partial \kappa }
\end{eqnarray*}
The right hand side of (\ref{rd}) divided by $n$ is exactly $\frac{R(z)}{zR'(z)}$.
\end{proof}

\begin{lemma}\label{l44}Assume all the edge weights $x_j$'s are distinct. Let $m\geq 1$ be a positive integer. Then 
\begin{enumerate}
\item For $\kappa=0$ and any $\chi\in(0,m)$, the equation
\begin{eqnarray}
F_{0,m}(z)=\chi\label{f0m}
\end{eqnarray}
has a unique root satisfying
\begin{eqnarray*}
\mathrm{Arg}(z)=\frac{\pi}{m}
\end{eqnarray*}
\item For each $z$ satisfying $\mathrm{Arg}(z)=\frac{\pi}{m}$, there exists $\chi\in(0,m)$ such that equation (\ref{f0m}) holds.
\end{enumerate}
\end{lemma}
\begin{proof}We first prove Part (1). The equation (\ref{f0m}) has the following form
\begin{eqnarray*}
\frac{1}{n}\sum_{j=0}^{n}\frac{mz^m}{z^m-x_j^m}=\chi.
\end{eqnarray*}
Let $\sqrt[m]{a}$ be the nonnegative $m$th root of a nonnegative number $a$.
Assume $z=\sqrt[m]{R}e^{\frac{\mathbf{i}\pi}{m}}$, where $R=|z^m|\geq 0$, then we have
\begin{eqnarray}
\frac{1}{n}\sum_{j=1}^{n}\frac{-mR}{-R-x_j^m}=\chi.\label{vve}
\end{eqnarray}
It suffices to show that (\ref{vve}) has a unique solution in $R\in[0,\infty)$. Explicit computations show that (\ref{vve}) is equivalent to the following equation
\begin{eqnarray*}
f(-R):=\sum_{j=1}^{n}(-mR)\prod_{i\in\{1,2,\ldots,n\},i\neq j}(-R-x_i^m)-\chi n\prod_{i=1}^{n}(-R-x_i^m)=0.
\end{eqnarray*}
Without loss of generality, assume that
\begin{eqnarray*}
0<x_1<x_2<\ldots<x_n.
\end{eqnarray*}
We have
\begin{eqnarray*}
\mathrm{sgn}[f(x_i^m)]=(-1)^{n-i}.
\end{eqnarray*}
Moreover, when $\chi>0$,
\begin{eqnarray*}
\mathrm{sgn}[f(0)]=(-1)^{n+1}.
\end{eqnarray*}
Given $\chi<m$, we have
\begin{eqnarray*}
\mathrm{sgn}[f(-\infty)]=(-1)^n,
\end{eqnarray*}
Therefore, the equation $f(z)=0$ has a solution in each one of the following intervals 
\begin{eqnarray*}
(-\infty,0),(x_1^m,x_2^m),(x_2^m,x_3^m)\ldots,(x_{n-1}^m,x_n^m).
\end{eqnarray*}
Given $\chi<m$, $f(z)$ is a degree-$n$ polynomial and has at most $n$ distinct roots in $\CC$. Therefore, we have $f(z)$ has exactly one root each one of the above intervals, and Part (1) of the lemma follows.

Now we prove Part (2) of the lemma. Let
\begin{eqnarray*}
g(t)=\frac{1}{n}\sum_{j=1}^{n}\frac{mt}{t+x_j^m}
\end{eqnarray*}
Then
\begin{eqnarray*}
g'(t)=\frac{1}{n}\sum_{j=1}^{n}\frac{m x_j^m}{(t+x_j^m)^2}>0
\end{eqnarray*}
for any $t\in \RR$. Moreover
\begin{eqnarray*}
g(0)=0, \qquad \lim_{t\rightarrow\infty}g(t)=m.
\end{eqnarray*}
Then Part (2) of the Lemma follows.
\end{proof}

\begin{lemma}\label{l45}Assume all the edge weights $x_j$'s are distinct. Let $m\geq 1$ be a positive integer. Assume
\begin{eqnarray*}
|I_2\cap\{1,2,\ldots,n\}|=r.
\end{eqnarray*}
 For each $\kappa\in[0,1]$ and each $\chi\in[0,m]$, the equation
\begin{eqnarray}
F_{\kappa,m}(z)=\frac{\chi}{1-\kappa}\label{fkm}
\end{eqnarray}
has a unique root $z_0(\chi,\kappa)$ satisfying 
\begin{enumerate}
\item
$\mathrm{Arg}z_0(\chi,0)=\frac{\pi}{m}$.
\item $z_0(0,\kappa)=0.$
\item $\lim_{\chi\rightarrow m+\left(\frac{r}{n}-1\right)\kappa}z_0\left(\chi,\kappa\right)=\infty$.
\item $z_0(\chi,1)\in[0,+\infty)$.
\item $z_0(\chi,\kappa)$ is continuous in $(\chi,\kappa)$.
\end{enumerate}
\end{lemma}
\begin{proof}Let
\begin{eqnarray*}
g_{m-1,j}(z)=\sum_{k=0}^{m-1}z^kx_j^{m-1-k}.
\end{eqnarray*}
Then when $\kappa=1$, the equation $(1-\kappa)F_{\kappa,m}=\chi$ has the following form
\begin{eqnarray*}
p(z):=\frac{z}{n}\sum_{i\in I_2\cap\{1,2,\ldots,n\}}\frac{y_i}{1+y_iz}+\frac{z}{n}\sum_{j=1}^n\frac{\frac{\partial g_{m-1,j}(z)}{\partial z}}{g_{m-1,j}(z)}-\chi=0.
\end{eqnarray*}
Note that $p(z)$ is well-defined on $(0,+\infty)$, since $g_{m-1,j}(z)>0$ whenever $z>0$.
When $\chi\in (0,m-1+\frac{r}{n})$, we have
\begin{eqnarray*}
p(0)<0;\ \mathrm{and}\ p(+\infty)>0.
\end{eqnarray*}
Moreover, we have 
\begin{lemma}
\begin{eqnarray*}
p'(z)>0
\end{eqnarray*}
when $z\in (0,+\infty)$ and $x_i,y_j>0$.
\end{lemma}
\begin{proof}Note that
\begin{eqnarray*}
p'(z)=\frac{1}{n}\sum_{i\in I_2\cap\{1,2,\ldots,n\}}\frac{y_i}{(1+y_iz)^2}+\frac{1}{n}\sum_{j=1}^{n}\frac{(zg_{m-1,j}')'g_{m-1,j}-z(g_{m-1,j}')^2}{g_{m-1,j}^2}
\end{eqnarray*}
It suffices to show that for $z>0$, $x_i>0$, we have
\begin{eqnarray}
T_m(z):=(zg_{m-1,j}')'g_{m-1,j}-z(g_{m-1,j}')^2\geq 0.\label{iie}
\end{eqnarray}
We prove (\ref{iie}) by induction on $m$. When $m=1$, we have $g_{m-1,j}=1$ and $T_m(z)=0$. Assume (\ref{iie}) holds when $m=l-1$, $l\geq 2$. When $m=l$, we have
\begin{eqnarray*}
g_{l-1,j}=z g_{l-2,j}+x_j^{m-1};
\end{eqnarray*}
and
\begin{eqnarray*}
T_{l}(z)&=&[zg_{l-2,j}+z^2g'_{l-2,j}]'(zg_{l-2,j}+x_j^{l-1})-z(g_{l-2,j}+zg_{l-2,j}')^2\\
&=&z^2 T_{l-1}(z)+x_j^{l-1}g_{l-2,j}+z x_j^{l-1}(3g_{l-2,j}'+g_{l-2,j}'')>0
\end{eqnarray*}
Since $T_{l-1}(z)>0$ by induction hypothesis, $g_{l-2,j}>0$, $g_{l-2,j}'>0$ and $g_{l-2,j}''>0$. Then the lemma follows.
\end{proof}

Hence $p(z)=0$ has exactly one root in $(0,\infty)$ when $\chi\in (0,m-1+\frac{r}{n})$. The root converges to 0 when $\chi$ goes to 0, and the root approaches $+\infty$ when $\chi$ goes to $m-1+\frac{r}{n}$. 

The fact that there is a unique root of (\ref{fkm}) satisfying Condition (1) follows from Lemma \ref{l44}. The root when $\kappa=0$ satisfying condition (1) and the root when $\kappa=1$ satisfying condition (4) can be considered as boundary conditions for the Burgers equation. More precisely, the slope of height is $\frac{1}{m}$ on the bottom boundary $\kappa=0$ of the rescaled square-hexagon lattice $\mathcal{R}$, while the slope of height is $0$ on the top boundary $\kappa=1$. Since the surface tension function is strictly convex in the liquid region,  the solution of the Burgers equations satisfying the given boundary conditions is unique; see \cite{CKP,KO}.  By Lemma \ref{l43}, a root satisfying (1)-(5) is also a solution of the Burgers equation satisfying given boundary conditions, then the lemma follows.

\end{proof}

\begin{lemma}\label{l46}Let $\mathcal{L}$ be liquid region of the limit shape, defined by\begin{enumerate}
\item the region is a subset of 
\begin{eqnarray*}
\mathcal{R}:=\left\{(\chi,\kappa): 0\leq \kappa <1;0<\chi<m+\left(\frac{r}{n}-1\right)\kappa\right\}
\end{eqnarray*}
\item the solution $z_0(\chi,\kappa)$ of $F_{\kappa,m}(z)=\frac{\chi}{1-\kappa}$ given as in Lemma \ref{l45} remains non-real in the region;
\item the region includes the bottom boundary $\kappa=0$, $0<\chi<m$ when $m\geq 2$. 
\end{enumerate} 
 Let $\mathcal{L}_o$ be the interior of $\mathcal{L}$ Let $T_{\mathcal{L}}$ be a mapping on $\mathcal{L}$ which maps each point $(\chi,\kappa)\in \mathcal{L}$ to $z_0(\chi,\kappa)$. Then restricted on $\mathcal{L}_o$, $T_{\mathcal{L}}$ is a homeomorphism from the interior of $\mathcal{L}_o$ to $T_{\mathcal{L}}(\mathcal{L}_o)$, where
\begin{eqnarray*}
T_{\mathcal{L}}(\mathcal{L}_o)=\left\{z\in \mathbb{C}:0< \mathrm{Arg}z<\frac{\pi}{m}\right\}
\end{eqnarray*}
and $\mathrm{Arg}z$ is the principal argument of $z$.
\end{lemma}

\begin{proof}First we show that $T_{\mathcal{L}}$ is one-to-one from $\mathcal{L}$ to $T_{\mathcal{L}}(\mathcal{L})$. Let 
\begin{eqnarray*}
U(z)&=&\frac{z}{n}\sum_{i\in\{1,2,\ldots,n\}\cap I_2}\frac{y_i}{1+y_iz};\\
V(z)&=&\frac{z}{n}\sum_{j=1}^n\frac{1}{z-x_j}\\
W(z)&=&\frac{z}{n}\sum_{j=1}^n\left(\frac{mz^{m-1}}{z^m-x_j^m}-\frac{1}{z-x_j}\right)
\end{eqnarray*}
Then if $z$ is a solution of $F_{\kappa,m}(z)=\frac{\chi}{1-\kappa}$, we have $F_{\kappa,m}(\ol{z})=\frac{\chi}{1-\kappa}$. Hence we can solve for $\chi$ and $\kappa$ in terms of $z$ as follows
\begin{eqnarray}
&&\chi_{\mathcal{L}}(z)\label{clm}\\
&=&\frac{W(\ol{z})U(z)+V(\ol{z})U(z)-U(\ol{z})V(z)-W(\ol{z})V(z)-W(z)U(\ol{z})+W(z)V(\ol{z})}{U(z)-U(\ol{z})-V(z)+V(\ol{z})}\notag\\
&&\kappa_{\mathcal{L}}(z)=\frac{W(\ol{z})-W(z)+V(\ol{z})-V(z)}{U(z)-U(\ol{z})-V(z)+V(\ol{z})}\label{klm}
\end{eqnarray}
Therefore $T_{\mathcal{L}}$ is one-to-one from $\mathcal{L}_o$ to $T_{\mathcal{L}}(\mathcal{L}_o)$. From the expression of $T_{\mathcal{L}}$, $\chi_{\mathcal{L}}$ and $\kappa_{\mathcal{L}}$ it is straightforward to see that both $T_{\mathcal{L}}$ and $\chi_{\mathcal{L}}$, $\kappa_{\mathcal{L}}$ are continuous. 

Then we claim that 
\begin{eqnarray}
\mathcal{L}_o=\mathcal{L}\setminus \{\kappa=0\};\label{lo}
\end{eqnarray}
and $\mathcal{L}_o$ is connected. To see why that is true,
let $(\chi_1,\kappa_1)\in\mathcal{L}\setminus\{\kappa=0\}$ and $z_1= T_{\mathcal{L}}(\chi_1,\kappa_1)$. From the definition of $\mathcal{L}$, we have $z_1\notin \RR$. Note that $z_1$ be the root of $F_{\kappa_1}(z)=\frac{\chi_1}{1-\kappa_1}$ as given by Lemma \ref{l45}. Given $0<\kappa_1<1$, we obtain
\begin{eqnarray*}
z_1\notin\{x_je^{\frac{2t\pi}{m}}:1\leq j\leq n,1\leq t\leq m\}.
\end{eqnarray*}

 We shall show that $(\chi_2,\kappa_2)\in \mathcal{L}\setminus \{\kappa=0\}$ whenever $|\chi_1-\chi_2|$ and $|\kappa_1-\kappa_2|$ are sufficiently small. Fix $\epsilon>0$ such that 
 \begin{itemize}
\item \begin{eqnarray} 
 B(z_1,\epsilon)\cap [\RR\cup\{x_je^{\frac{2t\pi}{m}}:1\leq j\leq n,1\leq t\leq m\}]=\emptyset \label{zeo}
\end{eqnarray}
\item 
\begin{eqnarray*}
\inf_{z\in\partial B(z_1,\epsilon)}\left|F_{\kappa_1}(z)-\frac{\chi_1}{1-\kappa_1}\right|>0;
\end{eqnarray*}
\item $F_{\kappa_1}(z)=\frac{\chi_1}{1-\kappa_1}$ has a unique zero in $B(z_1,\epsilon)$.
\item $B(z_1,\epsilon)\subset \mathcal{R}$.
\end{itemize}
By (\ref{zeo}) 
\begin{eqnarray*}
\left|F_{\kappa_1}(z)-\frac{\chi_1}{1-\kappa_1}-F_{\kappa_2}(z)+\frac{\chi_2}{1-\kappa_2}\right|<\epsilon,
\end{eqnarray*}
for any $\epsilon>0$, whenever $|\chi_1-\chi_2|$ and $|\kappa_1-\kappa_2|$ are sufficiently small, and $z\in B(z_1,\epsilon)$.

Therefore when $|\chi_1-\chi_2|$ and $|\kappa_1-\kappa_2|$ are sufficiently small, we have
\begin{eqnarray*}
\left|F_{\kappa_1}(z)-\frac{\chi_1}{1-\kappa_1}\right|>\left|F_{\kappa_1}(z)-\frac{\chi_1}{1-\kappa_1}-F_{\kappa_2}(z)+\frac{\chi_2}{1-\kappa_2}\right|.
\end{eqnarray*}
for any $z\in \partial B(z_1,\epsilon)$.
By Rouch\'e's theorem, $F_{\kappa_2}(z)-\frac{\chi_2}{1-\kappa_2}$ has a root in $B(z_1,\epsilon)$, which is as described in Lemma \ref{l45}. Hence $(\chi_2,\kappa_2)\in\mathcal{L}\setminus \{\kappa=0\}$, and we obtain (\ref{lo}). The statement that $\mathcal{L}_o$ is connected follows from the fact that $\mathcal{L}$ is connected and that for any point $(\chi,0)$ with $0<\chi<\mathcal{L}$ there is a neighborhood $B_{\delta}(\chi,0)$ such that
$B_{\delta}(\chi,0)\cap R\in \mathcal{L}$.

Now we claim that
\begin{eqnarray}
T_{\mathcal{L}}(\mathcal{L}_o)\subseteq \left\{z\in \mathbb{C}:0< \mathrm{Arg}z<\frac{\pi}{m}\right\}\label{tls}
\end{eqnarray}
To see why that is true, assume there exists a point $c\in T_{\mathcal{L}}(\mathcal{L}_o)\setminus \left\{z\in \mathbb{C}:0< \mathrm{Arg}z<\frac{\pi}{m}\right\}$.  Since $T_{\mathcal{L}}(\mathcal{L}_o)\cap \left\{z\in \mathbb{C}:0< \mathrm{Arg}z<\frac{\pi}{m}\right\}\neq \emptyset$, we can find $c_1\in T_{\mathcal{L}}(\mathcal{L}_o)\cap \left\{z\in \mathbb{C}:0< \mathrm{Arg}z<\frac{\pi}{m}\right\}$. Let $d,d_1\in \mathcal{L}_o$ such that
\begin{eqnarray*}
T_{\mathcal{L}}(d)=c;\qquad T_{\mathcal{L}}(d_1)=c_1
\end{eqnarray*}
 Since $\mathcal{R}$ is connected, we can find a path $p_{dd_1}$ in $\mathcal{R}$ joining $d$ and $d_1$ such that $p_{dd_1}\cap \{\kappa=0\}=\emptyset$. By continuity $T_{\mathcal{L}}(p_{dd_1})$ is a continuous curve in $\CC$ joining the point $c$ satisfying $\mathrm{Arg}c<\frac{\pi}{m}$ and the point $c_1$ satisfying $\mathrm{Arg}c_1>\frac{\pi}{m}$. Then there exists $c_2\in T_{\mathcal{L}}(d_2)$ such that $d_1\in p_{dd_1}$ and $\mathrm{Arg} c_2=\frac{\pi}{m}$. By (\ref{klm}), $d_2$ on the line $\kappa=0$, but this is a contraction to the fact that $p_{dd_1}\cap \{\kappa=0\}=\emptyset$. The contradiction implies (\ref{tls}).

We finally show that
\begin{eqnarray*}
\left\{z\in \mathbb{C}:0< \mathrm{Arg}z<\frac{\pi}{m}\right\}\subseteq T_{\mathcal{L}}(\mathcal{L}_o).
\end{eqnarray*}
Assume that there exists $t\in \partial T_{\mathcal{L}}(\mathcal{L}_o)$, such that $t\in \left\{z\in \mathbb{C}:0< \mathrm{Arg}z<\frac{\pi}{m}\right\}\setminus T_{\mathcal{L}}(\mathcal{L}_o)$. Then there exists a sequence $\{t_n\}_{n\in \NN}\subset T_{\mathcal{L}}(\mathcal{L}_o)$ such that
\begin{eqnarray*}
\lim_{n\rightarrow\infty}t_n=t.
\end{eqnarray*}
By continuity of (\ref{clm}), (\ref{klm}) we have
\begin{eqnarray*}
\lim_{n\rightarrow\infty}\chi(t_n)&=&\chi(t)\\
\lim_{n\rightarrow\infty}\kappa(t_n)&=&\kappa(t).
\end{eqnarray*}
Hence $(\chi(t)$, $\kappa(t))\in\mathcal{L}_o\cup \partial \mathcal{L}_o$ is such that $F_{\kappa(t),m}(z)=\frac{\chi(t)}{1-\kappa(t)}$ has a root $t$ in $\left\{z\in \mathbb{C}:0< \mathrm{Arg}z<\frac{\pi}{m}\right\}$ as described in Lemma \ref{l45}. 
Note that $\kappa(t)\neq 1$, because if $\kappa(t)=1$, then $\mathrm{Arg} t=0$. Also $\chi(t)\neq 0$, because if $\chi(t)=0$, then $t=0$. Moreover, $\chi(t)\neq m+\left(\frac{r}{n}-1\right)\kappa(t)$, because otherwise $t=\infty$.
Then $\chi(t),\kappa(t)\in \mathcal{L}$. Since $\mathrm{Arg} t\neq \frac{\pi}{m}$, $\kappa(t)\neq 0$, we have $(\chi(t),\kappa(t))\in \mathcal{L}_o$ and $t\in T_{\mathcal{L}}(\mathcal{L}_o)$. Then the proof is complete.
\end{proof}

 \subsection{Gaussian Free Field} 
 
 Let $C_0^{\infty}$ be the space of smooth real-valued functions with compact support in the upper half plane $\HH$. The \textbf{Gaussian free field} (GFF) $\Xi$ on $\HH$ with the zero boundary condition is a collection of Gaussian random variables $\{\xi_{f}\}_{f\in C_0^{\infty}}$ indexed by functions in $C_0^{\infty}$, such that the covariance of two Gaussian random variables $\xi_{f_1}$, $\xi_{f_2}$ is given by
 \begin{eqnarray*}
 \mathrm{Cov}(\xi_{f_1},\xi_{f_2})=\int_{\HH}\int_{\HH}f_1(z)f_2(w)G_{\HH}(z,w)dzd\ol{z}dwd\ol{w},
 \end{eqnarray*}
 where
 \begin{eqnarray*}
 G_{\HH}(z,w):=-\frac{1}{2\pi}\ln\left|\frac{z-w}{z-\ol{w}}\right|,\qquad z,w\in \HH
 \end{eqnarray*}
 is the Green's function of the Dirichlet Laplacian operator on $\HH$. The Gaussian free field $\Xi$ can also be considered as a random distribution on $C_0^{\infty}$ of $\HH$, such that for any $f\in C_0^{\infty}$, we have
 \begin{eqnarray*}
 \Xi(f)=\int_{\HH}f(z)\Xi(z)dz:=\xi_f.
 \end{eqnarray*}
 See \cite{SS07} for more about GFF.
 
 Consider a contracting square-hexagon lattice $\mathcal{R}(\Omega,\check{c})$. Let $\omega$ be a signature corresponding to the boundary row.
 
Let
 \begin{eqnarray*}
 \mathcal{S}^N=(\mu^{(N)},\nu^{(N)},\ldots,\mu^{(1)},\nu^{(1)}).
 \end{eqnarray*}
be the sequence of (random) partitions corresponding to the (random) dimer configuration on the contraction square-hexagon lattice $\mathcal{R}(\Omega,\check{c})$, as in Proposition \ref{p16}.
 
 Define a function $\Delta^N$ on $\RR_{\geq 0}\times \RR_{\geq 0}\times \mathcal{S}\rightarrow \NN$ as follows
 \begin{eqnarray*}
&& \Delta^{N}: (x,y,(\mu^{(N)},\nu^{(N)},\ldots,\mu^{(1)},\nu^{(1)}))\longrightarrow\\
 &&\sqrt{\pi}|\{1\leq s\leq N-\lfloor y\rfloor\}:\mu_s^{N-\lfloor y\rfloor}+(N-\lfloor y\rfloor)-s\geq x\}|
 \end{eqnarray*}

Let $\Delta_{\mathcal{M}}^N(x,y)$ be the pushforward of the measure $P_{\omega}^N$ on $\mathcal{S}^N$ with respect to $\Delta^N$. For $z\in \HH$, define
\begin{eqnarray*}
\mathbf{\Delta}_{\mathcal{M}}^N(z):=\Delta_{M}^N(N\chi_{\mathcal{L}}(z),N\kappa_{\mathcal{L}}(z)),
\end{eqnarray*}
where $\chi_{\mathcal{L}}(z)$, $\kappa_{\mathcal{L}}(z)$ are defined by (\ref{clm}), (\ref{klm}), respectively.

Here is the main theorem we shall prove in the section.

\begin{theorem}\label{t710}Let $\mathbf{\Delta}_{\mathcal{M}}^{N}(z)$ be a random function corresponding to the random perfect matching of the contracting square-hexagon lattice, as explained above. Then
\begin{eqnarray*}
\mathbf{\Delta}_{\mathcal{M}}^{N}(z)-\EE\mathbf{\Delta}_{\mathcal{M}}^{N}(z)\rightarrow \Xi(z),\ \mathrm{as}\ N\rightarrow\infty.
\end{eqnarray*}
Here $\Xi(z)$ is the Gaussian free field in $\mathbb{S}$ with zero boundary conditions as defined above. The convergence is in the sense that for $0<\kappa\leq 1$, $j\in \NN$,
\begin{eqnarray}
M_j^{\kappa}\rightarrow \mathbf{M}_j^{\kappa},\ \mathrm{as}\ N\rightarrow\infty,\label{mjk}
\end{eqnarray}
where
\begin{eqnarray*}
M_j^{\kappa}=\int_{-\infty}^{+\infty}\chi^j(\mathbf{\Delta}_{\mathcal{M}}^{N}(N\chi,N\kappa)-\mathbb{E}\mathbf{\Delta}_{\mathcal{M}}^{N}(N\chi,N\kappa))d\chi,
\end{eqnarray*}
and
\begin{eqnarray*}
\mathbf{M}_j^{\kappa}=\int_{z\in\HH;\kappa_{\mathcal{L}}(z)=\kappa}\chi_{\mathcal{L}}(z)^j\Xi(z)d\chi_{\mathcal{L}}(z).
\end{eqnarray*}
\end{theorem}
\begin{proof}By Theorem \ref{gff1}, we have
\begin{eqnarray}
&&\lim_{N\rightarrow\infty}\frac{\mathrm{cov}\left(p_{l_1}^{\left(1-\kappa_1) N\right)},p_{l_2}^{\left((1-\kappa_2)N\right)}\right)}{N^{l_1+l_2}}\label{cor}\\
&=&\frac{(1-\kappa_1)^{l_1}(1-\kappa_2)^{l_2}}{(2\pi\mathbf{i})^2}\sum_{i=1}^{n}\sum_{j=1}^{n}\oint_{|z-x_i|=\epsilon}\oint_{|w-x_j|=\epsilon}\left[F_{\kappa_1,m}(z)\right]^{l_1}\left[F_{\kappa_2,m}(w)\right]^{l_2}Q(z,w)dzdw,\notag
\end{eqnarray}
The poles of $F_{\kappa,m}(z)$ are of 3 types
\begin{enumerate}
\item $x_1,\ldots,x_n$ lying on the positive real axis;
\item $-\frac{1}{y_j}$ for $j\in I_2\cap\{1,2,\ldots,n\}$ lying on the negative real axis;
\item roots of $z^m=x_j^m$ except $x_j$ for $j=1,\ldots,n$, lying on the circle centered at $0$ with radius $x_j$.
\end{enumerate}
We may change the sum of contour integrals in the RHS of (\ref{cor}) into an integral over a contour enclosing all the poles of $F_{\kappa,m}(z)$ of type (1), yet enclosing no poles of types (2) and (3), with respect to both $z$ and $w$.

For $\kappa\in(0,1)$, let 
\begin{eqnarray*}
\mathcal{Z}_1(\kappa)=\left\{z: F_{\kappa,m}(z)=\frac{\chi}{1-\kappa}; \chi\in \left[0, m+\kappa\left(\frac{r}{n}-1\right)\right]; z\ \mathrm{is\ a\ root\ as\ given\ by\ Lemma\ \ref{l45}}\right\}.
\end{eqnarray*}
Let 
\begin{eqnarray*}
\mathcal{Z}_2(\kappa)=\left\{z:\ol{z}\in \mathcal{Z}_1(\kappa)\right\}.
\end{eqnarray*}
and
\begin{eqnarray*}
\mathcal{Z}(\kappa)=\mathcal{Z}_1(\kappa)\cup \mathcal{Z}_2(\kappa).
\end{eqnarray*}

We claim that for $\kappa\in(0,1)$, $\mathcal{Z}(\kappa)$ is a contour in the complex plane $\CC$ enclosing all the poles of $F_{\kappa,m}(z)$ of type (1), yet enclosing no poles of type (2) and (3). By Lemma \ref{l46}, we have
\begin{eqnarray*}
\mathcal{Z}_1(\kappa)\in \{0,+\infty\}\cup \left\{z:0<\mathrm{Arg}z<\frac{\pi}{m}\right\}.
\end{eqnarray*}
Since $\mathcal{Z}_2(\kappa)$ is the complex conjugate of $\mathcal{Z}_1(\kappa)$, then the claim follows.

For $\kappa_1,\kappa_2\in(0,1)$, (\ref{cor}) becomes
\begin{eqnarray*}
&&\lim_{N\rightarrow\infty}\frac{\mathrm{cov}\left(p_{l_1}^{\left(1-\kappa_1) N\right)},p_{l_2}^{\left((1-\kappa_2)N\right)}\right)}{N^{l_1+l_2}}\\
&=&\frac{(1-\kappa_1)^{l_1}(1-\kappa_2)^{l_2}}{(2\pi\mathbf{i})^2}\oint_{z\in \mathcal{Z}(\kappa_1)}\oint_{w\in\mathcal{Z}(\kappa_2)}\left[F_{\kappa_1,m}(z)\right]^{l_1}\left[F_{\kappa_2,m}(w)\right]^{l_2}Q(z,w)dzdw,\\
&=&\frac{1}{(2\pi\mathbf{i})^2}\oint_{z\in \mathcal{Z}(\kappa_1)}\oint_{w\in\mathcal{Z}(\kappa_2)}\left[\chi_{\mathcal{L}}(z)\right]^{l_1}\left[\chi_{\mathcal{L}}(z)\right]^{l_2}Q(z,w)dzdw
\end{eqnarray*}
where
\begin{eqnarray*}
Q(z,w)=\frac{m^2z^{m-1}w^{m-1}}{(z^m-w^m)^2}
\end{eqnarray*}
Integrate (\ref{mjk}) by parts we obtain
\begin{eqnarray*}
M_j^{\kappa}=\frac{N^{-(j+1)}\sqrt{\pi}}{j+1}(p_{j+1}^{\kappa}-\mathbf{E}p_{j+1}^{\kappa}).
\end{eqnarray*}
Hence the random variables $\{M_j^{\kappa}\}_{\kappa\in(0,1),j\in \NN}$ converge to the Gaussian distribution with mean 0 and limit covariance
\begin{eqnarray}
&&\lim_{N\rightarrow\infty}\mathrm{cov}(M_{j_1}^{\kappa_1},M_{j_2}^{\kappa_2})\label{ra}\\
&=&\frac{-1}{4\pi(j_1+1)(j_2+1)}\oint_{z\in \mathcal{Z}(\kappa_1)}\oint_{2\in \mathcal{Z}(\kappa_2)}\chi_{\mathcal{L}}(z)^{j_1+1}\chi_{\mathcal{L}}(w)^{j_2+1}\frac{m^2z^{m-1}w^{m-1}}{(z^m-w^m)^2}dzdw.\notag
\end{eqnarray}
Let 
\begin{eqnarray*}
\mathbb{S}=\left\{z:0<\mathrm{Arg} z<\frac{\pi}{m}\right\}.
\end{eqnarray*}
By definition, the random variables $\{\mathcal{M}_j^{\kappa}\}_{\kappa\in(0,1),j\in \NN}$ are Gaussian with mean 0 and covaraince
\begin{eqnarray}
&&\mathrm{cov}(\mathcal{M}_{j_1}^{\kappa_1},\mathcal{M}_{j_2}^{\kappa_2})\label{rb}\\
&=&\oint_{z\in\mathbb{S}:\kappa_{\mathcal{L}}(z)=\kappa_1}\oint_{w\in \mathbb{S}:\kappa_{\mathcal{L}}(w)=\kappa_2}\chi_{\mathcal{L}}(z)^{j_1}\chi_{\mathcal{L}}(w)^{j_2}\frac{d\chi_{\mathcal{L}}(z)}{dz}\frac{d\chi_{\mathcal{L}}(w)}{dw}\mathcal{G}_{\mathbb{S}}(z,w)dzdw,\notag
\end{eqnarray}
where $\mathcal{G}_{\mathbb{S}}(z,w)$ is the Green's function on $\mathbb{S}$ given by
\begin{eqnarray*}
-\frac{m^2z^{m-1}w^{m-1}}{2\pi}\mathrm{ln}\left|\frac{z^m-w^m}{z^m-\ol{w}^m}\right|
\end{eqnarray*}
Then integration by parts shows that the right hand sides  of (\ref{ra}) and (\ref{rb}) are equal.
\end{proof}

\section{Gaussian Free Field in Piecewise Boundary Condition}\label{gffp}

For $1\leq i\leq n$, let
\begin{eqnarray*}
F_{i,\kappa}(z)=\frac{1}{n}\frac{z}{z-1}+\frac{1}{1-\kappa}zA_i(z);
\end{eqnarray*}
where $A_i(z)$ is defined by (\ref{aj}).

\begin{lemma}For any $\chi>0$, $\kappa\in (0,1)$ and $1\leq i\leq n$ the equation 
\begin{eqnarray}
F_{i,\kappa}(z)=\frac{\chi}{1-\kappa}.\label{fik}
\end{eqnarray}
has at most one pair of complex conjugate roots.
\end{lemma}
\begin{proof}See Proposition 7.2 of \cite{Li18}.
\end{proof}

Let $J_i$ be defined as in (\ref{ji}). Under Assumption \ref{ap32}, we may assume that 
\begin{eqnarray*}
J_i=\left\{\begin{array}{cc}\{d_i,d_i+1,\ldots,d_{i+1}-1\}& \mathrm{if}\ 1\leq i\leq n-1\\ \{d_n, d_n+1,\ldots,s\}& \mathrm{if}\ i=n \end{array}\right.
\end{eqnarray*}
where for $1\leq i\leq n$, 
\begin{eqnarray*}
1=d_1<d_2<\ldots <d_n\leq s
\end{eqnarray*}

\begin{lemma}Let $1\leq i\leq n$. Let $\mathcal{S}_i$ consisting of all the $(\chi,\kappa)$ in $\mathcal{R}$ (the rescaled square-hexagon lattice $\frac{1}{N}\mathcal{R}(\Omega,\check{c})$ in the limit as $N\rightarrow\infty$) such that equation (\ref{fik}) has exactly one pair of complex conjugate roots. Let $\HH$ be the upper half plane defined by
\begin{eqnarray*}
\HH=\{z: \mathrm{Im}z>0\}.
\end{eqnarray*}
 The mapping
\begin{eqnarray*}
T_{\mathcal{S}_i}:\mathcal{S}_i\rightarrow \mathbb{H}
\end{eqnarray*}
maps $(\chi,\kappa)\in \mathcal{S}_i$ to $t:=\mathrm{St}_{\mathbf{m}_i}^{(-1)}(\log z)$, where $z$ is the unique root of (\ref{fik}) in $\HH$. Let
\begin{eqnarray*}
p(t)&=&\sum_{l\in I_2\cap\{1,2,\ldots,n\}}\frac{y_lx_1 \exp[\mathrm{St}_{\mathbf{m}_i}(t)]}{1+y_lx_1 \exp[\mathrm{St}_{\mathbf{m}_i}(t)]}\\
q(t)&=&\frac{\exp(\mathrm{St}_{\mathbf{m}_i}(t))}{\exp(\mathrm{St}_{\mathbf{m}_i}(t))-1}
\end{eqnarray*}
Then $T_{\mathcal{S}_1}$ is a homeomorphism with inverse $t\rightarrow (\chi_{\mathcal{S}_1}(t),\kappa_{\mathcal{S}_1}(t))$ for all $t\in \HH$, given by 
\begin{eqnarray}
\chi_{\mathcal{S}_1}(t)&=&\frac{\ol{t}(p(t)-q(t))-t(p(\ol{t})-q(\ol{t}))-(n-1)(\ol{t}-t)}{n[p(t)-p(\ol{t})-q(t)+q(\ol{t})]}\label{cst1}\\
\kappa_{\mathcal{S}_1}(t)&=&\frac{\ol{t}-t}{p(t)-p(\ol{t})-q(t)+q(\ol{t})}\label{kst1}
\end{eqnarray}
and for $2\leq i\leq n$
\begin{eqnarray}
\chi_{\mathcal{S}_i}(t)&=&\frac{\ol{t}q(t)-tq(\ol{t})+(\ol{t}-t)(n-i)}{n[q(t)-q(\ol{t})]}\label{csti}\\
\kappa_{\mathcal{S}_i}(t)&=&\frac{\ol{t}-t}{-q(t)+q(\ol{t})}\label{ksti}
\end{eqnarray}
\end{lemma}

\begin{proof}The proof is an adaptation of Proposition 6.2 of \cite{bk}; see also Theorem 2.1 of \cite{DM15}.
 
 It suffices to show all the following statements for each $1\leq i\leq n$:
 \begin{enumerate}
 \item $\mathcal{S}_i$ is nonempty.
 \item $\mathcal{S}_i$ is open.
 \item $T_{\mathcal{S}_i}: \mathcal{S}_i\rightarrow \HH$ is continuous.
 \item $T_{\mathcal{S}_i}: \mathcal{S}_i\rightarrow \HH$ is injective.
 \item $T_{\mathcal{S}_i}: \mathcal{S}_i\rightarrow T_{\mathcal{S}_i}(\mathcal{S}_i)$ has continuous inverse for all $t\in T_{\mathcal{S}_i}(\mathcal{S}_i)$.
 \item $T_{\mathcal{S}_i}(\mathcal{S}_i)=\HH$.
 \end{enumerate}
 
 We first prove (1). Explicit computations show that $(\chi_{\mathcal{S}_i}(t),\kappa_{\mathcal{S}_i}(t))$ satisfies (\ref{cst1}), (\ref{kst1}) when $i=1$ and (\ref{csti}), (\ref{ksti}) when $2\leq i\leq n$. Since $\mathbf{m}_{i}$ is a measure on $\RR$ with compact support, assume that $\mathrm{Support}(\mathbf{m}_i)\subset[a,b]$ where $a,b\in\RR$. The Stieltjes transform satisfies
 \begin{eqnarray*}
\mathrm{St}_{\mathbf{m}_i}(t)=\frac{1}{t}+\frac{\alpha}{t^2}+\frac{\beta}{t^3}+O(|t|^{-4}),
\end{eqnarray*}
where
\begin{eqnarray*}
\alpha&=&\int_a^b x\mathbf{m}_{i}(dx)\\
\beta&=&\int_a^b x^2\mathbf{m}_{i}(dx).
\end{eqnarray*}
Let
\begin{eqnarray*}
c_i=\frac{1}{y_ix_1}
\end{eqnarray*}
After computations we have
\begin{eqnarray*}
\chi_{\mathcal{S}_1}&=&\frac{\alpha-1}{n}-\frac{n-1}{n}+O(|t|^{-1}).\\
\kappa_{\mathcal{S}_1}&=&1+\left(\alpha^2-\beta-\sum_{i\in I_2\cap\{1,2,\ldots,n\}}\frac{c_i}{(1+c_i)^2}\right)\frac{1}{|t|^2}+O(|t|^{-3}).
\end{eqnarray*}
and for $2\leq i\leq n$, 
\begin{eqnarray*}
\chi_{\mathcal{S}_i}&=&\frac{\alpha-1}{n}-\frac{n-i}{n}+O(|t|^{-1}).\\
\kappa_{\mathcal{S}_i}&=&1+\left(\alpha^2-\beta\right)\frac{1}{|t|^2}+O(|t|^{-3}).
\end{eqnarray*}

Let $\lambda$ be the Lebesgue measure on $\RR$. Recall that $\mathbf{m}_{i}$ is the limit counting measure for $\phi^{(i,\si_0)}$ as $N\rightarrow\infty$. By Assumption \ref{ap428}, we have
\begin{eqnarray*}
\alpha\geq n-i+1
\end{eqnarray*}

Similarly,
\begin{eqnarray*}
\beta-\alpha^2\geq \frac{1}{2}\int_{0}^1\int_{0}^1(x-y)^2dxdy=\frac{1}{12}.
\end{eqnarray*}
As a result, $(\chi,\kappa)\in\left(0,\infty\right)\times(0,1)$ whenever $|t|$ is sufficiently large. Then (1) follows.

The facts (2) and (3) follow from Rouch\'e's theorem by the same arguments as in the proof of Proposition 6.2 in \cite{bk}. The facts (4)-(6) can also be obtained by the same arguments as in the proof of Proposition 6.2 in \cite{bk}.
\end{proof}

\begin{theorem}Let $\mathbf{\Delta}_{\mathcal{M}}^{N}(z)$ be a random function corresponding to the random perfect matching of the contracting square-hexagon lattice, as in Theorem \ref{t710}, but with piecewise boundary conditions satisfying Assumptions \ref{ap423}, \ref{ap32} and \ref{ap428}. Then
\begin{eqnarray*}
\mathbf{\Delta}_{\mathcal{M}}^{N}(z)-\EE\mathbf{\Delta}_{\mathcal{M}}^{N}(z)\rightarrow \Xi_1(z)+\Xi_2(z)+\ldots+\Xi_n(z),\ \mathrm{as}\ N\rightarrow\infty.
\end{eqnarray*}
Here for $1\leq i\leq n$ $\Xi_i(z)$'s are $n$ independent Gaussian free fields in $\mathbb{H}$ with zero boundary conditions. The convergence is in the sense that for $0<\kappa\leq 1$, $j\in \NN$,
\begin{eqnarray}
M_j^{\kappa}\rightarrow \mathbf{M}_j^{\kappa},\ \mathrm{as}\ N\rightarrow\infty,\label{mjk}
\end{eqnarray}
where
\begin{eqnarray*}
M_j^{\kappa}=\int_{-\infty}^{+\infty}\chi^j(\mathbf{\Delta}_{\mathcal{M}}^{N}(N\chi,N\kappa)-\mathbb{E}\mathbf{\Delta}_{\mathcal{M}}^{N}(N\chi,N\kappa))d\chi,
\end{eqnarray*}
and
\begin{eqnarray*}
\mathbf{M}_j^{\kappa}=\int_{z\in\HH;\kappa_{\mathcal{L}}(z)=\kappa}\chi_{\mathcal{L}}(z)^j\sum_{i=1}^n\Xi_i(z)d\chi_{\mathcal{L}}(z).
\end{eqnarray*}
\end{theorem}

\begin{proof}Similar arguments as the proof of Theorem 6.3 in \cite{bk}.
\end{proof}

\bigskip
\bigskip
\noindent\textbf{Acknowledgements.} ZL's research is supported by National Science Foundation grant DMS 1608896. ZL thanks Vadim Gorin for helpful discussions.

\bibliography{fpm19}
\bibliographystyle{amsplain}
\end{document}